\documentclass[11pt,reqno,oneside]{amsart}

\usepackage[asymmetric,top=3.5cm,bottom=4.3cm,left=3.1cm,right=3.1cm]{geometry}
\usepackage{amsmath,amsfonts,amsthm,mathrsfs,amssymb,cite}
\usepackage[usenames]{color}
\usepackage{mathtools} %for \mathclap, \mathrlap, \smashoperator etc
\usepackage{graphics}
\usepackage{framed}

\usepackage[colorlinks,citecolor=blue,linkcolor=blue]{hyperref} %cyan

\usepackage{enumerate}

\usepackage{hyperref}
\usepackage{xcolor}
\hypersetup{
  colorlinks,
  linkcolor={blue!80!black},
  urlcolor={blue!80!black},
  citecolor={blue!80!black}
}
\usepackage[capitalize,noabbrev]{cleveref}

\usepackage{graphicx}
\usepackage{latexsym} 
\usepackage{enumerate}

\theoremstyle{plain}
\newtheorem{theorem}{Theorem}[section]
\newtheorem{proposition}{Proposition}[section]
\newtheorem{lemma}{Lemma}[section]
\newtheorem{corollary}{Corollary}[section]

\newtheorem{definition}{Definition}[section]

\newtheorem{remark}{Remark}[section]

\renewcommand{\eqref}[1]{\textnormal{(\ref{#1})}}

\numberwithin{equation}{section}

\newcommand{\rmd}{\mathrm{d}}

\newcommand{\R}{\mathbb{R}}

\def\Oh{{\mathcal O}}
\usepackage{color}
\newcommand{\bl}{\color{black}}

\begin{document}

\title[Geometric structures of conductive transmission eigenfunctions]{On the geometric structures of {\bl  transmission eigenfunctions with a conductive boundary condition and  applications}}

\author{Huaian Diao}
\address{School of Mathematics and Statistics, Northeast Normal University,
Changchun, Jilin 130024, China.}
\email{hadiao@nenu.edu.cn}

\author{Xinlin Cao}
\address{Department of Mathematics, Hong Kong Baptist University, Kowloon, Hong Kong, China.}
\email{xlcao.math@foxmail.com}

\author{Hongyu Liu}
\address{Department of Mathematics, City University of Hong Kong, Kowloon, Hong Kong, China.}
\email{hongyu.liuip@gmail.com, hongyliu@cityu.edu.hk}

%}

\begin{abstract}
This paper is concerned with the intrinsic geometric structures of conductive transmission eigenfunctions. The geometric properties of interior transmission eigenfunctions were first studied in \cite{BL2017b}. It is shown in two scenarios that the interior transmission eigenfunction must be locally vanishing near a corner of the domain with an interior angle less than $\pi$. We significantly extend and generalize those results in several aspects. First, we consider the conductive transmission eigenfunctions which include the interior transmission eigenfunctions as a special case. The geometric structures established for the conductive transmission eigenfunctions in this paper include the results in \cite{BL2017b} as a special case. Second, the vanishing property of the conductive transmission eigenfunctions is established for any corner as long as its interior angle is not $\pi$ when the conductive transmission eigenfunctions satisfy certain Herglotz functions approximation properties. That means, as long as the corner singularity is not degenerate, the vanishing property holds if the underlying conductive transmission eigenfunctions can be approximated by a sequence of Herglotz functions under mild approximation rates.  Third, the regularity requirements on the interior transmission eigenfunctions in \cite{BL2017b} are significantly relaxed in the present study for the conductive transmission eigenfunctions. In order to establish the geometric properties for the conductive transmission eigenfunctions, we develop technically new methods and the corresponding analysis is much more complicated than that in \cite{BL2017b}. Finally, as an interesting and practical application of the obtained geometric results, we establish a unique recovery result for the inverse problem associated with the transverse electromagnetic scattering by a single far-field measurement in simultaneously determining 
 a polygonal conductive obstacle and its surface conductive parameter.

\medskip



\end{abstract}

\maketitle

\section{Introduction}

Let $\Omega$ be a bounded Lipschitz domain in $\R^n$, $n=2, 3$, and $V\in L^\infty(\Omega )$ and $\eta\in L^\infty(\partial\Omega)$ be possibly complex-valued functions. Consider the following {\bl interior} transmission eigenvalue problem {\bl with a conductive boundary condition} for 
$v,\,w \in H^1(\Omega )$,
\begin{align}\label{eq:in eig}
%$$%\lb{1div}
\left\{
\begin{array}{l}
\Delta w+k^2(1+V) w=0 \quad\ \mbox{ in } \Omega, \\[5pt] 
\Delta v+ k^2  v=0\hspace*{1.85cm}\ \mbox{ in } \Omega, \\[5pt] 
w= v,\ \partial_\nu v + \eta v=\partial_\nu w \ \ \mbox{ on } \partial \Omega,
  \end{array}
\right.
%$$
\end{align}
where $\nu\in\mathbb{S}^{n-1}$ signifies the exterior unit normal vector to $\partial\Omega$. Clearly, $v=w\equiv 0$ are trivial solutions to \eqref{eq:in eig}. If for a certain $k\in\mathbb{R}_+$, there exists a pair of nontrivial solutions $(v, w)\in H^1(\Omega)\times H^1(\Omega)$ to \eqref{eq:in eig}, then $k$ is called a conductive transmission eigenvalue and $(v, w)$ is referred to as the corresponding pair of conductive transmission eigenfunctions. For a special case with $\eta\equiv 0$, \eqref{eq:in eig} is known to be the interior transmission eigenvalue problem. {\bl For terminological convenience, we refer to the nontrivial solutions $(v,\,w) \in H^1(\Omega )$ to  \eqref{eq:in eig} as the conductive transmission eigenfunctions and the corresponding $k$ is named as the conductive transmission eigenvalue.} The study on the transmission eigenvalue problems arises in the wave scattering theory and has a long and colourful history; see \cite{BP,CGH,CKP,CM,Kir,LV,PS,Robbiano,RS} for the spectral study of the interior transmission eigenvalue problem, and \cite{BHK,Bon,HK} for the related study of the conductive transmission eigenvalue problem, and a recent survey \cite{CHreview} and the references therein for comprehensive discussions on the state-of-the-art developments. The problem is a type of non-elliptic and non-self-adjoint eigenvalue problem, so its study is mathematically interesting and challenging. The existing results in the literature mainly focus on the spectral properties of the transmission eigenvalues, namely their existence, discreteness, infiniteness and Weyl's laws. Roughly speaking, the theorems for the transmission eigenvalues follow in a similar flavour to the results in the spectral theory of the Laplacian on a bounded domain. However, the transmission eigenfunctions reveal certain distinct and intriguing features. In \cite{BPS,PSV}, it is proved that the interior transmission eigenfunctions cannot be analytically extended across the boundary $\partial\Omega$ if it contains a corner with an interior angle less than $\pi$. In \cite{BL2017b}, geometric structures of interior transmission eigenfunctions were discovered for the first time. It is shown that under certain regularity conditions on the interior transmission eigenfunctions, the eigenfunctions must be locally vanishing near a corner of the domain with an interior angle less than $\pi$. With the help of numerics, it is further shown in \cite{BLLW,LW2018} that under the $H^1$-regularity of the interior transmission eigenfunctions, the eigenfunctions are either vanishing or localizing at a corner with an interior angle bigger than $\pi$. Recently, more geometric properties of the interior transmission eigenfunctions were discovered in \cite{BL2018,LW2018}, which are linked with the curvature of a specific boundary point. It is noted that a corner point considered in \cite{BL2017b,BLLW} can be regarded as having an infinite extrinsic curvature since the derivative of the normal vector has a jump singularity there. 

In addition to the angle of the corner, we would like to emphasize the critical role played by the regularity of the transmission eigenfunctions in the existing studies of the geometric structures in the aforementioned literatures. In \cite{BL2017b}, the regularity requirements are characterized in two ways. The first one is $H^2$-smoothness, and the other one is $H^1$-regularity with a certain Hergoltz approximation property. The $H^2$-regularity requirement can be weakened a bit to be H\"older-continuity with any H\"older index $\alpha\in (0,1)$.

In this paper, we establish the vanishing property of the conductive transmission eigenfunctions associated with \eqref{eq:in eig} at a corner as long as its interior angle is not $\pi$ {\bl when the conductive transmission eigenfunctions satisfy certain Herglotz functions approximation properties.} That means, as long as the corner singularity is not degenerate, the vanishing property holds {\bl if the underlying conductive transmission eigenfunctions can be approximated by a sequence of Herglotz functions under mild approximation rates.} In fact, in the three-dimensional case, the corner singularity is a more general edge singularity. To establish the vanishing property, we need to impose certain regularity conditions on the conductive transmission eigenfunctions which basically follow a similar manner to those considered in \cite{BL2017b}. That is, the first regularity condition is the H\"older-continuity with any H\"older index $\alpha\in (0,1)$, and the second regularity condition is characterized by the Herglotz approximation. Nevertheless, for the latter case, the regularity requirement is much more relaxed in the present study compared to that in \cite{BL2017b}. Finally, we would like to emphasize that in principle, the geometric properties established for the conductive transmission eigenfunctions include the results in \cite{BL2017b} as a special case by taking the parameter $\eta$ to be zero. Hence, in the sense described above, the results obtained in this work significantly extend and generalize the ones in \cite{BL2017b}.

The mathematical argument in \cite{BL2017b} is indirect which connects the vanishing property of the interior transmission eigenfunctions with the stability of a certain wave scattering problem with respect to variation of the wave field at the corner point. In \cite{Bsource,BL2018}, direct mathematical arguments based on certain microlocal analysis techniques are developed for dealing with the vanishing properties of the interior transmission eigenfunctions. However, the H\"older continuity on the interior transmission eigenfunctions is an essential assumption in \cite{Bsource,BL2018}. In this paper, in order to establish the vanishing property of the conductive transmission eigenfunctions under more general regularity conditions, we basically follow the direct approach. But we need to develop technically new ingredients for this different type of eigenvalue problem and the corresponding analysis becomes radically much more complicated. 

As an interesting and practical application, we apply the obtained geometric results for the conductive transmission eigenfunctions to an inverse problem associated with the transverse electromagnetic scattering. In a certain scenario, we establish the unique recovery result by a single far-field measurement in simultaneously determining a polygonal  conductive obstacle and its surface conductivity. This contributes to the well-known Schiffer's problem in the inverse scattering theory which is concerned with recovering the shape of an unknown scatterer by a single far-field pattern; see \cite{AR,BL2016,BL2017,CY,HSV,Liua,LPRX,LRX,Liu-Zou,Liu-Zou3,Ron2} and the references therein for background introduction and the state-of-the-art developments on the Schiffer's problem.

The rest of the paper is organized as follows. In Sections \ref{sec:2}  and \ref{sec:3}, we respectively derive the vanishing results of the conductive transmission eigenfunctions near a corner in the two-dimensional and three-dimensional cases. Section \ref{sec:4} is devoted to the uniqueness study in determining a polygonal conductive obstacle as well as its surface conductivity by a single far-field pattern. 
%Assume that there is a sector contained in $ \R^2$. 

\section{ Vanishing near corners of conductive transmission eigenfunctions: two-dimensional case}\label{sec:2}

In this section, we consider the vanishing near corners of conductive transmission eigenfunctions in the two-dimensional case. First, let us introduce some notations for the subsequent use. Let $(r,\theta)$ be the polar coordinates in $\R^2$; that is, $x=(x_1,x_2)=(r\cos\theta,r\sin\theta)\in\mathbb{R}^2$. For $x\in\mathbb{R}^2$, $B_h(x)$ denotes an {\bl open} ball of radius $h\in\mathbb{R}_+$ and centered at $x$. $B_h:=B_h(0)$. Consider an open sector in $\mathbb{R}^2$ with the boundary $\Gamma^\pm $ as follows, 
\begin{equation}\label{eq:W}
	W=\Big \{ x\in \R^2 ~ |~  x\neq 0,\quad   \theta_m < {\rm arg}(x_1+ i x_2 ) < \theta_M \Big \},
\end{equation}
 where $-\pi < \theta_m < \theta_M < \pi$, $i:=\sqrt{-1}$ and $\Gamma^+$ and $\Gamma^-$ respectively correspond to $(r, \theta_M)$ and $(r,\theta_m)$ with $r>0$.  Henceforth, set
\begin{equation}\label{eq:sh}
	S_h= W\cap B_h,\, \Gamma_h^{\pm }= \Gamma^{\pm } \cap B_h,\, \overline  S_h=\overline{ W} \cap B_h, \, 	\Lambda_h={\bl W \cap \partial { B_h} }, \,  \ \mbox{and}\ \Sigma_{\Lambda_{h}} = S_h \backslash S_{h/2} . 
\end{equation}
In Figure \ref{fig1}, we give a schematic illustration of the geometry considered here.

{\bl For $g\in L^2(\mathbb{S}^{n-1})$, the Herglotz wave function with kernel $g$ is defined by}
\begin{equation}\label{eq:hergnew}
v(x)=\int_{{\mathbb S}^{n-1}} e^{i k \xi \cdot x} g(\xi ) {\rm d} \sigma(\xi ),\ \ \xi\in\mathbb{S}^{n-1},\quad x\in\mathbb{R}^n. 
\end{equation}
%\begin{equation}\label{eq:herg}
%v_j(x)=\int_{{\mathbb S}^{n-1}} e^{i k \xi \cdot x} g_j(\xi ) {\rm d} \sigma(\xi ),\ \ \xi\in\mathbb{S}^{n-1}, x\in\mathbb{R}^n. 
%\end{equation}
It can be easily seen that $v$ is an entire solution to the Helmholtz equation $\Delta v+k^2 v=0$. {\bl By Theorem 2 and Remark 2 in \cite{Wec}, we have the following Herglotz approximation result. 

\begin{lemma}\label{lem:Herg}
Let $\Omega \Subset \mathbb R^n$ be a bounded Lipschitz domain and  ${\mathbf H}_k$ be the space of all Herglotz wave functions of the form \eqref{eq:hergnew}. Define 
$$
{\mathbf S}_k(\Omega ) =  \{u\in C^\infty (\Omega)~|~ \Delta u+k^2u=0\}
$$
and 
$$
{\mathbf H}_k(\Omega ) =  \{u|_\Omega~|~ u\in {\mathbf H}_k\}. 
$$
Then  ${\mathbf H}_k(\Omega )$ is dense in ${\mathbf S}_k(\Omega )  \cap  L^2 ( \Omega )$ with respect to the topology induced by the $H^1(\Omega)$-norm. 
\end{lemma}
\begin{remark}
	From Lemma \ref{lem:Herg}, for any $v \in H^1(\Omega )$ being a solution to the Helmholtz equation in $\Omega$, we can conclude that there exists a sequence of Herglotz functions which can approximate $v$ to an arbitrary accuracy. 
	\end{remark}
}

\begin{figure}
  \centering
  \includegraphics[width=0.35\textwidth]{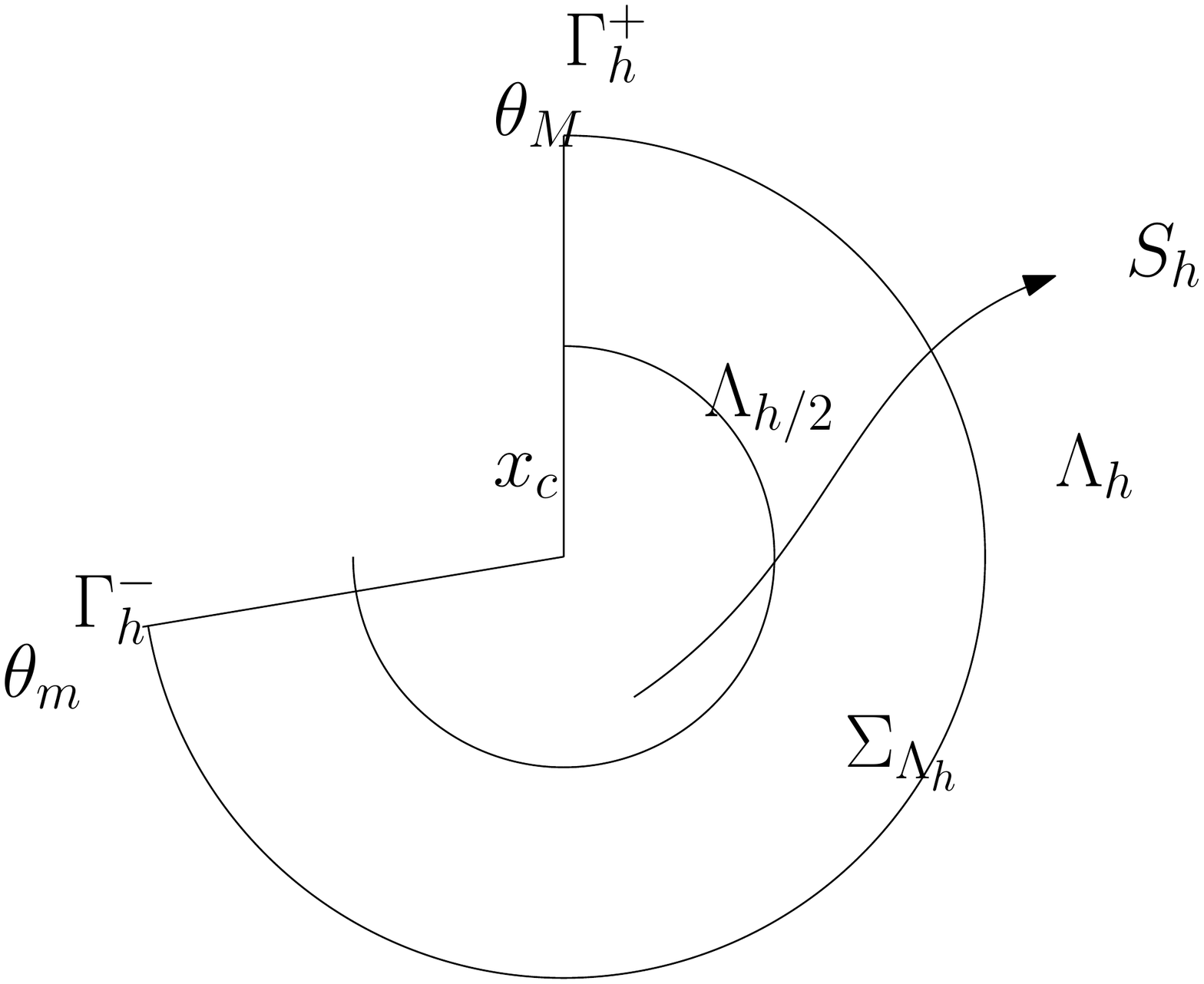}
  \caption{Schematic illustration of the corner in 2D.}
  \label{fig1}
\end{figure}

% Moreover, in order to prove the vanishing property of the conductive transmission eigenfunctions $v$ and $w$, we require that $v$ can be approximated by $v_j$ in the exponential decay rate with respect to $j$ while the $L^2$ norm of the kernel $g_j$ can increase by a positive  power function of $j$. Using the complex geometrical optics solution  $u_0(sx)$ constructed in  \cite{Bsource} where $s>0$, we first derive the integral identity between the domain integrals and boundary integrals. Adopting Jacobi-Anger expansion (cf. \cite[Page 75]{CK}), we have explicit  expressions of the Herglotz wave functions $v_j$ involving  the Bessel functions of the first kind \cite{Abr}, which can be used to carefully investigate the order of $s$ of each integrals involving the CGO solution $u_0(sx)$ when $s \rightarrow \infty$. We will show that if the conductivity parameter $\eta$ does not vanish at the corner, $qw$ and $\eta$ are $C^\alpha $ function near the corner for $0< \alpha <1$, and the open angle of the  sector $W$ containing the corner does not equal to $\pi$, the vanishing property of $v$ must hold. If higher regularities of $v$ and $w$ are allowed (i.e., $v$ and $w$ are $H^2$ smooth near the corner), the problem  can be tackled in a much easier manner; see Theorem \ref{Th:1.2}. 
We shall also need the following lemma, which gives a particular type of planar complex geometrical optics (CGO) solution whose logarithm is a branch of the square root (cf. \cite{Bsource}). 

\begin{lemma}\label{lem:1}\cite[Lemma 2.2]{Bsource}
For $x\in \R^2$ denote $r=|x|,\, \theta={\rm arg}(x_1 +i x_2)$. Let $s\in \mathbb R_+$ and 
\begin{equation}\label{eq:u0}
	u_0(sx):= \exp \left( \sqrt {sr} \left(\cos \left(\frac{\theta}{2}+\pi\right) +i \sin \left(\frac{\theta}{2} +\pi\right)  \right ) \right) .
\end{equation}
Then $\Delta u_0=0$ in $\R^2\backslash\mathbb{R}_{0,-}^2  $, where $\R_{0,-}^2:=\{{ x}\in\mathbb{R}^2|{ x}=(x_1,x_2); x_1\leq 0, x_2=0\}$, and $s \mapsto u_0(sx) $ decays exponentially in $\R_+$.  Let $\alpha, s >0$. Then
\begin{equation}\label{eq:xalpha}
	\int_W |u_0(sx)| |x|^\alpha {\rm d} x \leq \frac{2(\theta_M-\theta_m )\Gamma(2\alpha+4) }{ \delta_W^{2\alpha+4}} s^{-\alpha-2},
\end{equation}
where $\delta_W=-\max_{ \theta_m < \theta <\theta_M }  \cos(\theta/2+\pi ) >0$. Moreover
\begin{equation}\label{eq:u0w}
	\int_W u_0(sx) {\rm d} x= 6 i (e^{-2\theta_M i }-e^{-2\theta_m i }  ) s^{-2},
	\end{equation}
	and for $h>0$
	\begin{equation}\label{eq:1.5}
	%\[
	\int_{W \backslash B_h } |u_0(sx)|   {\rm d} x \leq \frac{6(\theta_M-\theta_m )}{\delta_W^4} s^{-2} e^{-\delta_W \sqrt{hs}/2}. 
	%\]
	\end{equation}
\end{lemma}

%For 2D, now we are in the position to prove the main theorem below, where we  show that the  eigenfunctions $v$ and $w$  to   \eqref{eq:in eig} vanish at the corner point if there is a corner point of $\Omega$ under some mild assumptions.  Before that, 

%Let $D \subset \R^n,\, n\in \{2,3\}$, represent a collection of bounded simply connected domains, and we define the Sobolev space 
%$$
%H_0^1( D ) =\left  \{ u \in L^2(D):  \nabla u  \in L^2(D )  \mbox{ and } u=0 \mbox{ on } \partial D  \right \}. 
%$$

{\bl The following lemma states the the regularity of the CGO solution $u_0(sx)$ defined in \eqref{eq:u0}.
\begin{lemma}\label{lem:23}
	Let $S_h$ be defined in \eqref{eq:sh} and $u_0(sx)$ be given by \eqref{eq:u0}. Then $u_0(sx) \in H^1(S_h)$ and $\Delta u_0 (sx)=0$ in $S_h$.  Furthermore, it holds that
	\begin{equation} \label{eq:u0L2}
	\|u_0(sx)\|_{L^2(S_h)}^2\leq 	\frac{ (\theta_M-\theta_m) e^{- 2\sqrt{s \Theta } \delta_W } h^2  }{2}
	\end{equation}
	and
	\begin{equation}\label{eq:22}
		\left  \||x|^\alpha  u_0(sx) \right \|_{L^{2}(S_h ) }^2 \leq s^{-(2\alpha+2 )} \frac{2(\theta_M-\theta_m)  }{(4\delta_W^2)^{2\alpha+2  } } \Gamma(4\alpha+4), 
	\end{equation}
where $ \Theta  \in [0,h ]$ and $\delta_W$ is defined in \eqref{eq:xalpha}. 
\end{lemma}
\begin{proof}
	Recalling the expression of $u_0$ given in \eqref{eq:u0}, using change of variables and the integral mean value theorem, we can deduce that
\begin{align}\notag
	\|u_0(sx)\|_{L^2(S_h)}^2&=\int_{0}^h r {\rm d} r\int_{\theta_m}^{\theta_M} e^{2\sqrt{sr} \cos(\theta/2+\pi) }{\rm d} \theta \leq \int_{0}^h r {\rm d} r\int_{\theta_m}^{\theta_M} e^{-2\sqrt{sr} \delta_W }{\rm d} \theta \nonumber
	\\
	&=\frac{ (\theta_M-\theta_m) e^{- 2\sqrt{s \Theta } \delta_W } h^2  }{2}	, \notag \end{align}
where $\Theta  \in [0,h ]$ and $\delta_W$ is defined in \eqref{eq:xalpha}.  Furthermore, it can directly verified that
	\begin{equation}\notag
	\begin{split}
		\frac{\partial u_0(sx)}{\partial r}&=-\frac{ s^{1/2}}{2r^{1/2}} e^{-\sqrt{sr} (\cos( \theta/2)+i \sin (\theta/2 ) )+i \theta/2 },\\ \frac{\partial u_0(sx)}{\partial \theta}&=- \frac{i \sqrt{sr}}{2} e^{-\sqrt{sr} (\cos( \theta/2)+i \sin (\theta/2 ) )+i \theta/2 },
	\end{split}
	\end{equation}
	which can be used to obtain that
		\begin{equation}\notag
	\begin{split}
		\frac{\partial u_0(sx)}{\partial x_1}&=-\frac{ s^{1/2}}{2r^{1/2}} e^{-\sqrt{sr} (\cos( \theta/2)+i \sin (\theta/2 ) )-i \theta/2 },\\ \frac{\partial u_0(sx)}{\partial x_2}&=- \frac{i s^{1/2}}{2r^{1/2}} e^{-\sqrt{sr} (\cos( \theta/2)+i \sin (\theta/2 ) )-i \theta/2 }.
	\end{split}
	\end{equation}
	Therefore, it yields that
	$$
	\|\nabla u_{0}(sx)\|_{L^2(S_h) }^2 \leq \frac{(\theta_M-\theta_m)sh}{2}e^{-2\sqrt{s\vartheta}\delta_W}
	$$
	by the integral mean value theorem,
%	by using the fact that
%	$$
%	e^{-\sqrt{sr} \sin( \theta/2)} \leq e^{-\sqrt{sr} \delta_W} \leq 1 \mbox{ for }  \theta \in [\theta_m, \theta_M ] , 
%	$$
	where $\vartheta\in[0,h]$ and $\delta_W$ is defined in \eqref{eq:xalpha}. Hence, we know that $u_0(sx) \in H^1(S_h)$ and $\Delta u_0 (sx)=0$ in the weak sense. 
	
Using polar coordinates transformation, we can deduce that
\begin{align}\notag
	&\left  \||x|^\alpha  u_0(sx) \right \|_{L^{2}(S_h ) }^2=\int_0^h r {\rm d} r \int_{\theta_m }^{\theta_M} r^{2\alpha  } e^{2 \sqrt{sr} \cos (\theta/2+\pi) } {\rm d} \theta \nonumber \\
	 \leq & \int_0^h r {\rm d} r \int_{\theta_m }^{\theta_M} r^{2\alpha  } e^{-2 \sqrt{sr} \delta_W } {\rm d} \theta=(\theta_M-\theta_m ) \int_{0}^h r^{2\alpha+1} e^{-2 \delta_W \sqrt{sr} }{\rm d} r\quad (t=2 \delta_W\sqrt{sr}) \nonumber \\
	 = & s^{-(2\alpha+2 )} \frac{2(\theta_M-\theta_m)  }{(4\delta_W^2)^{2\alpha+2  } } \int_{0}^{2 \delta_W \sqrt{sh }} t^{4\alpha+3} e^{-t }{\rm d} r \leq s^{-(2\alpha+2 )} \frac{2(\theta_M-\theta_m)  }{(4\delta_W^2)^{2\alpha+2  } } \Gamma(4\alpha+4),  \notag
\end{align}
where $\delta_W$ is defined in \eqref{eq:xalpha}.  This completes the proof of the lemma.  
\end{proof}

\begin{lemma}\label{lem:zeta}
	For any $\zeta>0$, if $\omega(\theta ) >0 $, then
\begin{equation}\label{eq:zeta}
	 \int_{0}^h r^\zeta  e^{-\sqrt{sr} \omega(\theta)} {\rm d} r=\Oh( s^{-\zeta-1} ),
\end{equation}
as  $s\rightarrow +\infty$.   
\end{lemma}

\begin{proof}
	Using variable substitution $t=\sqrt{sr}$, it is easy to derive \eqref{eq:zeta}. 
\end{proof}

%\begin{remark}
%	Here, \eqref{eq:zeta} shows that the lowest increasing term in the integral of  \eqref{eq:zeta}  with respect to $s$ as $s \rightarrow +\infty$ is $s^{-\zeta-1} $.
%\end{remark}

%The follow lemma is regarding on the Green formula in a bounded Lipschitz domain for functions with $H^1$-regularity, which shall be used to establish the integral equality in the subsequent proof of the vanishing property of the conductive transmission eigenfunction. 

Next, we recall a special type of Green formula for $H^1$ functions, which shall be needed in establishing a key integral identify for deriving the vanishing property of the conductive transmission eigenfunction. 

\begin{lemma}\label{lem:green}
	Let $\Omega\Subset \mathbb R^n $ be a bounded Lipschitz domain. For any $f,g\in H^1_\Delta:=\{f\in H^1(\Omega)|\Delta f\in L^2(\Omega)\}$, there holds the following second Green identity:
	\begin{equation}\label{eq:GIN1}
		\int_\Omega(g\Delta f-f\Delta g)\,\mathrm{d}x=\int_{\partial\Omega}(g\partial_{\nu}f-f\partial_{\nu}g)\,\mathrm{d}\sigma.
	\end{equation}
\end{lemma}
Lemma \ref{lem:green} is a special case of more general results in \cite[Lemma 3.4]{costabel88} and \cite[Theorem 4.4]{McLean}. It is pointed out that in Lemma~\ref{lem:green} one needs not to require that $H^2$-regularity as the usual Green formula. In particular, for the transmission eigenfunctions $(v, w)\in H^1(\Omega)\times H^1(\Omega)$ to \eqref{eq:in eig}, we obviously have $v, w\in H_\Delta^1$, and hence the Green identity \eqref{eq:GIN1} holds for $v, w$. This fact shall be frequently used in our subsequent analysis.

We proceed to derive several auxiliary lemmas that shall play a key role in establishing our first main result in Theorem~\ref{Th:1.1} in what follows. 

 \begin{lemma}
 	\label{lem:int1}
Let $S_h$ and $\Gamma_h^\pm$ be defined in \eqref{eq:sh}. Suppose that $v \in H^1(S_h )$ and $w  \in H^1(S_h) $ satisfy the following PDE system, 
	\begin{align}\label{eq:in eignew}
%$$%\lb{1div}
\left\{
\begin{array}{l}
\Delta w+k^2q w=0 \hspace*{1.6cm} \mbox{ in } S_h, \\[5pt] 
\Delta v+ k^2  v=0\hspace*{1.85cm}\ \mbox{ in } S_h \\[5pt] 
w= v,\ \partial_\nu v + \eta  v=\partial_\nu w \ \ \mbox{ on } \Gamma_h^\pm ,
  \end{array}
\right.
%$$
\end{align}
with  $\nu\in\mathbb{S}^{1}$ signifying the exterior unit normal vector to $\Gamma_h^\pm $, $k \in \mathbb{R}_+ $, $q\in L^\infty(S_h ) $ and $\eta ( x) \in C^{\alpha}(\overline{\Gamma_h^\pm } )$, where
\begin{equation}\label{eq:eta}
	\eta(x)=\eta(0)+\delta \eta(x),\quad |\delta \eta(x)| \leq \|\eta \|_{C^\alpha  } |x|^\alpha .
\end{equation}
	Recall that the CGO solution $u_0(sx)$ is defined in \eqref{eq:u0} with the parameter $s\in \mathbb R_+$.  Then the following integral equality holds
\begin{equation}\label{eq:221 int}
\begin{split}
	\int_{S_h} u_0(sx) (f_1-f_2)\rmd x&=\int_{\Lambda_h} (u_0(sx) \partial_\nu (v-w)- (v-w) \partial_\nu u_0(sx))\rmd \sigma\\
	&\quad -\int_{\Gamma_h^\pm } \eta u_0(sx) v\rmd \sigma, 
	\end{split}
\end{equation}	
where $f_1=-k^2 v$	and $f_2=-k^2 qw$. 
	
	Denote 
 \begin{equation}\label{eq:deltajs}
	 \widetilde{f}_{1j} (x) =- k^2v_j (x),
	 \end{equation}
where 	 \begin{equation}\label{eq:herg}
v_j(x)=\int_{{\mathbb S}^{1}} e^{i k \xi \cdot x} g_j(\xi ) {\rm d} \sigma(\xi ),\quad \xi\in\mathbb{S}^{1}, \quad x\in\mathbb{R}^2, \quad g_j\in L^2(\mathbb S^1)
\end{equation}
is the Herglotz wave function. 
Denote
\begin{equation}\label{eq:varphi}
	 \varphi=\angle(\xi,x). 
\end{equation}
Then  $\widetilde{f}_{1j} (x) \in C^\alpha (\overline S_h )$ and it has the expansion
\begin{equation}\label{eq:f1jf2 notation new}
		\widetilde f_{1j}(x)=- k^2v_j (x)=\widetilde f_{1j} (0)+\delta \widetilde f_{1j} (x),\quad |\delta \widetilde f_{1j} (x)  | \leq \|\widetilde f_{1j} \|_{C^\alpha } |x|^\alpha. 
\end{equation}
Assume that $f_2=-k^2 qw\in C^\alpha(\overline S_h)$ ($0<\alpha <1$) satisfying
\begin{equation}\label{eq:f1jf2 notation}
		f_2(x)=f_2(0)+\delta f_2(x),\quad |\delta f_2(x)  | \leq \|f_2\|_{C^\alpha } |x|^\alpha,
\end{equation}	
it holds that
\begin{align}\label{eq:intimportant}
	( \widetilde f_{1j} (0)-f_2(0)) \int_{S_h} u_0(sx) {\rm d} x+\delta_j(s)  &= I_3-I_2^\pm  -\int_{S_h}  \delta \widetilde f_{1j} (x) u_0(sx) {\rm d} x \nonumber \\
	&\quad +\int_{S_h}  \delta f_2(x) u_0(sx) {\rm d} x -\xi_j^\pm(s). 
\end{align}
where
\begin{equation}\label{eq:intnotation}
\begin{split}
	 I_2^\pm &=\int_{\Gamma_{h } ^\pm } \eta(x)  u_0(sx) v_j (x)  {\rm d} \sigma ,\quad I_3 =\int_{\Lambda_h} ( u_0 (sx)\partial_\nu (v-w)- (v-w)\partial_\nu u_0(sx)  ) {\rm d} \sigma,\\
	 \delta_j(s)&=-k^2 \int_{S_h} ( v(x)-v_j(x))u_0(sx)  {\rm d} x,\quad 
	   \xi_{j}^\pm(s)=  \int_{\Gamma_{h } ^\pm } \eta(x)  u_0(sx) (v(x)- v_j (x) )   {\rm d} \sigma.
	\end{split} 
	\end{equation}
Furthermore, it yields that
\begin{align}\label{eq:deltaf1j}
	\left|\int_{S_h}  \delta \widetilde f_{1j} (x) u_0(sx) {\rm d} x \right | & \leq \frac{2 \sqrt{2\pi } (\theta_M- \theta_m) \Gamma(2 \alpha+4)}{ \delta_W^{2\alpha+4   }  } k^2  {\rm diam}(S_h)^{1-\alpha }\\
	&\quad \quad \times   (1+k) \|g_j\|_{L^2( {\mathbb S}^{1})}   s^{-\alpha-2   },  \nonumber
\end{align}
and 
\begin{equation}\label{eq:deltaf2}
	\left| \int_{S_h}  \delta f_2(x) u_0(sx) {\rm d} x\right|  \leq \frac{2(\theta_M- \theta_m) \Gamma(2 \alpha+4) }{ \delta_W^{2\alpha+4   }  } \|f_2\|_{C^\alpha }  s^{-\alpha-2   },
\end{equation}
as $s \rightarrow + \infty$. 
 If $v-w\in H^2(\Sigma_{ \Lambda_h} )$, then one has as $s \rightarrow + \infty$ that
\begin{align}\label{eq:I3}
\left| I_3 \right| &\leq  C e^{-c' \sqrt s},  
\end{align} 
 where $C$ and $c'$ are two positive constants.
 \end{lemma}

\begin{proof}
From \eqref{eq:in eignew}, we have
\begin{equation}\label{eq:vw}
	\Delta v  =-k^2  v:= f_1,\quad  \Delta w =-k^2 q w:= f_2. 
	\end{equation}
 Subtracting the two equations of \eqref{eq:vw} together with the use of the boundary conditions of \eqref{eq:in eignew} we deduce that
\begin{equation}\label{eq:219 pde}
	\Delta(v-w )=f_1-f_2 \mbox{ in } S_h, \quad v-w=0, \,   \partial_\nu (v-w)=-\eta v \mbox{ on }   \Gamma_h^\pm . 
\end{equation}

Recall that $u_0(sx)\in H^1(S_h)$ from Lemma \ref{lem:23}. Since $v,\, w \in H^1(S_h)$ and $q\in L^{\infty}(S_h)$, it yields that $f_1,\ f_2 \in L^2(S_h)$. Since $S_h$ is obviously a bounded Lipschitz domain, by virtue of Lemma \ref{lem:green}, we have the following integral identity 
\begin{equation}\label{eq:220 int}
	\int_{S_h} u_0(sx) \Delta (v-w)\rmd x=\int_{\partial S_h} u_0(sx) \partial_\nu (v-w)- (v-w) \partial_\nu u_0(sx)\rmd \sigma,
\end{equation}
by using the fact that $\Delta u_0(sx)=0$ in $S_h$. Substituting \eqref{eq:219 pde} into \eqref{eq:220 int} it yields \eqref{eq:221 int}.

%Since $v \in H^1(S_h)$ is a solution of the Helmholtz equation in $S_h$, from \cite[Theorem 2.1]{Wec}, $v$ can be approximated by Herglotz wave function
%\begin{equation}\label{eq:herg}
%	v_j(x)=\int_{{\mathbb S}^{n-1}} e^{i k \theta \cdot x} g_j(\theta ) {\rm d} \sigma(\theta )
%\end{equation}
%in the topology induced by $H^1$- norm. 

Recall that $v$ can be approximated by the Herglotz wave function $v_j$
 given by \eqref{eq:herg} in the topology induced by the $H^1$- norm.  It is clear that
\begin{align}\label{eq:222 int}
	\int_{S_h} f_1(x)u_0(sx)  {\rm d} x= \int_{S_h } \widetilde{f}_{1j} (x) u_0(sx)  {\rm d} x+ \delta_j(s),
	\end{align}
where $ \widetilde{f}_{1j} (x)$ and $\delta_j(s)$ are defined in   \eqref{eq:deltajs} and \eqref{eq:intnotation}, respectively. Furthermore, it can be derived that
\begin{align}\label{eq:int3}
	\int_{\Gamma_{h } ^\pm } \eta(x)  u_0(sx) v(x)  {\rm d} \sigma &=\int_{\Gamma_{h } ^\pm } \eta(x)  u_0(sx) v_j (x)  {\rm d} \sigma + \xi_{j}^\pm(s),   
\end{align}
where $ \xi_{j}^\pm(s)$ is defined in \eqref{eq:intnotation}. 

Combining \eqref{eq:221 int}, \eqref{eq:222 int}  with \eqref{eq:int3}, we have  the following integral identity: 
%\begin{align}\label{eq:int identy1}
%	\int_{S_h} u_0(sx) ( \widetilde f_{1j} (x)-f_2(x)) {\rm d} x+\delta_j  &= \int_{\Lambda_h} ( u_0 (sx)\partial_\nu (v-w)- (v-w)\partial_\nu u_0(sx)  ) {\rm d} \sigma \\
%	&- \int_{\Gamma_{h } ^\pm } \eta(x)  u_0(sx) v_j (x)  {\rm d} \sigma + \xi_j^\pm .  \nonumber
%\end{align}
\begin{align}\label{eq:int identy}
	I_1+\delta_j (s) &= I_3 -  I_2^\pm - \xi_j^\pm(s) , 
\end{align}
where 
\begin{equation}\label{eq:I1 notation}
	I_1= \int_{S_h} u_0(sx) ( \widetilde f_{1j} (x)-f_2(x)) {\rm d} x,
\end{equation}
 $I_2^\pm$, $I_3$, $\delta_j (s)$ and $ \xi_j^\pm(s)$ are defined in \eqref{eq:intnotation}. 

It is easy to verify that $v_j\in C^\alpha(\overline{ S}_h)$. Therefore $\widetilde f_{1j}\in C^\alpha(\overline{ S}_h )$ and for $x \in S_h$ we have the splitting \eqref{eq:f1jf2 notation new}. Since $f_2 \in C^\alpha(\overline S_h )$, substituting \eqref{eq:f1jf2 notation new} and \eqref{eq:f1jf2 notation} into $I_1$ defined in \eqref{eq:I1 notation}, 
%\begin{equation}
%	\widetilde f_{1j}(x)&=\widetilde f_{1j} (0)+\delta \widetilde f_{1j} (x),\quad |\delta \widetilde f_{1j} (x)  | \leq \|\widetilde f_{1j} \|_{C^\alpha } |x|^\alpha, \nonumber \\
%	f_2(x)&=f_2(0)+\delta f_2(x),\quad |\delta f_2(x)  | \leq \|f_2\|_{C^\alpha } |x|^\alpha.
%\end{align}
we have
\begin{align}\label{eq:I1 227}
I_1 &=( \widetilde f_{1j} (0)-f_2(0)) \int_{S_h} u_0(sx) {\rm d} x+\int_{S_h}  \delta \widetilde f_{1j} (x) u_0(sx) {\rm d} x-\int_{S_h}  \delta f_2(x) u_0(sx) {\rm d} x.
\end{align}
Substituting \eqref{eq:I1 227} into \eqref{eq:int identy}, we can further deduce the  integral equality \eqref{eq:intimportant}. 

In the following, we shall prove \eqref{eq:deltaf1j},  \eqref{eq:deltaf2} and \eqref{eq:I3}, separately. From \eqref{eq:xalpha} and \eqref{eq:f1jf2 notation new}, it can be derived that
\begin{align}\label{eq:deltaf1j int bound}
	\left|\int_{S_h}  \delta \widetilde f_{1j} (x) u_0(sx) {\rm d} x \right |&  \leq \int_{S_h}   \left| \delta \widetilde f_{1j} (x) \right |  |u_0(sx) | {\rm d} x \leq \|\widetilde f_{1j} \|_{C^\alpha } \int_{W }   |u_0(sx) | |x|^\alpha {\rm d} x  \nonumber  \\
	&\leq \frac{2(\theta_M- \theta_m) \Gamma(2 \alpha+4) }{ \delta_W^{2\alpha+4   }  } \|\widetilde f_{1j} \|_{C^\alpha }  s^{-\alpha-2   }.
\end{align}
Recall that $\widetilde f_{1j}=- k^2v_j (x)$ and $v_j$ is the  Herglotz wave function given by \eqref{eq:herg}. Using the property of compact embedding of H{\"o}lder spaces, we can derive that
$$
\| \widetilde f_{1j} \|_{C^\alpha } \leq k^2  {\rm diam}(S_h)^{1-\alpha } \|v_j\|_{C^1},
$$
where $ {\rm diam}(S_h)$ is the diameter of $S_h$. By direct computation, we have
$$
\|v_j\|_{C^1} \leq \sqrt{2\pi } (1+k) \|g_j\|_{L^2( {\mathbb S}^{1})},
$$
and therefore we can deduce \eqref{eq:deltaf1j}. Similarly, using  \eqref{eq:xalpha} and \eqref{eq:f1jf2 notation},  we can derive  \eqref{eq:deltaf2}.

It is easy to see that on $\Lambda_h$, one has
\begin{align*}
	|u_0(sx)|&=e^{\sqrt{sr} \cos (\theta/2 +\pi ) } \leq e^{-\delta_W \sqrt{s h }},\\\
	\left| \partial_{\nu}  u_0(sx) \right| &=\left| \frac{\sqrt{s} e^{i \cos (\theta/2+\pi)} }{2\sqrt{h}}  e^{\sqrt {sh }  \exp (i (\theta/2+\pi))}\right| \leq \frac{1}{2} \sqrt{\frac{s}{h}}e^{-\delta_W \sqrt{s h }},
\end{align*}
both of which decay exponentially as $s \rightarrow \infty$. Hence we know that
$$
\left\| u_0(sx) \right\|_{L^2(\Lambda_h )} \leq e^{-\delta_W \sqrt{s h }} \sqrt{ (\theta_M- \theta_m) h} , \quad \left\|\partial_{\nu}   u_0(sx) \right\|_{L^2(\Lambda_h )} \leq  \frac{1}{2} e^{-\delta_W \sqrt{s h }} \sqrt{s(\theta_M- \theta_m) }.
$$
Under the assumption that $v-w\in H^2(\Sigma_{ \Lambda_h} )$, using Cauchy-Schwarz inequality and the trace theorem, we can prove as $s \rightarrow +\infty$ that 
\begin{align}\notag %\label{eq:I3}
\left| I_3 \right| &\leq \left\| u_0(sx) \right\|_{L^2(\Lambda_h )} \left\| \partial_\nu (v-w)  \right\|_{L^2(\Lambda_h )} +\left\|\partial_\nu  u_0(sx) \right\|_{L^2(\Lambda_h )} \left\|  v-w  \right\|_{L^2(\Lambda_h )} \\
&\leq  \left ( \left\| u_0(sx) \right\|_{L^2(\Lambda_h )}+\left\|\partial_\nu  u_0(sx) \right\|_{L^2(\Lambda_h )} \right) \left\|  v-w  \right\|_{H^2(\Sigma_{ \Lambda_h } )} \leq C e^{-c' \sqrt s},  \notag
\end{align} 
 where $C,c'$ are positive constants.
 
 The proof is complete. 
\end{proof}

\begin{lemma}\label{lem:27}
	Under the same setup in Lemma \ref{lem:int1}, we assume that a sequence of Herglotz wave functions $\{v_j\}_{j=1}^{+\infty} $ possesses the form \eqref{eq:herg}, which can approximate $v$  in $H^1(S_h)$  satisfying
	\begin{equation}\label{eq:ass1}
		\|v-v_j\|_{H^1(S_h)} \leq j^{-1-\Upsilon },\quad  \|g_j\|_{L^2({\mathbb S}^{1})} \leq C j^{\varrho},
	\end{equation}
	for some constants $C>0$, $\Upsilon >0$ and $0< \varrho<1 $.  Then we have the following estimates: 
\begin{equation}\label{eq:deltajnew}
	|\delta_j(s)| \leq    \frac{\sqrt { \theta_M-\theta_m } k^2 e^{-\sqrt{s \Theta } \delta_W } h } {\sqrt 2 }  j^{-1-\Upsilon},
\end{equation}
and
\begin{equation}\label{eq:xij}
	|\xi_j^\pm (s) |\leq C \left( |\eta(0)|  \frac{\sqrt { \theta_M-\theta_m }  e^{-\sqrt{s \Theta } \delta_W } h } {\sqrt 2 }  + \|\eta\|_{C^\alpha } s^{-(\alpha+1 )} \frac{\sqrt{2(\theta_M-\theta_m) \Gamma(4\alpha+4) } }{(2\delta_W)^{2\alpha+2  } } \right) j^{-1-\Upsilon}. 
\end{equation}
where $ \Theta  \in [0,h ]$ and $\delta_W$ is defined in \eqref{eq:xalpha}.
\end{lemma}

\begin{proof}
	In the following, we shall prove \eqref{eq:deltajnew} and \eqref{eq:xij}, separately. Recall that $u_0 \in H^1(S_h)$ from Lemma \ref{lem:23}.  Clearly $\widetilde{f}_{1j} (x)  \in H^2(S_h)$, which can be embedded into $C^\alpha(\overline{ S}_h) $ for $\alpha\in(0,1)$.  Moreover, by using the Cauchy-Schwarz inequality,  
we know that
\begin{equation}\label{eq:deltaj}
	|\delta_j(s)|\leq k^2 \|v-v_j\|_{L^2(S_h)}  \|u_0(sx)\|_{L^2(S_h)}.
\end{equation}
Substituting \eqref{eq:u0L2}  into \eqref{eq:deltaj} and using \eqref{eq:ass1}, one readily has \eqref{eq:deltajnew}. 

Since $\eta \in C^\alpha\left(\overline{\Gamma_h^\pm} \right)$, we have the expansion of $\eta(x)$ at the origin  as \eqref{eq:eta}.  Therefore, using Cauchy-Schwarz inequality and the trace theorem, we have
\begin{align*}%\label{eq:19}
	|\xi_{j}^\pm(s)|&\leq  | \eta(0) |\int_{\Gamma_{h } ^\pm }   |u_0(sx)|  |v(x)- v_j (x) |    {\rm d} \sigma +  \| \eta \|_{C^\alpha }\int_{\Gamma_{h } ^\pm } |x|^\alpha   |u_0(sx)|  |v(x)- v_j (x) |    {\rm d} \sigma  \\
	 & \leq  | \eta(0) | \|v-v_j\|_{H^{1/2}(\Gamma_h^\pm ) } \| u_0(sx)\|_{H^{-1/2}(\Gamma_h^\pm ) } \\
	 &\quad + \| \eta \|_{C^\alpha }  \|v-v_j\|_{H^{1/2}(\Gamma_h^\pm ) } \| |x|^\alpha u_0(sx)\|_{H^{-1/2}(\Gamma_h^\pm ) }  \\
	 & \leq  | \eta(0) | \|v-v_j\|_{H^{1}(S_h ) } \| u_0(sx)\|_{L^2(S_h ) }  + \| \eta \|_{C^\alpha }  \|v-v_j\|_{H^{1}(S_h ) } \| |x|^\alpha u_0(sx)\|_{L^2(S_h ) }  . 
	 \end{align*}
%where $C$ is a positive constant.
% from Minkowski inequality it can be derived that  
%\begin{align}\label{eq:21}
%	 \|\eta (x) u_0(sx)\|_{L^{2}(S_h ) } \leq |\eta(0)|  \|u_0(sx)\|_{L^{2}(S_h ) }+  \|\eta \|_{C^\alpha  }  \||x|^\alpha  u_0(sx)\|_{L^{2}(S_h ) }.  
%\end{align}
%Combining  \eqref{eq:u0L2} with  \eqref{eq:22}, from \eqref{eq:21}, we have the estimation
%\begin{equation}\label{eq:etau0}
% \|\eta (x) u_0(sx)\|_{L^{2}(S_h ) } \leq |\eta(0)|  \frac{\sqrt { \theta_M-\theta_m }  e^{-\sqrt{s \Theta } \delta_W } h } {\sqrt 2 }  + \|\eta\|_{C^\alpha } s^{-(\alpha+1 )} \frac{\sqrt{2(\theta_M-\theta_m) \Gamma(4\alpha+4) } }{(2\delta_W)^{2\alpha+2  } }. 
%\end{equation}
Using \eqref{eq:ass1}, \eqref{eq:u0L2} and \eqref{eq:22},  we readily derive \eqref{eq:xij}.

The proof is complete.  
\end{proof}

\begin{lemma}\label{lem:u0 int}
	Recall that $\Gamma_h^\pm$ and $u_0(sx)$ are defined in \eqref{eq:sh} \eqref{eq:u0}, respectively. We have	\begin{align}\label{eq:I311}
	\begin{split}
\int_{\Gamma_{h} ^+ }  u_0 (sx)  {\rm d} \sigma &=2 s^{-1}\left( \mu(\theta_M )^{-2}-   \mu(\theta_M )^{-2} e^{ -\sqrt{sh} \mu(\theta_M ) }\right.  \\&\left. \hspace{3.5cm} -  \mu(\theta_M )^{-1} \sqrt{sh}   e^{ -\sqrt{sh} \mu(\theta_M ) }  \right  ),  \\
%	&+ \left( \mu(\theta_M )^{-2}-   \mu(\theta_M )^{-2} e^{ -\sqrt{sh}} -  \mu(\theta_M )^{-1} \sqrt{sh}   e^{ -\sqrt{sh}} \right  ) \Bigg] ,\\
\int_{\Gamma_{h} ^- }  u_0 (sx)  {\rm d} \sigma &=2 s^{-1} \left( \mu(\theta_m )^{-2}-   \mu(\theta_m )^{-2} e^{ -\sqrt{sh}\mu(\theta_m )} \right.  \\&\left. \hspace{3.5cm}  -  \mu(\theta_m )^{-1} \sqrt{sh}   e^{ -\sqrt{sh}\mu(\theta_m ) }  \right  ), 
%	&+ \left( \mu(\theta_m )^{-2}-   \mu(\theta_m )^{-2} e^{ -\sqrt{sh}} -  \mu(\theta_m )^{-1} \sqrt{sh}   e^{ -\sqrt{sh}} \right  ) \Bigg] . \nonumber	
	\end{split}
\end{align}
	where $ \mu(\theta )=-\cos(\theta/2+\pi) -i \sin( \theta/2+\pi )$. 

\end{lemma}
\begin{proof}
	Using variable substitution $t=\sqrt{sr}$ and by direct calculations, we can derive  \eqref{eq:I311}. 
%\begin{equation}\label{eq:I311}
%	{\mathcal I}_{311}^-=2s^{-1}\left( \mu(\theta_m )^{-2}-   \mu(\theta_m )^{-2} e^{ -\sqrt{sh}\mu(\theta_m )  }-  \mu(\theta_m )^{-1} \sqrt{sh}  e^{ -\sqrt{sh} \mu(\theta_m ) } \right  ).
%\end{equation}
\end{proof}

Using the Jacobi-Anger expansion (cf. \cite[Page 75]{CK}), for the Herglotz wave function $v_j$ given in \eqref{eq:herg}, we have
\begin{equation}\label{eq:vjex}
	v_j(x)= v_j(0) J_0(k |x| )+2 \sum_{p=1}^\infty  \gamma_{pj}   i^p  J_p ( k |x| ), \quad x\in \R^2 , 
	\end{equation}
	where
	\begin{align}\label{eq:239 gamma}
		v_j(0)=  \int_{{\mathbb S}^{1}}  g_j(\theta ) {\rm d} \sigma(\theta ),\quad  
\gamma_{pj}= \int_{{\mathbb S}^{1}}  g_j(\xi  ) \cos (p \varphi ) {\rm d} \sigma(\xi  ), \quad p,\, j \in \mathbb N, 
			\end{align}
	 $J_p(t)$ is the $p$-th Bessel function of the first kind \cite{Abr},  $g_j$ is the kernel of $v_j$ defined in \eqref{eq:herg} and $\varphi$ is given by \eqref{eq:varphi}.

\begin{lemma}\label{lem:28}
Let  $\Gamma_h^\pm$ be defined in \eqref{eq:sh}. 	 Recall that the CGO solution $u_0(sx)$ is defined in \eqref{eq:u0} with the parameter $s\in \mathbb R_+$, and $I_2^\pm$ is defined by \eqref{eq:intnotation}. Denote 
\begin{equation}\label{eq:omegamu}
	 \omega(\theta )=-\cos(\theta/2+\pi) ,\quad  \mu(\theta )=-\cos(\theta/2+\pi) -i \sin( \theta/2+\pi ).
\end{equation}
Recall that the Herglotz wave function  $v_j$ is given in the form \eqref{eq:herg}. Suppose that $\eta ( x) \in C^{\alpha}(\overline{\Gamma_h^\pm } )$ ($0< \alpha<1$) satisfying \eqref{eq:eta} and let
\begin{align}\label{eq:Ieta}
\begin{split}
{\mathcal I}_1^-&= \int_{0}^h \delta \eta\  J_0(kr) e^{-\sqrt{sr} \mu(\theta_m)} {\rm d} r ,\quad  {\mathcal I}_2^- =2\sum_{p=1}^\infty i^p  \int_0^h  \delta \eta\ \gamma_{pj} J_{p}(kr ) e^{-\sqrt{sr} \mu (\theta_m ) } {\rm d} r,\\
	{\mathcal I}_1^+&= \int_{0}^h  \delta \eta\  J_0(kr) e^{-\sqrt{sr} \mu(\theta_M)} {\rm d} r ,\quad  {\mathcal I}_2^+ =2\sum_{p=1}^\infty i^p  \int_0^h  \delta \eta\ \gamma_{pj} J_{p}(kr ) e^{-\sqrt{sr} \mu (\theta_M ) } {\rm d} r,  \\
	I_\eta^-&= v_j(0){\mathcal I}_1^-+ {\mathcal I}_2^- ,\quad I_\eta^+= v_j(0){\mathcal I}_1^-+ {\mathcal I}_2^+. 
	\end{split}
\end{align}
 Assume that for a fixed $k\in 
 \mathbb R_+$, $kh<1$, where $h$ is the length of $\Gamma_h^\pm$,   and $-\pi< \theta_m < \theta_M <\pi $, where $\theta_m$ and $\theta_M$ are defined in \eqref{eq:W}. Then 
\begin{align}\label{eq:I2-final}
	I_2^-&=2\eta(0)v_j(0)s^{-1}\left( \mu(\theta_m )^{-2}-   \mu(\theta_m )^{-2} e^{ -\sqrt{sh}\mu(\theta_m )  } -  \mu(\theta_m )^{-1} \sqrt{sh}   e^{ -\sqrt{sh}\mu(\theta_m )  } \right  ) \nonumber\\
	&\quad +v_j(0)\eta(0) {\mathcal I}_{312}^- +\eta(0){\mathcal I}_{32}^-+I_\eta^-, 
	\end{align}
	where 
	\begin{align}\notag
	&{\mathcal I}_{312}^-=\sum_{p=1}^\infty   \frac{(-1)^p k^{2p}}{4^p(p!)^2} \int_{0}^h r^{2p}   e^{-\sqrt{sr} \mu (\theta_m)} {\rm d} r, \quad  {\mathcal I}_{32}^-=2 \sum_{p=1}^\infty  \int_{0}^h \gamma_{pj} i^p J_p(kr) e^{-\sqrt{sr} \mu (\theta_m)} {\rm d} r, \\
	&|{\mathcal I}_{312}^-|\leq \Oh(s^{-3}),\quad
	|{\mathcal I}_{32}^-|\leq \Oh (\|g_j\|_{L^2 ({\mathbb S} ^{1})} s^{-2}), \nonumber \\
	&|I_\eta^-| \leq  |v_j(0)| |{\mathcal I}_1^-| + |{\mathcal I}_2^-|  ,\quad  \left| {\mathcal I}_{1}^- \right|  \leq \Oh ( \|\eta \|_{C^\alpha }  s^{-1-\alpha }),\quad 
	\left| {\mathcal I}_{2}^- \right|  \leq \Oh (\|\eta \|_{C^\alpha } \|g_j\|_{L^2 ({\mathbb S} ^{1})} s^{-2-\alpha }). \notag
\end{align}
as $s \rightarrow +\infty$.  Similarly, we have
\begin{align}\label{eq:I2+final}
	I_2^+&=2\eta(0)v_j(0)s^{-1}\left( \mu(\theta_M )^{-2}-   \mu(\theta_M )^{-2} e^{ -\sqrt{sh}  \mu(\theta_M ) }-  \mu(\theta_M )^{-1} \sqrt{sh}  e^{ -\sqrt{sh} \mu(\theta_M ) } \right  ) \nonumber\\
	&\quad +v_j(0)\eta(0) {\mathcal I}_{312}^+ +\eta(0){\mathcal I}_{32}^+ +I_\eta^+, 
\end{align}
where
\begin{align*}
	{\mathcal I}_{312}^+&=\sum_{p=1}^\infty   \frac{(-1)^p k^{2p}}{4^p(p!)^2} \int_{0}^h r^{2p}   e^{-\sqrt{sr} \mu (\theta_M)} {\rm d} r, \quad  {\mathcal I}_{32}^+=2 \sum_{p=1}^\infty  \int_{0}^h \gamma_{pj} i^p J_p(kr) e^{-\sqrt{sr} \mu (\theta_M)} {\rm d} r ,\\ 
%I_{\eta}^+&=\int_{\Gamma^+_h } \delta \eta(x)  u_0(sx) v_j (x) {\rm d} \sigma,\\
|{\mathcal I}_{312}^+| &\leq \Oh(s^{-3}),\quad
	|{\mathcal I}_{32}^+| \leq \Oh (\|g_j\|_{L^2 ({\mathbb S} ^{1})} s^{-2}), \nonumber \\
	|I_\eta^+| &\leq   |v_j(0)| |{\mathcal I}_1^+| + |{\mathcal I}_2^+ |  ,\ 
	\left| {\mathcal I}_{1}^+ \right|  \leq \Oh (\|\eta \|_{C^\alpha } s^{-1-\alpha }),\ 
	\left| {\mathcal I}_{2}^+  \right|  \leq \Oh (\|\eta \|_{C^\alpha }\|g_j\|_{L^2 ({\mathbb S} ^{1})} s^{-2-\alpha })
\end{align*}
as $s \rightarrow +\infty$. 
\end{lemma}

\begin{proof}
We first  investigate   the boundary integral $I_2^- $ defined in \eqref{eq:intnotation}.  Recall that $\Gamma_h^\pm$ is defined in \eqref{eq:sh}, for $x\in \Gamma_h^\pm$, the polar coordinates $x=(r\cos \theta, r \sin \theta )$ satisfy  $r \in (0, h)$ and $\theta=\theta_m$ or $\theta=\theta_M$ when $x\in \Gamma_h^-$ or $x\in \Gamma_h^+$, respectively.  Since $\eta \in C^\alpha\left(\overline{\Gamma}_h^\pm \right)$, we know that $\eta$ has the expansion \eqref{eq:eta}.  
%we have the following expansion: 
%\begin{equation}\label{eq:eta}
%	\eta(x)=\eta(0)+\delta \eta(x),\quad |\delta \eta(x)| \leq \|\eta \|_{C^\alpha  } |x|^\alpha .
%\end{equation}
Denote
\begin{align}\label{eq:Ieta proof}
I_{21}^-&=	 \int_{\Gamma_h^- } u_0(sx) v_j(x)  {\rm d} \sigma ,\quad 
I_{\eta}^- =\int_{\Gamma^-_h } \delta \eta(x)  u_0(sx) v_j (x) {\rm d} \sigma. 
\end{align} 
Substituting  \eqref{eq:eta} into the expression of $I_2^-$, we have\begin{align}\label{eq:I2-}
	I_2^-   &=\eta(0) I_{21}^-+ I_\eta^-. 
\end{align}
 %$ \omega(\theta )=-\cos(\theta/2+\pi) -i \sin( \theta/2+\pi ).$

%$$
%{\mathcal I}_1^+ = \int_{0}^h r^\alpha J_0(kr) e^{-\sqrt{sr} \omega(\theta_M)} {\rm d} r ,\quad  {\mathcal I}_2^+ =2\sum_{p=1}^\infty \int_0^h r^\alpha \gamma_{pj} i^p J_{p}(kr ) e^{-\sqrt{sr} \mu (\theta_M ) } {\rm d} r,
%$$

Recall that $\gamma_{pj}$ is defined by \eqref{eq:239 gamma}. Using Cauchy-Scharwz inequality, it is clear that
	 \begin{equation}\label{eq:gampj}
	 	|\gamma_{pj}| \leq \sqrt{2\pi} \|g\|_{L^2(\mathbb S^1)}. 
	 \end{equation}
For the Bessel function $J_p(t)$, we have from \cite{Abr} the following series expression: 
	 \begin{equation}\label{eq:Jp}
	 	J_p(t)= \frac{t^p}{2^p p!}+\frac{t^p}{2^p } \sum_{\ell=1}^\infty  \frac{(-1)^\ell t^{2\ell }}{4^\ell (\ell !)^2  }, \quad  \mbox{ for } p =1,\,2,\ldots, 
	 \end{equation}
which is uniformly and absolutely convergent with respect to  $t \in [0,+\infty)$.

For ${\mathcal I}_1^-$ and ${\mathcal I}_2^-$ defined in \eqref{eq:Ieta}, by substituting \eqref{eq:vjex} into $I_\eta^-$ given by \eqref{eq:Ieta proof}, it is directly verified that $I_\eta^-= v_j(0){\mathcal I}_1^-+ {\mathcal I}_2^-$. 

For $\omega(\theta )$ defined in \eqref{eq:omegamu}, it is easy to see that $\omega(\theta ) >0 $  for $-\pi< \theta_m \leq \theta \leq \theta_M <\pi $. By virtue of \eqref{eq:eta}, we have
\begin{align*}
	|{\mathcal I}_1^-| &\leq \|\eta\|_{C^\alpha }\int_{0}^h r^\alpha |J_0(kr)|  e^{-\sqrt{sr} \omega(\theta_m)} {\rm d} r:=\|\eta\|_{C^\alpha } \left({\mathcal I}_{11}^-+{\mathcal I}_{12}^-\right), 
%	=  \int_{0}^h r^\alpha  e^{-\sqrt{sr} \omega(\theta_m)} {\rm d} r+ \sum_{p=1}^\infty   \frac{(-1)^p k^{2p}}{4^p(p!)^2} \int_{0}^h r^{\alpha+2p}   e^{-\sqrt{sr} \omega(\theta_m)} {\rm d} r\\
%	 &=2 s^{-1-\alpha } \int_{0}^{\sqrt{sh}} t^{2\alpha +1} e^{t \omega(\theta_m) } {\rm d} t+{\mathcal I_{11}},
\end{align*}
where 
$$
{\mathcal I}_{11}^-= \int_{0}^h r^\alpha  e^{-\sqrt{sr} \omega(\theta_m)} {\rm d} r,\quad {\mathcal I_{12}^-}=\sum_{p=1}^\infty   \frac{ k^{2p}}{4^p(p!)^2} \int_{0}^h r^{\alpha+2p}   e^{-\sqrt{sr} \omega(\theta_m)} {\rm d} r.
$$ 
From \eqref{eq:zeta} in Lemma \ref{lem:zeta} and noting that $\omega(\theta_m ) >0 $, we have
\begin{equation*}%\label{eq:I11}
	{\mathcal I}_{11}^-=\Oh(s^{-1-\alpha }) 
\end{equation*}
as $s\rightarrow +\infty$. 
%Using variable substitution $t=\sqrt{sr}$, we can show that
%\begin{equation}
%	{\mathcal I}_{11}=2 s^{-1-\alpha } \int_{0}^{\sqrt{sh}} t^{2\alpha +1} e^{-t \omega(\theta_m) } {\rm d} t=\Oh(s^{-1-\alpha }) 
%\end{equation}
%as $s\rightarrow \infty$ because $\Re(\omega(\theta_m ) )>0 $. 
For ${\mathcal I_{12}^-} $, we have the estimate
\begin{align*}
	|{\mathcal I_{12}^-}| &\leq \sum_{p=1}^\infty   \frac{h^{2p-2} k^{2p}}{4^p(p!)^2} \int_{0}^h r^{\alpha+2}   e^{-\sqrt{sr} \omega(\theta_m)} {\rm d} r=\Oh(s^{-3-\alpha })
\end{align*}
as $s\rightarrow +\infty$, where we suppose that $k h<1$ for sufficiently  small $h$. Therefore, we conclude that
\begin{equation}\label{eq:I1}
	|{\mathcal I}_1^-| \leq \Oh( \|\eta\|_{C^\alpha } s^{-1-\alpha })\ \ \mbox{as}\ \ s\rightarrow+\infty. 
\end{equation}

% Denote
%$$
%{\mathcal I}_2^-=2\sum_{p=1}^\infty \int_0^h r^\alpha |\gamma_{pj}|  | J_{p}(kr )| e^{-\sqrt{sr} \omega (\theta_m ) } {\rm d} r.
%$$ 
For sufficiently small $h>0$ fulfilling that $kh<1$,  using \eqref{eq:eta},  \eqref{eq:gampj} and \eqref{eq:Jp},  we have
\begin{align}
	|{\mathcal I}_2^-| & \leq 2 \sqrt{2\pi }\|\eta\|_{C^\alpha}  \|g_j\|_{L^2 ({\mathbb S}^{1})}  \sum_{p=1}^\infty \left[ \frac{ k^p}{2^p p!}\int_0^h r^{p+\alpha}  e^{-\sqrt{sr} \omega(\theta_m ) } {\rm d} r \right. \nonumber \\
	&\left. \hspace{3.5cm} +\frac{k ^p }{2^p } \sum_{\ell=1}^\infty  \frac{k^{2\ell }h^{2(\ell-1)  }}{4^\ell (\ell !)^2  }\left  (\int_0^h r^{p+\alpha+2 }  e^{-\sqrt{sr} \omega(\theta_m ) } {\rm d} r \right ) \right] \nonumber \\
	& \leq 2\sqrt{2\pi } \|\eta\|_{C^\alpha} \|g_j\|_{L^2 ({\mathbb S}^{1})}\sum_{p=1}^\infty \left[ \frac{ k^p  }{2^p p!}\int_0^h r^{p+\alpha}  e^{-\sqrt{sr} \omega(\theta_m ) } {\rm d} r \right. \nonumber \\
	&\left. \hspace{3.5cm}+\frac{(k h)^p}{2^p } \sum_{\ell=1}^\infty  \frac{k^{2\ell }h^{2(\ell-1)  }}{4^\ell (\ell !)^2  }\left  (\int_0^h r^{\alpha+2 }  e^{-\sqrt{sr} \omega(\theta_m ) } {\rm d} r \right ) \right]    \nonumber \\
	& \leq 2\sqrt{2\pi }\|\eta\|_{C^\alpha} \|g_j\|_{L^2 ({\mathbb S}^{1})}\sum_{p=1}^\infty \left[  \frac{ k^p h^{p-1}  }{2^p p!}\int_0^h r^{\alpha+1 } e^{-\sqrt{sr} \omega(\theta_m ) } {\rm d} r+\Oh\left  (s^{-\alpha-3} \right ) \right]    . \nonumber
\end{align}
Using Lemma \ref{lem:zeta}, we know that
$$
\int_0^h r^{\alpha+1} e^{-\sqrt{sr} \omega(\theta_m ) } {\rm d} r=\Oh(s^{-\alpha-2})\ \ \mbox{as}\ \ s\rightarrow +\infty. 
$$
Therefore we can derive that
\begin{equation}\label{eq:I2}
	|{\mathcal I}_2^-| \leq  \Oh (\|\eta\|_{C^\alpha}  \|g_j\|_{L^2 ({\mathbb S}^{1})} s^{-\alpha -2})\ \ \mbox{as}\ \ s\rightarrow +\infty. 
\end{equation}

We proceed to investigate $I_{21}^-$ given by\eqref{eq:I2-}. Denote
	 \begin{equation}
	 	{\mathcal I}_{31}^-= \int_0^h J_0(k r) e^{-\sqrt{sr}\mu(\theta_m) }\rmd r , \quad  {\mathcal I}_{32}^-=2 \sum_{p=1}^\infty  \int_{0}^h \gamma_{pj} i^p J_p(kr) e^{-\sqrt{sr} \mu (\theta_m)} {\rm d} r.
	 \end{equation}
	 Substituting the expansion \eqref{eq:vjex} into $I_{21}^-$,  it is easy to see that
\begin{align*}
	I_{21}^-%&= \int_{0}^h \left(  v_j(0)  J_0(kr)+2 \sum_{p=1}^\infty \gamma_{pj} i^p J_p(kr) \right ) e^{-\sqrt{sr} \mu (\theta_m)} {\rm d} r\\
	&=v_j(0){\mathcal I}_{31}^-+ {\mathcal I}_{32}^-. 
\end{align*}

From \eqref{eq:Jp}, we know that
\begin{equation}\label{eq:J0p}
	J_0(t)=\sum_{p=0}^\infty (-1)^p \frac{t^{2p}}{4^p(p!)^2}.
\end{equation} 
Let
$$
{\mathcal I}_{311}^-= \int_{0}^h   e^{-\sqrt{sr} \mu(\theta_m)} {\rm d} r,\quad {\mathcal I_{312}}^-=\sum_{p=1}^\infty   \frac{(-1)^p k^{2p}}{4^p(p!)^2} \int_{0}^h r^{2p}   e^{-\sqrt{sr} \mu(\theta_m)} {\rm d} r.
$$
Substituting the expansion \eqref{eq:J0p} of $J_0$ into ${\mathcal I}_{31}^-$, we have
\begin{align*}
	{\mathcal I}_{31}^-={\mathcal I}_{311}^-+{\mathcal I}_{312}^-,
\end{align*}
where ${\mathcal I}_{311}^-$ can be derived directly from Lemma \ref{lem:u0 int}.

For ${\mathcal I_{312}^-}$ , we have
 \begin{align}\label{eq:241}
 	\left| {\mathcal I_{312}^-}\right| \leq \sum_{p=1}^\infty   \frac{  k^{2p} h^{2p-2}}{4^p(p!)^2} \int_{0}^h r^{2}   e^{-\sqrt{sr} \omega(\theta_m)} {\rm d} r=\Oh(s^{-3})\ \ \mbox{as}\ \ s\rightarrow +\infty. 
 \end{align}

 Substituting the expansion \eqref{eq:Jp} of $J_p$ into ${\mathcal I_{32}^-} $, using \eqref{eq:gampj}  we can deduce that
 \begin{align}\label{eq:I32}
 	|{\mathcal I}_{32}^-|&\leq 2\sqrt{2\pi } \|g_j\|_{L^2 ({\mathbb S} ^{1})}  \sum_{p=1}^\infty \left [ \frac{k^p}{2^p p!} \int_0^h r^p e^{-\sqrt{sr} \omega(\theta_m) } {\rm d} r \right. \nonumber \\
	&\left. \hspace{3.5cm}+\frac{(k h)^p }{2^p } \sum_{\ell=1}^\infty  \frac{k^{2\ell }h^{2(\ell-1)  }}{4^\ell (\ell !)^2  } \int_0^h r^{2 }  e^{-\sqrt{sr} \omega(\theta_m ) } {\rm d} r   \right]\nonumber \\ 
 	&\leq 2\sqrt{2\pi } \|g_j\|_{L^2 ({\mathbb S} ^{1})}  \sum_{p=1}^\infty \left [ \frac{k^ph^{p-1}}{2^p p!} \int_0^h r e^{-\sqrt{sr} \omega(\theta_m) } {\rm d} r+\Oh\left(s^{-3}\right) \right] \nonumber \\
 	&=\Oh (\|g_j\|_{L^2 ({\mathbb S} ^{1})} s^{-2}),
 \end{align}
 where we suppose that $k h<1$ for sufficiently small $h$.

Finally, substituting  \eqref{eq:I1}, \eqref{eq:I2}, \eqref{eq:I311}, \eqref{eq:241} and \eqref{eq:I32} into \eqref{eq:I2-}, we can obtain the integral equality \eqref{eq:I2-final}.

Following a completely similar argument in deriving the integral equality \eqref{eq:I2-final} of $I_2^-$, we can derive the integral equality \eqref{eq:I2+final} for $I_2^+$ as well.

The proof is complete.  	
\end{proof}

\begin{lemma}\label{lem:29}
Let $\mu(\theta )$ be defined in  \eqref{eq:omegamu}.  	Assume that 
	\begin{equation}\label{eq:lem 29 cond}
		-\pi < \theta_m < \theta_M <\pi ,
	\end{equation}
	then 
\begin{equation}\label{eq:mutheta}
	\mu(\theta_m)^{-2}+\mu(\theta_M)^{-2} \neq 0.
\end{equation}	
\end{lemma}
\begin{proof}
	It can be calculated that
\begin{align*}
	\mu(\theta_m)^{-2}+\mu(\theta_M)^{-2}=\frac{(\cos \theta_m+\cos \theta_M)+i (\sin \theta_m+\sin \theta_M )}{ (\cos \theta_m+i \sin \theta_m )(\cos \theta_M+i \sin \theta_M )}.
\end{align*}
Then under the assumption \eqref{eq:lem 29 cond}, it is straightforward to verify that
$$
\cos \theta_m+\cos \theta_M \mbox{ and } \sin \theta_m+\sin \theta_M
$$
can not be zero simultaneously,  which immediately implies \eqref{eq:mutheta}. 
\end{proof}

}

We are in a position to present one of the main theorems in this section. 
%{\color{blue} Before that, denote $W_{x_c}(\theta_W)$ is a open sector in $\mathbb R^2$ with the vertex $x_c$ and the open angle $\theta_W $.  } 
%$\Lambda_h \cup \Lambda_{h/2} \cup \Gamma^+_{(h/2,h)} \cup \Gamma^-_{(h/2,h)}$, where $  \Gamma^{\pm }_{(h/2,h)}=( \partial  S_h \cap B_h )\backslash ( \partial  S_h \cap B_{h/2}) $. } 

\begin{theorem}\label{Th:1.1}
	Let $v\in H^1(\Omega ) $ and $w \in H^1(\Omega )  $ be a pair of eigenfunctions to \eqref{eq:in eig} associated with $k\in\mathbb{R}_+$. Assume that the Lipschitz domain $\Omega\subset\mathbb{R}^2$ contains a corner {\bl$\Omega\cap B_h= \Omega \cap W$ with the vertex being $0\in \partial \Omega$, where $W$ is the sector defined in \eqref{eq:W}  and $h\in \mathbb R_+$}. Moreover,  there exits a sufficiently small  neighbourhood {\bl $S_h=  \Omega\cap B_h= \Omega \cap W$ of $0$}, 
%	(i.e. $h>0$ is sufficiently  small) of $x_c$ in $\Omega$, where
%	{\bl 
%	\begin{equation}\label{eq:thm21}
%	\begin{split}
%		\Gamma_h^\pm(x_c)&: =\partial W_{x_c} (\theta_W )  \cap B_h(x_c),\quad \Sigma_{\Lambda_h}(x_c):=S_{h}(x_c)\backslash S_{h/2} (x_c), \\
%		 S_{h/2} (x_c)&:=  \Omega\cap B_{h/2} (x_c)= \Omega \cap W_{x_c} (\theta_W ),
%	\end{split}
%	\end{equation}
such that   $q w \in C^\alpha(\overline {S _h} ) $  with $q:=1+V$ and $\eta \in C^\alpha\left(\overline{\Gamma_h^\pm } \right)$ for $0< \alpha  <1$,  and $ v-w \in H^2(\Sigma_{\Lambda_h})$, where $S_h$, $\Gamma_h^\pm$ and $\Sigma_{\Lambda_h}$ are defined in \eqref{eq:sh}.     If the following conditions are fulfilled:
	\begin{itemize}
		\item[(a)] the transmission eigenfunction $v$ can be approximated in $H^1(S_h)$ by the Herglotz functions $v_j$, $j=1,2,\ldots$,  with kernels $g_j$ satisfying the approximation property \eqref{eq:ass1};
%	\begin{equation}\label{eq:ass1thm}
%		\|v-v_j\|_{H^1(S_h(x_c))} \leq j^{-1-\Upsilon },\quad  \|g_j\|_{L^2({\mathbb S}^{1})} \leq C j^{\varrho},
%	\end{equation}
%	for some constants $C>0$, $\Upsilon >0$ and $0< \varrho<1 $;
	\item[(b)] the function $\eta(x)$ doest not vanish at the vertex $0$, {\bl where $0$ is the vertex of $S_h$,}
%	i.e., $\eta(x_c) \neq 0; $S
	 i.e., 
	\begin{equation}\label{eq:ass2}
		\eta(0) \neq 0;
	\end{equation}
	\item[(c)]  {\bl the open angle of 
	% $\theta_m$ and $\theta_M$ of 
	 $S_h$  satisfies}
	\begin{equation}\label{eq:ass3}
	%\theta_W \neq \pi;
	-\pi < \theta_m < \theta_M < \pi \mbox{ and } \theta_M-\theta_m \neq \pi;
	%\cos \theta_m+\cos \theta_M \mbox{ and } \sin \theta_m+\sin \theta_M
\end{equation}
%can not be zero simultaneously, 
	\end{itemize}
	 then one has {\bl
	%\begin{equation}\label{eq:vn1}
	\begin{equation}\label{eq:nnv1}
	\lim_{ \rho \rightarrow +0 }\frac{1}{m(B(0, \rho  ) \cap \Omega )} \int_{B(0, \rho ) \cap \Omega } |v(x)|  {\rm d} x=0,
	\end{equation}
	%\end{equation}
	where $m(B(0, \rho  ) \cap \Omega )$ is the area of $B(0,\rho )\cap \Omega $.} 
\end{theorem}

\begin{remark}\label{rem:th1.1}
 In Theorem~\ref{Th:1.1}, we consider the case that $v, w$ are a pair of conductive transmission eigenfunctions to \eqref{eq:in eig} and show the vanishing property near a corner. We would like to emphasize that the result can be localized in the sense that as long as $v, w$ satisfy all the conditions stated in Theorem~\ref{Th:1.1} in $\Omega\cap S_h$, then one has the vanishing property \eqref{eq:nnv1} near the corner. That is, $v, w$ are not necessary conductive transmission eigenfunctions, and it suffices to require that $v, w$ satisfy the equations in \eqref{eq:in eig} in $S_h\cap\Omega$ and the conductive transmission conditions on $\overline{S}_h \cap\partial\Omega$, then one has the same vanishing property as stated in Theorem~\ref{Th:1.1}. Indeed, the subsequent proof of Theorem~\ref{Th:1.1} is for the aforementioned localized problem. 
\end{remark}

\begin{remark}\label{rem:hn1}
The condition \eqref{eq:ass1} signifies a certain regularity condition of the transmission eigenfunction $v\in H^1(\Omega)$. In \cite{BL2017b}, the following regularity condition was introduced,
\begin{equation}\label{eq:vjgj}
	\|v-v_j\|_{L^2( \Omega )} \leq e^{-j},\quad \|g_j\|_{L^2( {\mathbb S}^{n-1} )} \leq C (\ln j)^\beta, 
\end{equation}
where the constants $C>0$ and $0<\beta < 1/(2n+8), (n=2, 3)$. Here, we allow the polynomial growth of the kernel functions. Moreover, we would like to remark that $qw\in C^\alpha(\overline S_h )$ is technically required in our mathematical argument of proving Theorem~\ref{Th:1.1}. {\bl This technical condition can be fulfilled in our study for the unique recovery result  of the inverse scattering problem. Indeed, when $q$ is a constant, it is shown in Lemma \ref{lem41} that $qw\in C^\alpha(\overline S_h)$.} The interior regularity requirement $v-w \in H^2(\Sigma_{\Lambda_h})$ can be fulfilled in certain practical scenarios; see Theorem~\ref{Th:4.1} in what follows on the study of an inverse scattering problem. The introduction of this interior regularity condition shall play a critical role in the proof of Theorem~\ref{Th:4.1}. 
\end{remark}

%Clearly, since $qW $ is  $H^2$ in a neighborhood of $0$,  then it embeds into $C^{\alpha }$ where $0<\alpha \leq 1/2$. So we know that  $f_2$ is Holder-continuous.

%\begin{remark}\label{rem:25}
%	{\bl  Since the partial differential operator $\Delta+k^2$ is invariant under rigid motions, we may assume without loss of generality that $x_c$ is the origin. Therefore $S_h(x_c)$, $\Gamma_h^\pm (x_c)$ and  $\Sigma_{\Lambda_h }(x_c)$ defined in \eqref{eq:thm21} can coincide with the corresponding $S_h$, $\Gamma_h^\pm$ and $\Sigma_{\Lambda_h }$ given by \eqref{eq:sh} after  suitably  rotating,  and the angle  $\theta_m$ and $\theta_M$ of $S_h$ fulfill
%\begin{equation}\label{eq:ass 260}
%	 \theta_W=\theta_M-\theta_m \neq \pi \mbox{ and }-\pi < \theta_m < \theta_M <\pi ,
%\end{equation}
% where $\theta_W$ is the open angle of $W$. 
% 
% 
% It is clear that $v \in H^1(S_h)$ and $w \in H^1(S_h)$ satisfy the PDE system \eqref{eq:in eignew}. Furthermore, from the assumptions in this theorem, we have
% $$
% qw\in C^\alpha(\overline S_h ), \quad v-w \in H^2 (\Sigma_{\Lambda_h } ),\quad  \eta \in C^\alpha(\overline \Gamma_h^\pm ),\quad \eta(0) \neq 0
% $$
% and $v_j$ given by \eqref{eq:herg} can approximate $v$ with the condition \eqref{eq:ass1} by virtue of \eqref{eq:ass1thm}.}
%\end{remark}

\begin{proof}[Proof of Theorem~\ref{Th:1.1}]

{\bl It is clear that the transmission eigenfunctions $v \in H^1(\Omega )$ and $w \in H^1(\Omega )$ to \eqref{eq:in eig} fulfill \eqref{eq:in eignew}.   From Lemma \ref{lem:int1}, we know that \eqref{eq:intimportant} holds.} Using the fact that
 \begin{equation}\label{eq:int u0 sh}
\int_{S_h} u_0(sx) {\rm d} x=\int_{W} u_0(sx){\rm d} x- \int_{ W\backslash S_h  } u_0(sx)  {\rm d} x,
 \end{equation}
{\bl where $u_0(sx)$ is defined in \eqref{eq:u0}, substituting \eqref{eq:int u0 sh} into \eqref{eq:intimportant},}  we obtain the following integral equation 
\begin{align*}
	& ( \widetilde f_{1j} (0)-f_2(0)) \int_{ W  } u_0(sx) {\rm d} x + \delta_j(s)  = I_3-I_2^\pm -\int_{S_h} \delta\widetilde f_{1j}(x)  u_0(sx) {\rm d} x\\
	& \quad \quad \quad +\int_{S_h}  \delta f_2(x) u_0(sx) {\rm d} x +( \widetilde f_{1j} (0)-f_2(0)) \int_{ W \backslash S_h  } u_0(sx) {\rm d} x- \xi_j^\pm (s),
\end{align*}
{\bl where $\delta  \widetilde f_{1j} (x)$ and $\delta f_2(x) $ are defined in \eqref{eq:f1jf2 notation new} and \eqref{eq:f1jf2 notation}, respectively, $\delta_j(s)$ and $\xi_j^\pm (s)$ is given by \eqref{eq:intnotation}.}

From \eqref{eq:u0w}, we know that
\begin{equation}\label{eq:f1jf2}
	( \widetilde f_{1j} (0)-f_2(0)) \int_{ W  } u_0(sx) {\rm d} x =6 i (\widetilde f_{1j} (0)-f_2(0) ) (e^{-2\theta_M i }-e^{-2\theta_m i }  )  s^{-2}. 
\end{equation}

{\bl From Lemma \ref{lem:28} it yields  \eqref{eq:I2-final} and \eqref{eq:I2+final}.} Substituting \eqref{eq:I2-final} and \eqref{eq:I2+final} into \eqref{eq:intimportant}, multiplying  $s$ on the both sides of \eqref{eq:intimportant}, and rearranging terms, we deduce that
\begin{align}\label{eq:45}
	& 2v_j(0)\eta(0)\Bigg[ \left( \mu(\theta_M )^{-2}-   \mu(\theta_M )^{-2} e^{ -\sqrt{sh}  \mu(\theta_M ) } -  \mu(\theta_M )^{-1} \sqrt{sh}   e^{ -\sqrt{sh}  \mu(\theta_M ) }  \right  ) \nonumber\\
	&\qquad\qquad+ \left( \mu(\theta_m )^{-2}-   \mu(\theta_m )^{-2} e^{ -\sqrt{sh}  \mu(\theta_m ) } -  \mu(\theta_m )^{-1} \sqrt{sh}   e^{ -\sqrt{sh} \mu(\theta_m )} \right  ) \Bigg] \nonumber \\
	=& s\Bigg[ I_3-( \widetilde f_{1j} (0)-f_2(0)) \int_{S_h} u_0(sx) {\rm d} x- \delta_j(s)- v_j(0) \eta(0)\left({\mathcal I}_{312}^-+{\mathcal I}_{312}^+ \right) \nonumber \\
	&-\eta(0)({\mathcal I}_{32}^++ {\mathcal I}_{32}^- ) -I_\eta^+ - I_\eta^- -\int_{S_h}  \delta \widetilde f_{1j} (x) u_0(sx) {\rm d} x +\int_{S_h}  \delta f_2(x) u_0(sx) {\rm d} x -\xi_j^\pm(s) \Bigg]. 
\end{align}
{\bl Taking} $s=j$, under the assumption \eqref{eq:ass1}, using \eqref{eq:I2-final} and \eqref{eq:I2+final} in Lemma \ref{lem:28},   we know that
\begin{align}\label{eq:46}
&j | {\mathcal I}_{32}^-|  \leq \Oh( j^{-1} \|g_j\|_{L^2({\mathbb S} ^{1} )} )\leq \Oh(j^{-1+\varrho }),\quad j | {\mathcal I}_{32}^+ |  \leq \Oh( j^{-1} \|g_j\|_{L^2({\mathbb S} ^{1} )} ) \leq \Oh(j^{-1+\varrho }), \nonumber \\
&j |I_\eta^-| \leq  \|\eta \|_{C^\alpha } \left( |v_j(0)| \Oh (j^{-\alpha })+ \Oh (\|g_j\|_{L^2({\mathbb S} ^{1})} j^{-1-\alpha } ) \right),\nonumber \\
&\quad \leq  \|\eta \|_{C^\alpha } \left( |v_j(0)| \Oh (j^{-\alpha })+ \Oh ( j^{-1-\alpha +\varrho } ) \right),\nonumber\\
& j|I_\eta^+|  \leq  \|\eta \|_{C^\alpha } \left( |v_j(0)| \Oh (j^{-\alpha })+ \Oh (\|g_j\|_{L^2({\mathbb S} ^{1})} j^{-1-\alpha } ) \right),\nonumber \\
&\quad \leq  \|\eta \|_{C^\alpha } \left( |v_j(0)| \Oh (j^{-\alpha })+ \Oh ( j^{-1-\alpha +\varrho } ) \right),\nonumber\\
& j {\mathcal I}_{312}^- \leq \Oh(j^{-2}),\quad j {\mathcal I}_{312}^+ \leq \Oh(j^{-2}),
\end{align}
Clearly, when $s=j$, {\bl under the assumption \eqref{eq:ass1}, by virtue of \eqref{eq:deltaf1j},  \eqref{eq:deltaf2} and \eqref{eq:I3} in Lemma \ref{lem:int1}, we can obtain
\begin{equation}
	\begin{split}
		&
	j \left|\int_{S_h}  \delta \widetilde f_{1j} (x) u_0(j x) {\rm d} x \right | \leq \frac{2\sqrt{2\pi}(\theta_M- \theta_m) \Gamma(2 \alpha+4) }{ \delta_W^{2\alpha+4   }  } k^2 {\rm diam}({S_h})^{1-\alpha }   \\
	&\hspace{4.5cm} \times  (1+k) \|g_j\|_{L^2( {\mathbb S}^{1})}   j^{-\alpha-1   } \leq \Oh(j^{-1-\alpha +\varrho  }), \\
	& j \left| \int_{S_h}  \delta f_2(x) u_0(j x) {\rm d} x\right|  \leq \frac{2(\theta_M- \theta_m) \Gamma(2 \alpha+4) )}{ \delta_W^{2\alpha+4   }  } \|f_2\|_{C^\alpha }  j^{-\alpha-1   }, \  j |I_3| \leq  C j e^{-c' \sqrt j}, 
	\end{split}
\end{equation}
where $C,c'>0$ and $\delta_W$ is defined in \eqref{eq:xalpha}. Similarly,  under the assumption \eqref{eq:ass1}, when $s=j$, in view of \eqref{eq:deltajnew} and \eqref{eq:xij} in Lemma \ref{lem:27},    it can be derived that
\begin{align}
	&j |\xi_j^\pm (j) |\leq C \left( |\eta(0)|  \frac{\sqrt { \theta_M-\theta_m }  e^{-\sqrt{j \Theta } \delta_W } h } {\sqrt 2 } j  + \|\eta\|_{C^\alpha } j^{-\alpha } \frac{\sqrt{2(\theta_M-\theta_m) \Gamma(4\alpha+4) } }{(2\delta_W)^{2\alpha+2  } } \right) j^{-1-\Upsilon } , \nonumber \\
	&j|\delta_j(j)| \leq    \frac{\sqrt { \theta_M-\theta_m } k^2 e^{-\sqrt{j \Theta } \delta_W } h } {\sqrt 2 } j^{-\Upsilon }  , \quad  \Theta  \in [0,h ].
\end{align}
Furthermore, taking $s=j$ and using \eqref{eq:int u0 sh}, from \eqref{eq:u0w} and \eqref{eq:1.5}, we can deduce that
\begin{equation}\label{eq:47}
		j \left| \int_{S_h} u_0(jx) {\rm d} x \right| \leq 6 |e^{-2\theta_M i }-e^{-2\theta_m i }  | j^{-1} +  \frac{6(\theta_M-\theta_m )}{\delta_W^4} j^{-1} e^{-\delta_W \sqrt{h j}/2}. 
\end{equation}
}

The coefficient of $v_j(0)$ of \eqref{eq:45} with respect to the zeroth order of $s$ is 
$$
2 \eta(0)\left  (\mu(\theta_m)^{-2}+\mu(\theta_M)^{-2} \right ). 
$$
Under  the assumption \eqref{eq:ass3}, from Lemma \ref{lem:29}, we have
\begin{equation}\label{eq:umneq0}
	\mu(\theta_m)^{-2}+\mu(\theta_M)^{-2}\neq 0
\end{equation}

%Substituting \eqref{eq:46} and \eqref{eq:47} into \eqref{eq:45} together with  $s=j$, and using the assumption \eqref{eq:ass1} and  $\eta(0)\neq 0$,  from \eqref{eq:umneq0}, by letting  $j\rightarrow \infty$, we prove that

 We take $s=j$ in \eqref{eq:45}.  By letting  $j\rightarrow \infty$ in \eqref{eq:45}, from \eqref{eq:46} and \eqref{eq:47}, we can prove that
 $$
\eta(0)\left(\mu(\theta_m)^{-2}+\mu(\theta_M)^{-2} \right) \lim_{j \rightarrow \infty} v_j(0)=0. 
$$
Since $\eta(0)\neq 0$  and using \eqref{eq:umneq0}, it is easy to see that 
$$
\lim_{j \rightarrow \infty} v_j(0)=0. 
$$
Using the fact that
\begin{align}
	\lim_{ \rho \rightarrow +0 } & \frac{1}{m(B(0, \rho  )\cap \Omega )} \int_{B(0, \rho )\cap \Omega } |v(x)| {\rm d} x \leq \lim_{j \rightarrow \infty}  \Big( \lim_{ \rho \rightarrow +0 }\frac{1}{m(B(0, \rho  )\cap \Omega )} \nonumber \\
	&\times \int_{B(0, \rho ) \cap \Omega } |v(x)-v_j(x)| {\rm d} x +\lim_{ \rho \rightarrow +0 }\frac{1}{m(B(0, \rho  )\cap \Omega )} \int_{B(0, \rho )\cap \Omega} |v_j(x)| {\rm d} x\Big),  \label{eq:250}
\end{align}
we readily finish the proof of this theorem. 
\end{proof}

%{\color{blue}
%
%\begin{remark}
%	The technical  assumption $v-w \in H^2( \Sigma_{\Lambda_h} ) $ of  Theorem \ref{Th:1.1} can be fulfilled when we investigate the unique recovery in the inverse conductive  scattering problem in Section \ref{sec:4}; see 	Theorem \ref{Th:4.1} for more details. 
%\end{remark}
%} 

We next consider the degenerate case of Theorem~\ref{Th:1.1} with $\eta\equiv 0$. The conductive transmission eigenvalue problem \eqref{eq:in eig} is reduced to the following interior transmission eigenvalue problem
\begin{align}\label{eq:in eig reduce}
%$$%\lb{1div}
\left\{
\begin{array}{l}
\Delta w+k^2(1+V) w=0 \quad\ \mbox{ in } \Omega, \\[5pt] 
\Delta v+ k^2  v=0\hspace*{1.85cm}\ \mbox{ in } \Omega, \\[5pt] 
w= v,\quad \partial_\nu v=\partial_\nu w \hspace*{.9cm} \mbox{ on } \partial \Omega,
  \end{array}
\right.
%$$
\end{align}
By slightly modifying our proof of Theorem~\ref{Th:1.1}, we can show the following result. 
%
%In the next corollary, we will consider the case $ \eta(x) \equiv 0 $ near the corner, which means that \eqref{eq:in eig} degenerates to \eqref{eq:in eig reduce} near the corner. The interior transmission eigenfunctions $v \in H^1(\Omega )$ and $w\in H^1(\Omega ) $ to \eqref{eq:in eig reduce} still have the vanishing property near the corner provided that $v$ can be approximated by  the Herglotz waves $v_j$ under the assumption \eqref{eq:ass1 int} and  the interior angle  of the sector $W$ containing the corner is not $\pi$. 
\begin{corollary}\label{cor:2.1}
{\bl Let $\Omega\Subset \mathbb R^2$ be a bounded Lipschitz domain containing a corner  $\Omega\cap B_h= \Omega \cap W$  with the vertex being $0\in \partial \Omega$, where $W$ is the sector defined in \eqref{eq:W}  and $h\in \mathbb R_+$}.  Suppose  $v \in H^1(\Omega )$ and $w\in H^1(\Omega ) $ are a pair of interior transmission eigenfunctions to \eqref{eq:in eig reduce}. Let $W$ and $S_h$ be the same as described in Theorem~\ref{Th:1.1}. Assume that { $ v-w \in H^2(\Sigma_{\Lambda_h} ) $ }  and  $q w \in C^\alpha(\overline {S}_h ) $  for $0< \alpha  <1$. Under the conditions \eqref{eq:ass3} and that the transmission eigenfunction $v$ can be approximated in $H^1(S_h)$ by the Herglotz functions $v_j$, $j=1,2,\ldots$,  with kernels $g_j$ satisfying
	\begin{equation}\label{eq:ass1 int}
		\|v-v_j\|_{H^1(S_h)} \leq j^{-2-\Upsilon},\quad  \|g_j\|_{L^2({\mathbb S}^{1})} \leq C j^{\varrho}, 
	\end{equation}
	for some constants $C>0$, $\Upsilon >0$ and $0< \varrho<\alpha $, one has
	{\bl 
	\[
	%\begin{equation}\label{eq:vn2}
	\lim_{ \rho \rightarrow +0 }\frac{1}{m(B(0, \rho  )\cap \Omega )} \int_{B(0, \rho ) \cap \Omega }V(x) w(x)  {\rm d} x=0. 
	%\end{equation} 
	\]
	}
\end{corollary}

\begin{remark}\label{rem:hn2}
As discussed in the introduction, the vanishing near a corner of the interior transmission eigenfunctions was considered in \cite{BL2017b}. Compared to the main result in \cite{BL2017b}, Corollary~\ref{cor:2.1} is more general in two aspects. First, the corner in \cite{BL2017b} must be a convex one, whereas in Corollary~\ref{cor:2.1}, the corner could be an arbitrary one as long as the corner is not degenerate, namely \eqref{eq:ass3} is fulfilled. Second, the regularity requirement on the eigenfunction $v$ is relaxed from \eqref{eq:vjgj} to \eqref{eq:ass1 int}. Moreover, technical condition $qw\in C^\alpha(\overline{ S}_h)$ in Corollary~\ref{cor:2.1} can be readily fulfilled when we consider the unique recovery of the inverse scattering problem under the condition that $q$ is a constant. Please refer to Lemma \ref{lem41}.
%However, we also need to impose a new technical condition by requiring that $qw\in C^\alpha(\overline{ S}_h (x_c))$ in Corollary~\ref{cor:2.1}. 
\end{remark}

\begin{proof}[Proof of Corollary~\ref{cor:2.1}] The proof follows from the one for Theorem~\ref{Th:1.1} with some necessary modifications, and we only outline it in the following. {\bl It is clear that the transmission eigenfunctions $v \in H^1(\Omega )$ and $w \in H^1(\Omega )$ to \eqref{eq:in eig reduce} fulfill \eqref{eq:in eignew}   for $\eta \equiv 0$.}   Since  $ \eta(x) \equiv 0 $ near the corner,  similar to \eqref{eq:intimportant} in Lemma \ref{lem:int1}, we have the following integral identity, 
\begin{align}\label{eq:intimportant int}
	( \widetilde f_{1j} (0)-f_2(0)) \int_{S_h} u_0(sx) {\rm d} x+\delta_j(s)  &= I_3 -\int_{S_h}  \delta \widetilde f_{1j} (x) u_0(sx) {\rm d} x +\int_{S_h}  \delta f_2(x) u_0(sx) {\rm d} x  ,
\end{align}
where $f_2(x)$, $\widetilde f_{1j} (x)$, $\delta_j(s)$, $I_3$, $ \delta \widetilde f_{1j} (x)$ and $ \delta f_2(x)$ are defined in \eqref{eq:vw},  \eqref{eq:deltajs}, \eqref{eq:intnotation} and  \eqref{eq:f1jf2 notation}, respectively.  

From \eqref{eq:u0w}, it follows that 
\begin{align}\label{eq:258sh}
	( \widetilde f_{1j} (0)-f_2(0)) \int_{S_h} u_0(sx) {\rm d} x&=( \widetilde f_{1j} (0)-f_2(0)) \int_{W} u_0(sx) {\rm d} x \\
	&\quad  -( \widetilde f_{1j} (0)-f_2(0)) \int_{W \backslash S_h} u_0(sx) {\rm d} x \notag \\
	&= 6 i (\widetilde f_{1j} (0)-f_2(0) ) (e^{-2\theta_M i }-e^{-2\theta_m i }  )  s^{-2} \notag\\
	 &\quad  -( \widetilde f_{1j} (0)-f_2(0)) \int_{W \backslash S_h} u_0(sx) {\rm d} x \notag. 
\end{align}

%{\bl Since \eqref{eq:ass1 int} implies \eqref{eq:ass1} after rotating and translating $S_h(x_c)$ to be $S_h$,} 
{\bl Combining \eqref{eq:ass1 int} with \eqref{eq:u0L2}} in Lemma \ref{lem:23}, one can see that
\begin{equation}\label{eq:deltajnew int page14}
	j^2 |\delta_j(s)| \leq    \frac{\sqrt { \theta_M-\theta_m } k^2 e^{-\sqrt{s \Theta } \delta_W } h } {\sqrt 2 }  j^{-\Upsilon }, 
\end{equation} 
where  $ \Theta  \in [0,h ]$ and $\delta_W$ is defined in \eqref{eq:xalpha}. By \eqref{eq:deltaf1j} {\bl in Lemma \ref{lem:int1}}, we can also deduce that
\begin{align}\label{eq:deltaf1j2}
& j^2 \left|\int_{S_h}  \delta \widetilde f_{1j} (x) u_0(j x) {\rm d} x \right |  \leq  \frac{2\sqrt{2\pi}(\theta_M- \theta_m) \Gamma(2 \alpha+4) }{ \delta_W^{2\alpha+4   }  } k^2 {\rm diam}({S_h})^{1-\alpha }  \nonumber \\
&\hspace{4.5cm} \times  (1+k) \|g_j\|_{L^2( {\mathbb S}^{1})}   j^{-\alpha   } \leq \Oh(j^{-(\alpha -\varrho ) }),
\end{align}
for $0< \varrho<\alpha$. After substituting \eqref{eq:258sh} into \eqref{eq:intimportant int}, we take $s=j$.  Since \eqref{eq:258sh}, multiply $j^2$ on the both sides of \eqref{eq:intimportant int}. Using the  assumptions \eqref{eq:ass1 int} and \eqref{eq:ass3}, by letting  $j \rightarrow \infty$, from \eqref{eq:1.5} {\bl in Lemma \ref{lem:1},   \eqref{eq:deltaf1j}, \eqref{eq:deltaf2} and \eqref{eq:I3} in Lemma \ref{lem:int1},}   and \eqref{eq:deltajnew int page14}, we prove that
$$
\lim_{j \rightarrow \infty} v_j(0) = \frac{f_2(0)}{-k^2}. 
$$ 
Since
\begin{align*}
	\lim_{j \rightarrow \infty}  v_j(0)&=\lim_{j \rightarrow \infty}  \lim_{ \rho \rightarrow +0 }\frac{1}{m(B(0, \rho  )\cap \Omega )} \int_{B(0, \rho ) \cap \Omega} v_j(x)  {\rm d} x\\
	&= \lim_{ \rho \rightarrow +0 }\frac{1}{m(B(0, \rho  ) \cap\Omega )} \int_{B(0, \rho )\cap \Omega } v(x)  {\rm d} x, \\
	\frac{f_2(0)}{-k^2}&= \lim_{ \rho \rightarrow +0 }\frac{1}{m(B(0, \rho  ) \cap\Omega )} \int_{B(0, \rho )\cap 
	\Omega } qw(x)  {\rm d} x,
\end{align*}
together  with  
$$
 \lim_{ \rho \rightarrow +0 }\frac{1}{m(B(0, \rho  )\cap \Omega )} \int_{B(0, \rho )\cap \Omega} v(x)  {\rm d} x = \lim_{ \rho \rightarrow +0 }\frac{1}{m(B(0, \rho  )\cap \Omega )} \int_{B(0, \rho )\cap \Omega} w(x)  {\rm d} x, 
$$
	we finish the proof of this corollary. 
	\end{proof}

\begin{remark}
If $V(x)$ is continuous near the corner $0$ and $ V(0) \neq 0$, from the fact that	
\begin{align*}
	&\lim_{ \rho \rightarrow +0 }\frac{1}{m(B(0, \rho  )\cap \Omega)} \int_{B(0, \rho )\cap \Omega}V(x) w(x)  {\rm d} x\\
	&= V(x_c)
	\lim_{ \rho \rightarrow +0 }\frac{1}{m(B(0, \rho  )\cap \Omega)} \int_{B(0, \rho )\cap \Omega}w(x)  {\rm d} x,
\end{align*}
we can prove that the vanishing property near the corner $0$ of the interior transmission eigenfunctions $v   \in H^1(\Omega )$ and $w \in H^1(\Omega )$ under the assumptions  \eqref{eq:ass3} and \eqref{eq:ass1 int}. 

\end{remark}

If stronger regularity conditions are satisfied by the conductive transmission eigenfunctions $v$ and $w$ to \eqref{eq:in eig}, we can show that more apparent vanishing properties hold at the corner. The rest of this section is devoted to this case. In fact, we have the following theorem.
%
% where we  suppose that $v$ and $w$ have $H^2$ regularity near the neighborhood of the    corner $O$ which can be regard as the origin without loss of generality. The proof of Theorem \ref{Th:1.2} is similar to the proof of Theorem \ref{Th:1.1}. The only difference is that  the $H^2$ regularity of $v$ can guarantee that $v$ has $C^\alpha $ continuous property from Sobolev embedding theorem. Therefore $C^\alpha $ continuous property  of $v$ can help us to investigate the boundary integral on $\Gamma_h^\pm $ easily by splitting $v$ as the sum of $v(0)$ and $\delta v(x)$, where $|\delta v(x) | \leq \|v\|_{C^\alpha } |x|^\alpha$. 

\begin{theorem}\label{Th:1.2}
	Let $v \in H^2(\Omega )$ and $w  \in H^1(\Omega ) $ be eigenfunctions to \eqref{eq:in eig}. Assume that $\Omega \subset \R^2 $ contains a corner {\bl $\Omega\cap B_h= \Omega \cap W$ {\bl with the vertex being $0\in \partial \Omega$}, where $W$ is the sector defined in \eqref{eq:W}  and $h\in \mathbb R_+$}. Moreover,  there exits a sufficiently small neighbourhood $S_h$ (i.e. $h>0$ is sufficiently small) of $0$ in $\Omega$, such that  $qw  \in C^\alpha(\overline {S}_h ) $ and $\eta \in C^\alpha\left(\overline{\Gamma}_h^\pm \right)$ for $0< \alpha  <1$, and $v-w \in H^2(\Sigma_{\Lambda_h} )$.   Under the following assumptions: 
	\begin{itemize}
	\item[(a)] the function $\eta(x)$ doest not vanish at the vertex $0$, i.e., 
	\begin{equation}\label{eq:ass21}
		\eta(0) \neq 0, 
	\end{equation}
	\item[(b)]  the open angle of $S_h$ containing the corner satisfies
	\begin{equation}\label{eq:ass31}
-\pi < \theta_m < \theta_M < \pi \mbox{ and } \theta_M-\theta_m \neq \pi, %|\theta_M-\theta_M|\neq \pi,	%\cos \theta_m+\cos \theta_M \mbox{ and } \sin \theta_m+\sin \theta_M
\end{equation}
%can not be zero simultaneously, 
	\end{itemize}
	 then we have
	$
	v(0) =w(0)=0. 
	$
%	where $B(x_c,\rho )$ is a ball centered at $x_c$ with radius $\rho$ and $m(B(x_c, \rho  ))$ is the volume of $B(x_c,\rho )$. 
\end{theorem}

\begin{proof}
{\bl It is clear that the transmission eigenfunctions $v \in H^2(\Omega )$ and $w \in H^1(\Omega )$ to \eqref{eq:in eig} fulfill \eqref{eq:in eignew}.}  	Recall that $f_1$ and $f_2$ are defined by \eqref{eq:vw}. {\bl From Lemma \ref{lem:int1}, we know that \eqref{eq:221 int} is satisfied, which can be further formulated as 
	  \begin{align}\label{eq:1.55}
	 	\int_{S_h  } ( f_{1}-f_2)u_0(sx) {\rm d} x 
	 	&=I_3
	 	%\int_{\Lambda_h   } ( u_0(sx) \partial_\nu (v-w )-(v-w)  \partial_\nu u_0(sx) ) {\rm d} \sigma \nonumber \\
	 	%&  \quad 
	 	- \int_{\Gamma_{h} ^\pm } \eta(x)  u_0 (sx) v(x) {\rm d} \sigma, 
	 \end{align}
	 where $I_3$ is defined in \eqref{eq:intnotation}. } 
Since  $ f_2\in C^\alpha(\overline {S}_h )$ and $\eta \in C^\alpha\left(\overline{\Gamma}_h^\pm \right)$,   {\bl we know that  $\eta$ and $f_2$ have the expansions \eqref{eq:eta} and \eqref{eq:f1jf2 notation} around the origin, respectively. Furthermore, due to the fact that  $v \in H^2(S_h)$, which can be embedded into $C^\alpha(\overline { S}_h )$, we have} the following expansions
\begin{align}\label{eq:splitting}
\begin{split}
	f_1(x)&=f_1(0)+\delta f_1(x), \quad |\delta f_1(x)| \leq \|f_1\|_{C^\alpha } |x|^\alpha,\\
	%f_2(x)&=f_2(0)+\delta f_2(x), \quad |\delta f_2(x)| \leq \|f_2\|_{C^\alpha } |x|^\alpha,\nonumber \\
	%\eta(x)&=\eta(0)+\delta \eta(x), \quad |\delta\eta (x)| \leq \|\eta\|_{C^\alpha } |x|^\alpha,\nonumber \\
v(x)&=v(0)+\delta v(x), \quad |\delta v (x)| \leq \|v\|_{C^\alpha } |x|^\alpha. 
\end{split}
\end{align}
Substituting \eqref{eq:eta}, \eqref{eq:f1jf2 notation} and \eqref{eq:splitting} into \eqref{eq:1.55}, we can derive that
 \begin{align}\label{eq:1.57}
	 	&(f_1(0)-f_2(0))\int_{S_h  } u_0(sx) {\rm d} x +	\int_{S_h  } ( \delta f_{1}-\delta f_2)u_0(sx) {\rm d} x 
	 	\nonumber\\
	 	& =I_3
	 	%\int_{\Lambda_h   } ( u_0(sx) \partial_\nu (v-w )-(v-w)  \partial_\nu u_0(sx) ) {\rm d} \sigma   
	 	- \eta(0)  v(0)\int_{\Gamma_{h} ^\pm }  u_0 (sx)  {\rm d} \sigma -\eta(0) \int_{\Gamma_{h} ^\pm } \delta v(x) u_0 (sx) {\rm d} \sigma - v(0) \int_{\Gamma_{h} ^\pm } \delta \eta (x) u_0 (sx) {\rm d} \sigma  \nonumber\\
	 	& \quad - \int_{\Gamma_{h} ^\pm } \delta \eta (x)  \delta v(x) u_0 (sx) {\rm d} \sigma . 
	 \end{align}
{\bl Using \eqref{eq:I311} in Lemma \ref{lem:u0 int},}  it is easy to see that
\begin{align}\label{eq:1.59}
	 \eta(0)  v(0)\int_{\Gamma_{h} ^+ }  u_0 (sx)  {\rm d} \sigma &=2 s^{-1}v(0)\eta(0) \left( \mu(\theta_M )^{-2}-   \mu(\theta_M )^{-2} e^{ -\sqrt{sh} \mu(\theta_M ) }\right.  \\&\left. \hspace{3.5cm} -  \mu(\theta_M )^{-1} \sqrt{sh}   e^{ -\sqrt{sh} \mu(\theta_M ) }  \right  ) \nonumber \\
%	&+ \left( \mu(\theta_M )^{-2}-   \mu(\theta_M )^{-2} e^{ -\sqrt{sh}} -  \mu(\theta_M )^{-1} \sqrt{sh}   e^{ -\sqrt{sh}} \right  ) \Bigg] ,\\
	\eta(0)  v(0)\int_{\Gamma_{h} ^- }  u_0 (sx)  {\rm d} \sigma &=2 s^{-1}v(0)\eta(0) \left( \mu(\theta_m )^{-2}-   \mu(\theta_m )^{-2} e^{ -\sqrt{sh}\mu(\theta_m )} \right.  \nonumber \\&\left. \hspace{3.5cm}  -  \mu(\theta_m )^{-1} \sqrt{sh}   e^{ -\sqrt{sh}\mu(\theta_m ) }  \right  ), \nonumber
%	&+ \left( \mu(\theta_m )^{-2}-   \mu(\theta_m )^{-2} e^{ -\sqrt{sh}} -  \mu(\theta_m )^{-1} \sqrt{sh}   e^{ -\sqrt{sh}} \right  ) \Bigg] . \nonumber
\end{align}
where $\mu(\theta)$ is defined in \eqref{eq:omegamu}. 
Besides, from \eqref{eq:splitting}, using \eqref{eq:zeta}, we can estimate
\begin{align}\label{eq:1.60}
	s\left| \int_{\Gamma_{h} ^- } \delta v(x) u_0 (sx) {\rm d} \sigma \right| &\leq  s\|v\|_{C^\alpha }  \int_0^h r^\alpha  e^{-\sqrt{sr} \omega(\theta_m) } {\rm d} r=\Oh(s^{-\alpha}), \\ 
	s\left| \int_{\Gamma_{h} ^- } \delta \eta(x) u_0 (sx) {\rm d} \sigma \right| &\leq  s\|\eta \|_{C^\alpha }  \int_0^h r^\alpha  e^{-\sqrt{sr} \omega(\theta_m) } {\rm d} r=\Oh(s^{-\alpha}),  \nonumber \\ 
	s\left| \int_{\Gamma_{h} ^- } \delta v(x) \delta \eta (x) u_0 (sx) {\rm d} \sigma \right| &\leq  s\|v\|_{C^\alpha }  \|\eta \|_{C^\alpha }  \int_0^h r^{2\alpha}  e^{-\sqrt{sr} \omega(\theta_m) } {\rm d} r=\Oh(s^{-2\alpha}), \nonumber \\
	s\left| \int_{S_h  }  \delta f_{1} u_0(sx) {\rm d} x  \right|  & \leq s \cdot \|f_1\|_{C^\alpha } \int_W |u_0(sx)| |x|^\alpha {\rm d} x \nonumber\\
	& \leq \frac{2 \|f_1\|_{C^\alpha } (\theta_M-\theta_m )\Gamma(2\alpha+4) }{ \delta_W^{2\alpha+4}} s^{-\alpha-1}\nonumber\\
	s \left| \int_{S_h  }  \delta f_{2} u_0(sx) {\rm d} x  \right|  & \leq s \cdot \|f_2\|_{C^\alpha } \int_W |u_0(sx)| |x|^\alpha {\rm d} x \nonumber \\
	& \leq  \frac{2 \|f_2\|_{C^\alpha } (\theta_M-\theta_m )\Gamma(2\alpha+4) }{ \delta_W^{2\alpha+4}} s^{-\alpha-1}.  \nonumber
\end{align}

 Substituting \eqref{eq:1.59} into \eqref{eq:1.57} and multiplying $s$ on the both sides of \eqref{eq:1.57}, after arranging terms, we obtain that
  \begin{align}\label{eq:1.57new}
	 	2& v(0)\eta(0) \left( \mu(\theta_M )^{-2} +\mu(\theta_m )^{-2}   \right)= 2v(0)\eta(0) \Big ( \mu(\theta_M )^{-2} e^{ -\sqrt{sh} \mu(\theta_M )  }\\
	 	& + \mu(\theta_M )^{-1} \sqrt{sh}   e^{ -\sqrt{sh} \mu(\theta_M ) }  +\mu(\theta_m )^{-2} e^{ -\sqrt{sh} \mu(\theta_m )  }  + \mu(\theta_m )^{-1} \sqrt{sh}   e^{ -\sqrt{sh} \mu(\theta_m ) }  \Big )\nonumber \\
	 	&+s\Big [ I_3 - (f_1(0)-f_2(0))\int_{S_h  } u_0(sx) {\rm d} x -	\int_{S_h  } ( \delta f_{1}-\delta f_2)u_0(sx) {\rm d} x 
	 	\nonumber\\
	 	&   -\eta(0) \int_{\Gamma_{h} ^\pm } \delta v(x) u_0 (sx) {\rm d} \sigma  - v(0) \int_{\Gamma_{h} ^\pm } \delta \eta (x) u_0 (sx) {\rm d} \sigma - \int_{\Gamma_{h} ^\pm } \delta \eta (x)  \delta v(x) u_0 (sx) {\rm d} \sigma \Big]. \nonumber 
	 \end{align}
 Since { $v-w\in H^2(\Sigma_{ \Lambda_h} )$,}  \eqref{eq:I3} still holds. In \eqref{eq:1.57new}, letting $s \rightarrow \infty$, from   \eqref{eq:1.5},  \eqref{eq:I3} {\bl in Lemma \ref{lem:int1}}, \eqref{eq:258sh} and \eqref{eq:1.60}, we can show that
 $$
\eta(0) \left( \mu(\theta_M )^{-2} +\mu(\theta_m )^{-2} \right) v(0) =0. 
 $$
 Under the assumption \eqref{eq:ass31},  from {\bl Lemma \ref{lem:29}, we have} $\mu(\theta_M )^{-2} +\mu(\theta_m )^{-2} \neq 0$. Since $\eta(0)\neq 0$ from \eqref{eq:ass21}, we finish the proof of this theorem. 
%Multiplying $s$ on the both sides of \eqref{eq:1.57new}
%
%
%\eqref{eq:umneq0} 
%
%{\color{blue}Again from \cite[Page 9]{Bsource} using
%$$
%\left|\int_{\Lambda_h} ( u_0 (sx)\partial_\nu (v-w)- (v-w)\partial_\nu u_0(sx)  ) {\rm d} \sigma\right| \leq   C e^{-c' \sqrt s}, \quad C,c'>0, 
%$$
%under the assumption that $v-w\in H^2(S_h )$,  substituting \eqref{eq:1.59}  and \eqref{eq:1.60}  into \eqref{eq:1.57} together with  \eqref{eq:u0w}, multiplying $s$ on the both sides of \eqref{eq:1.57}, let $s \rightarrow \infty$, combining the assumptions \eqref{eq:ass21} with \eqref{eq:ass31}, we finish the proof of this theorem. }
\end{proof}

\begin{remark}\label{rem:2.5}
Under the $H^2$ regularity,  the interior  transmission eigenfunctions to \eqref{eq:in eig reduce} have been  shown that they always  vanish at a corner point if the interior angle of the corner is not $\pi$; see  \cite[Theorem 4.2]{Bsource}  for more details. 
\end{remark}

\section{Vanishing near corners of conductive transmission eigenfunctions: three-dimensional case}\label{sec:3}

In this section, we study the vanishing property of the conductive transmission eigenfunctions for the 3D case. In principle, we could also consider a generic corner in the usual sense as the one for the 2D case. However, in what follows, we introduce a more general corner geometry that is described by {\bl $S_h \times (-M,M)$, where $S_h$ is defined in \eqref{eq:sh}} and $M\in\mathbb{R}_+$. It is readily seen that {\bl $S_h \times (-M,M)$}  actually describes an edge singularity and we call it a 3D corner for notational unification. Suppose that the Lipschitz domain $\Omega\subset\mathbb{R}^3$ {\bl with $0\in \partial \Omega$} possesses a 3D corner. {\bl Let $0 \in \R^{2}$} be the vertex of $S_h$ and {\bl $ x_3^c \in (-M,M)$.}  Then {\bl $(0,x_{3}^c )$} is defined as the edge point of {\bl $S_h \times (-M,M)$.}  
In Figure \ref{fig3d}, we give a schematic illustration of the geometry considered in 3D.    In this section, under some appropriate assumptions, we show that the conductive transmission eigenfunctions $v$ and $w$ vanish at {\bl $(0,x_{3}^c )$.}  Since the CGO solution constructed in Lemma \ref{lem:1} is only two dimensional, in order to make use of the similar arguments of Theorem \ref{Th:1.1}, we introduce the following dimension reduction operator.  The dimension reduction operator technique is also introduced in \cite[Lemma 3.4]{Bsource} for studying the vanishing property of nonradiating sources and the transmission eigenfunctions at edges in three dimension. Similar to  Theorem \ref{Th:1.1}, we first assume that $v$ is only $H^1$ smooth but can be approximated by the Herglotz wave functions  with some mild assumptions, where in Theorem \ref{Th:3.1} the interior angle of $S_h$ cannot be $\pi$.  Besides, if $v$ has $H^2$ regularity near the edge point, in Theorem \ref{Th:3.2} we also prove the vanishing property of $v$ and $w$ near the edge point.

\begin{figure}
  \centering
  \includegraphics[width=0.25\textwidth]{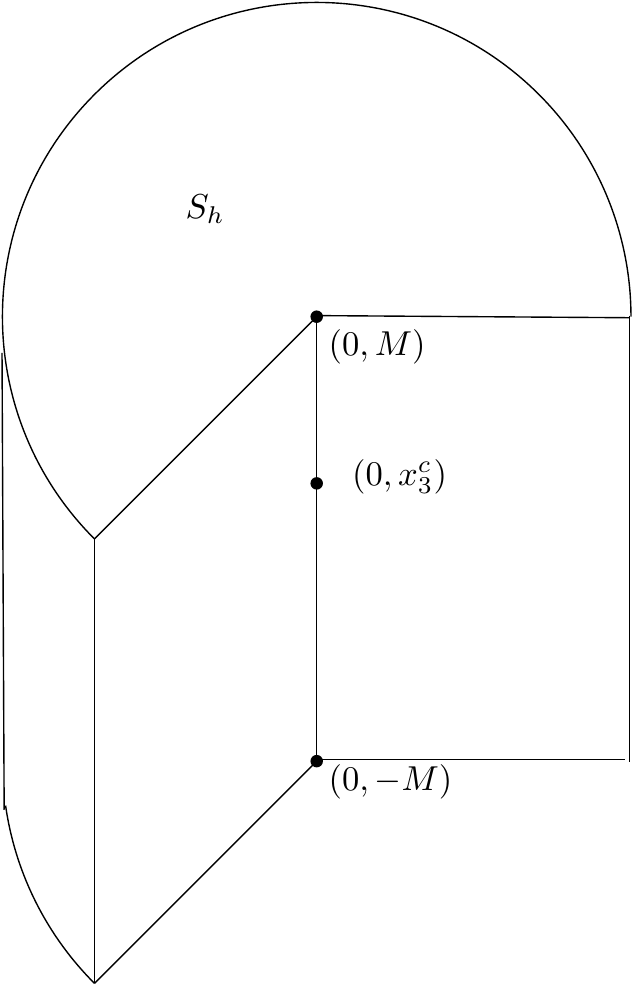}
  \caption{Schematic illustration of the corner in 3D.}
  \label{fig3d}
\end{figure}

%%%%%%%%%%%%%%%%%%%%%%%%%%%%%%%%%%%%%%%%%%%%%%%%%%%%%%%5

\begin{definition}\label{Def}
	{\bl Let ${S_h}\subset \R^{2}$ be defined in \eqref{eq:sh}, $M>0$. For a given function $g$ with the domain $S_h \times (-M,M )$.  Pick up any point $ x_{3}^c \in (-M, M)$. Suppose that $\psi \in C_0^{\infty}( (x_{3}^c-L, x_{3}^c+  L) )$ is  a nonnegative function and $\psi\not\equiv0$, where $L$  is   sufficiently small  such that $ (x_{3}^c-L, x_{3}^c+  L) \subset   (-M,M)$, and write $x=(x', x_3) \in \R^{3}, \, x'\in \R^{2}$.  The dimension reduction operator $\mathcal{R}$ is defined by
	\begin{equation}\label{dim redu op}
	\mathcal{R}(g)(x')=\int_{x_{3}^c -L}^{x_{3}^c+ L} \psi(x_3)g(x',x_3)\,\rmd x_3,
	\end{equation}
	where $x'\in {S_h}$. } 
\end{definition}

\begin{remark}
The assumption on the non-negativity of 	$\psi$ plays an important role in our proof of Theorem \ref{Th:3.1} in what follows, where we use the integral mean value theorem to carefully investigate the asymptotic property of the parameter $s$ appearing in the  CGO solution $u_0(sx')$ given in Lemma \ref{lem:1} as $s \rightarrow \infty$. In order to use the two dimensional  CGO solution $u_0(sx')$ to prove the vanishing property of the conductive transmission eigenfunctions in $\R^3$, we need the dimension reduction operator defined in Definition \ref{Def} in our proof of Theorem \ref{Th:3.1}. 
\end{remark}

Before presenting the main results of this section, we first analyze the regularity of the functions after applying the dimension reduction operator.  Using a similar argument of \cite[Lemma 3.4]{Bsource}, we can prove the following lemma, whose detailed proof is omitted. 

\begin{lemma}\label{lem h2}
	Let  $g\in H^2({S_h}\times(-M ,M))\cap  C^\alpha (\overline {S}_h\times [-M,M])$, where $0<\alpha<1$. Then 
	$$
	{\mathcal R}(g)(x') \in H^2({S_h})\cap  C^\alpha (\overline{S}_h).
	$$ 
 {\bl Similarly, if $g\in H^1({S_h}\times(-M ,M))$, we have
$
	{\mathcal R}(g)(x') \in H^1({S_h}).
	$ }
   
\end{lemma}

{\bl 
In Theorem \ref{Th:3.1}, we shall prove the vanishing property of conductive transmission eigenfunctions at an edge corner in 3D. Let us first introduce the mathematical setup. 

Let ${S_h}\subset \R^{2}$ be defined in \eqref{eq:sh}, $M>0$, $0<\alpha<1$. For any fixed {\bl $x_3^c \in (-M,M)$} and $L>0$ defined in Definition \ref{Def}, we suppose that  $L $ is sufficiently small such that $(x_3^c-L,x_3^c+L) \subset (-M,M) $.  Write  $x=(x', x_3) \in \R^{3}, \, x'\in \R^{2}$. Let $v,w\in H^1({S}_h\times(-M,M))$ fulfill that 
	\begin{align}\label{eq:3d in eig}
	\left\{
	\begin{array}{l}
	\Delta v+ k^2  v=0, \qquad x'\in {S_h}, -M<x_3<M,\\[5pt] 
	\Delta w+k^2 q w=0, \quad x'\in {S_h}, -M<x_3<M,\\[5pt] 
	w= v\quad \partial_\nu v + \eta v=\partial_\nu w, \quad x'\in\Gamma_h^\pm, -M<x_3<M,
	\end{array}
	\right.
	\end{align}
	where $\Gamma_h^\pm $ are defined in \eqref{eq:sh}, $\nu$ is the outward normal vector to $\Gamma_h^\pm \times (-M,M)$,   $q\in L^\infty(S_h \times (-M,M)) $ and $\eta \in L^\infty(\Gamma_h^\pm \times (-M,M) )$ is independent of $x_3$.

Lemmas \ref{lem:32} and \ref{lem:37 coeff} will be used to prove Theorem \ref{Th:3.1} in what follows.

\begin{lemma}\label{lem:32}
	Suppose that $v,w\in H^1({S}_h\times(-M,M))$ fulfill  \eqref{eq:3d in eig}. 
%	Recall that the dimensional reduction operator $\mathcal R$ associated with $\psi$ is defined in \eqref{Def}. 
	Denote
	\begin{eqnarray}\label{construct G}
	\begin{split}
		G(x')&=& \int_{-L}^{L} \psi''(x_3)v(x', x_3)\rmd x_3-k^2 {\mathcal R} (v)(x'),\\
		\widetilde G(x')&=& \int_{-L}^{L} \psi''(x_3)w(x', x_3)\rmd x_3-k^2 {\mathcal R} (qw)(x'),
			\end{split}
	\end{eqnarray}
	where $\mathcal{R}$ is the dimension reduction operator associated with $\psi$ defined in \eqref{Def}. Then there hold that
	\begin{align}\label{eq:eig reduc}
%$$%\lb{1div}
\left\{
\begin{array}{l}
 \Delta_{x'} {\mathcal R} (v) (x')
    =G(x') \hspace*{7.8cm} \mbox{ in } S_h, \\[5pt] 
 \Delta_{x'} {\mathcal R} (w) (x')    =\widetilde G(x')\hspace*{7.6cm}\ \mbox{ in } S_h, \\[5pt] 
\mathcal{R}(w)(x')= \mathcal{R}(v) (x'),\ \partial_\nu {\mathcal R} (v) (x') + \eta(x'){\mathcal R} ( v)  (x')= \partial_\nu {\mathcal R}( w)  (x') \ \ \ \ \mbox{ on } \Gamma_h^\pm ,
  \end{array}
\right.
%$$
\end{align}
in the distributional sense, where  $\nu$ signifies the exterior unit normal vector to $\Gamma_h^\pm $. Let \begin{equation}\label{eq:F1xi}
	{ G}(x')-\widetilde G(x')=F_{1} (x')+F_{2} (x')+F_{3} (x'),
\end{equation}
where
\begin{align*}
	F_{1} (x')&=
	\int_{-L}^{L}\psi''(x_3)(v(x', x_3)-w(x', x_3)) {\rm d} x_3,\, F_{2} (x')=k^2  {\mathcal R} (q w)(x'),\\ F_{3} (x')&=-k^2  {\mathcal R} (v)(x') . 
\end{align*}
Recall that the CGO solution $u_0(sx')$ is defined in \eqref{eq:u0} with the parameter $s\in \mathbb R_+$. There holds the following integral identity,
\begin{equation}\label{eq:36 int}
	\begin{split}
		&\int_{S_h   }  ( F_{1} (x')+F_{2} (x')+F_{3} (x') ) u_0(sx')  {\rm d} x'=I_3 - \int_{\Gamma_{ h} ^\pm } \eta(x') {\mathcal R} ( v)(x') u_0 (sx')  {\rm d} \sigma, 
	\end{split}
\end{equation}
where $I_3=\int_{\Lambda_h   } ( u_0(sx') \partial_\nu {\mathcal R}(v-w )(x')-{\mathcal R}(v-w) (x') \partial_\nu u_0(sx') ) {\rm d} \sigma $. 
If  $qw \in C^\alpha(\overline{S}_h \times [-M,M]  )$   for $0<\alpha<1$   and  $v-w\in H^2({S_h}\times (-M,M) )$, then we have  $F_{1} (x') \in C^\alpha(\overline{ S}_h)$ for $\alpha\in(0,1)$ and $F_{2} (x') \in C^\alpha(\overline{ S}_h )$. 
\end{lemma}

\begin{proof}
For the edge point $(0,x_3^c)\in S_h \times (-M,M)$, where $x_3^c\in (-M,M)$, without loss of generality, in the subsequent analysis, we assume that $x_3^c=0$. Since $\Delta_{x'} v=-k^2 v-\partial_{x_3}^2 v$ and $\Delta_{x'} w=-k^2 q w-\partial_{x_3}^2 w$, by the dominate  convergence theorem, integration by parts gives 
\begin{align}
    \Delta_{x'} {\mathcal R} (v) (x')&=
     \int_{-L}^{L} \psi''(x_3)v(x', x_3)\rmd x_3-k^2 {\mathcal R} (v)(x') =G(x'),\label{add1} \\
     \Delta_{x'} {\mathcal R} (w) (x')&    = \int_{-L}^{L} \psi''(x_3)w(x', x_3)\rmd x_3-k^2 {\mathcal R} (qw)(x') =\widetilde G(x'). \label{add2}
    \end{align}
	Moreover, we have
	\begin{equation}\label{3eq:bound1}
		\mathcal{R}(w)(x')= \mathcal{R}(v) (x') \mbox { on } \Gamma_h^\pm
	\end{equation}
	in the sense of distribution, since $w(x',x_3)=v(x',x_3)$ when $x' \in \Gamma_h^\pm$ and $-L < x_3 < L$. Similarly, using the fact that  $\eta$ is independent  of $x_3$, we can easily show that 
	\begin{equation}\label{3eq:bound2}
		\partial_\nu {\mathcal R} (v) (x') + \eta(x'){\mathcal R} ( v)  (x')= \partial_\nu {\mathcal R}( w)  (x') \mbox { on } \Gamma_h^\pm,
	\end{equation}
in the sense of distribution.

Subtracting \eqref{add2} from \eqref{add1}, combining with the boundary condition \eqref{3eq:bound1} and \eqref{3eq:bound2}, we deduce that
\begin{equation}\label{eq:310 pde}
\begin{split}
	&\Delta_{x'}({\mathcal R}(v)(x')-{\mathcal R}(w)(x') )=F_1 (x')+F_2(x')+F_3 (x') \mbox{ in } S_h, \\
	&\mathcal{R}(v)(x')- \mathcal{R}(w) (x')=0,\ \partial_\nu {\mathcal R} (v) (x') - \partial_\nu {\mathcal R}( w)  (x') = -\eta(x'){\mathcal R} ( v)  (x') \mbox{ on } \Gamma_h^\pm . 
	\end{split}
	\end{equation}

Recall that $u_0(sx')\in H^1(S_h)$ from Lemma \ref{lem:23}. Since $v,\, w \in H^1(S_h\times (-M,M) )$ and $q\in L^{\infty}(S_h \times (-M,M))$, by virtue of Lemma \ref{lem h2}, it yields that $F_1,\ F_2,\, F_3 \in L^2(S_h)$. By Lemma \ref{lem:green} and using the fact that $\Delta_{x'} u_0(sx')=0$ in $S_h$, we have the following Green identity 
\begin{equation}\label{eq:311 int}
	\int_{S_h} u_0(sx') \Delta_{x'}{\mathcal R} (v-w)(x')\rmd x'=\int_{\partial S_h} (u_0(sx') \partial_\nu {\mathcal R}(v-w)(x')- {\mathcal R}(v-w)(x') \partial_\nu u_0(sx'))\rmd \sigma .
\end{equation}
 Substituting \eqref{eq:310 pde} into \eqref{eq:311 int},  it yields \eqref{eq:36 int}. 

Recall that $F_1$ and $F_2$ are defined in \eqref{eq:F1xi}. Since  $v-w\in H^2({S_h}\times(-M,M))$, from  Lemma \ref{lem h2}, we know that  $F_{1} (x') \in H^2(S_h)$,  which can be embedded into $C^\alpha(\overline{ S}_h)$ for $\alpha\in(0,1)$. Moreover, from  Lemma \ref{lem h2},  we have $F_{2} (x') \in C^\alpha(\overline{ S}_h )$, since $qw \in C^{\alpha}(\overline{ S}_h \times [-M,M] )$ and $0<\alpha<1$. 	
\end{proof}

 \begin{lemma}
 	\label{lem:int2}
Let $S_h$ and $\Gamma_h^\pm$ be defined in \eqref{eq:sh}. 		Suppose that $v,w\in H^1({S}_h\times(-M,M))$ fulfill  \eqref{eq:3d in eig}.  Recall that the CGO solution $u_0(sx)$ is defined in \eqref{eq:u0} with the parameter $s\in \mathbb R_+$.  	Let
$$
	F_{3j} (x')=-k^2  {\mathcal R} (v_j)(x') ,
	$$ 
	where $v_j$ is the Herglotz wave function given by
\begin{equation}\label{eq:herg3}
	v_j(x)=\int_{{\mathbb S}^{2}} e^{i k d \cdot x} g_j(d ) {\rm d} \sigma(d ), \quad d \in {\mathbb S}^{2}. 
\end{equation}
Then  $F_{3j} (x') \in C^\alpha (\overline S_h )$ and it has the expansion
\begin{equation}\label{eq:F3j}
		F_{3j}(x')=F_{3j} (0)+\delta F_{3j }(x'),\quad |\delta F_{3j} (x')  | \leq \|F_{3j} \|_{C^\alpha } |x'|^\alpha. 
\end{equation}
Recall that $F_{1} (x')$ and $F_{2} (x')$  are defined in \eqref{eq:F1xi}. Assume that $F_{1} (x')
%=\int_{-L}^{L}\psi''(x_3)(v(x', x_3)-w(x', x_3)) {\rm d} x_3 
	\in C^\alpha(\overline S_h)$ and $F_2(x')
	%=-k^2 {\mathcal R} (qw)(x')
	\in C^\alpha(\overline S_h)$ ($0<\alpha <1$) satisfying
	\begin{align}\label{eq:326F3j}
\begin{split}
	F_{1}(x')&= F_{1} (0)+\delta  F_{1 } (x'),\quad |\delta F_{1 } (x')  | \leq \| F_{1} \|_{C^\alpha } |x'|^\alpha,\\
	F_{2}(x')&=F_{2} (0)+\delta F_{2 }(x'),\quad |\delta F_{2} (x')  | \leq \|F_{2} \|_{C^\alpha } |x'|^\alpha,
	\end{split}
\end{align}
then there holds that
\begin{align}\label{3eq:int identy}
\begin{split}
	&(F_{1 } (0)+F_{2 } (0)+F_{3j } (0)) \int_{S_h} u_0(sx') {\rm d} x'+\delta_j (s) = I_3 -  I_2^\pm - \epsilon_j^\pm(s)\\
	&\quad -\int_{S_h}  \delta  F_{1} (x') u_0(sx') {\rm d} x'  -\int_{S_h} \delta  F_{2} (x')u_0(sx') {\rm d} x' -\int_{S_h} \delta  F_{3j} (x')u_0(sx') {\rm d} x' .  
	\end{split}
\end{align}
where
\begin{align}\label{3eq:intnotation} 
\begin{split}
	 I_2^\pm& =\int_{\Gamma_{h } ^\pm } \eta(x')  u_0(sx') {\mathcal R} (v_j)  (x')  {\rm d} \sigma ,\\ 
	 \delta_j(s)&=-k^2 \int_{S_h} (   {\mathcal R} (v)(x')  -{\mathcal R} (v_j)(x')  )u_0(sx')  {\rm d} x', \\
	  \epsilon_{j}^\pm(s)&=  \int_{\Gamma_{h } ^\pm } \eta(x')  u_0(sx') {\mathcal R}   (v(x',x_3)- v_j (x',x_3) )   {\rm d} \sigma,
%	I_3&=\int_{\Lambda_h} ( u_0 (sx')\partial_\nu {\mathcal R} (v-w)- {\mathcal R} (v-w)\partial_\nu u_0(sx')  ) {\rm d} \sigma.  
\end{split}
	\end{align}
	and $I_3$ is defined in \eqref{eq:36 int}. 
Furthermore,  assuming that the transmission eigenfunction $v$ can be approximated in $H^1(S_h \times (-M,M))$ by the Herglotz wave functions $v_j$ defined in \eqref{eq:herg3}, $j=1,2,\ldots$,  with kernels $g_j$ satisfying 
\begin{equation}\label{3eq:ass0}
\|v-v_j\|_{H^1(S_h \times (-M,M))} \leq j^{-1-\Upsilon},\quad  \|g_j\|_{L^2({\mathbb S}^{2})} \leq C j^{{1+\varrho}}, 
\end{equation}
for some positive constants $C, \Upsilon$ and ${\varrho}$, then there hold that
\begin{subequations}
\begin{align}
		\left|\int_{S_h}  \delta F_{3j} (x) u_0(sx') {\rm d} x' \right | & \leq \frac{8  L\sqrt{\pi}\|\psi\|_{L^\infty}(\theta_M- \theta_m) \Gamma(2 \alpha+4) }{ \delta_W^{2\alpha+4   }  } k^2 {\rm diam}(S_h)^{1-\alpha }\nonumber \\
	& \quad \times (1+k) \|g_j\|_{L^2( {\mathbb S}^{2})}   s^{-\alpha-2   }, \label{3eq:deltaf1j}\\
	\label{3eq:deltaf1}
	\left|\int_{S_h}  \delta F_{1 } (x) u_0(sx') {\rm d} x' \right | &\leq  \frac{2\|F_{1}\|_{C^\alpha } (\theta_M-\theta_m )\Gamma(2\alpha+4) }{ \delta_W^{2\alpha+4}} s^{-\alpha-2}  , \\
		\left|\int_{S_h}  \delta F_{2} (x) u_0(sx') {\rm d} x' \right | &\leq \frac{2\|F_{2 }\|_{C^\alpha } (\theta_M-\theta_m )\Gamma(2\alpha+4) }{ \delta_W^{2\alpha+4}} s^{-\alpha-2}, 	\label{3eq:deltaf2}
\end{align}	
\end{subequations}
where $\Theta \in [0,h]$ and $\delta_W$ is defined in \eqref{eq:xalpha}, as $s\rightarrow +\infty $. If $v-w\in H^2(S_h \times (-M,M) )$, then one has
\begin{align}\label{3eq:I3}
\left| I_3 \right| &\leq  C e^{-c' \sqrt s}\ \ \mbox{as}\ \ s \rightarrow + \infty,
\end{align} 
 where $C>0$ and $c'>0$ are constants.
 \end{lemma}

 \begin{proof} It is clear that the Herglotz wave functions $v_j \in H^2 (S_h\times (-M,M))$. From Lemma  \ref{lem h2}, we have  ${\mathcal R}(v_j)(x')  \in H^2(S_h)$,  which can be embedded into $C^\alpha( \overline{ S}_h) $ satisfying \eqref{eq:F3j}.  
 
 Since   $v \in H^1(S_h \times (-M,M))$ is a solution to  the Helmholtz equation in $S_h \times (-M,M)$, from Lemma \ref{lem:Herg}, $v$ can be approximated by the Herglotz wave functions $v_j(x) $ given in \eqref{eq:herg3} in the $H^1$-topology. Therefore, we deduce that
\begin{align}\label{eq:313delta}
	\int_{S_h} -k^2 {\mathcal R} (v)(x') u_0(sx')  {\rm d} x'= \int_{S_h }  -k^2 {\mathcal R} (v_j)(x') u_0(sx')  {\rm d} x'+ \delta_j(s),
	\end{align}
where %$\widetilde{f}_{1j} (x) =- k^2v_j (x)$, 
$
\delta_j(s)%=-k^2 \int_{S_h} (   {\mathcal R} (v)(x')  -{\mathcal R} (v_j)(x')  )u_0(sx')  {\rm d} x',
$
is defined in \eqref{3eq:intnotation}.

 Let
 	$$
 	I_1= \int_{S_h} u_0(sx') (F_{1} (x')+F_{2} (x')+F_{3j} (x')) {\rm d} x'.
 	$$
 Substituting \eqref{eq:F3j} and \eqref{eq:326F3j} into $I_1$, it yields that
\begin{align*}
I_1 &=(F_{1 } (0)+F_{2 } (0)+F_{3j } (0)) \int_{S_h} u_0(sx') {\rm d} x'+\int_{S_h}  \delta  F_{1} (x') u_0(sx') {\rm d} x'\\
&\quad +\int_{S_h} \delta  F_{2} (x')u_0(sx') {\rm d} x' +\int_{S_h} \delta  F_{3j} (x')u_0(sx') {\rm d} x'. 
\end{align*} 
 Let $F_3(x')$ be defined in \eqref{eq:F1xi}. By virtue of \eqref{eq:313delta}  we have
 \begin{align}
 	%\begin{split}
 		\int_{S_h   }  ( F_{1} (x')+F_{2} (x')+F_{3} (x') ) u_0(sx')  {\rm d} x'&=(F_{1 } (0)+F_{2 } (0)+F_{3j } (0)) \int_{S_h} u_0(sx') {\rm d} x'\notag \\
 		&+\int_{S_h}  \delta  F_{1} (x') u_0(sx') {\rm d} x' +\int_{S_h} \delta  F_{2} (x')u_0(sx') {\rm d} x' \notag \\
&+\int_{S_h} \delta  F_{3j} (x')u_0(sx') {\rm d} x' +  \delta_j(s) . \label{eq:320 I1}
 	%\end{split}
 \end{align}

 Recall that $v$ can be approximated by the Herglotz wave functions $v_j$ given in \eqref{eq:herg3} in the $H^1$-norm. Then
\begin{align}\label{3eq:int3}
	\int_{\Gamma_{h } ^\pm } \eta(x')  u_0(sx') {\mathcal R}(v) (x') {\rm d} \sigma&=\int_{\Gamma_{h } ^\pm } \eta(x')  u_0(sx') {\mathcal R}(v_j) (x')  {\rm d} \sigma + \epsilon_{j}^\pm(s), \\
	 \epsilon_{j}^\pm(s)&=  \int_{\Gamma_{h } ^\pm } \eta(x')  u_0(sx') {\mathcal R}   (v(x',x_3)- v_j (x',x_3) )   {\rm d} \sigma. \nonumber
\end{align}

Plugging \eqref{eq:320 I1} and \eqref{3eq:int3} into \eqref{eq:36 int} in Lemma \ref{lem:32}, we can obtain \eqref{3eq:int identy}.

Recall that $F_{3j} (x')=-k^2  {\mathcal R} (v_j)(x')$.  Using the property of compact embedding of H{\"o}lder spaces, we can derive that for $0< \alpha<1$, 
$$
\|  F_{3j}  \|_{C^\alpha } \leq k^2 {\rm diam}(S_h)^{1-\alpha } \|{\mathcal R}(v_j) \|_{C^1},
$$
where $ {\rm diam}(S_h)$ is the diameter of $S_h$. By the definition of  the dimension reduction operator \eqref{dim redu op}, it is easy to see that
\begin{align*}
	|{\mathcal R}(v_j) (x')| \leq 4 L\sqrt{\pi }\|\psi \|_{L^\infty}\|g_j\|_{L^2({\mathbb S}^{2})},\quad |\partial_{x'}{\mathcal R}(v_j) (x')| \leq 4  k L\sqrt{\pi }\|\psi \|_{L^\infty}\|g_j\|_{L^2({\mathbb S}^{2})}.
\end{align*}
Thus we have
$$
\|{\mathcal R}( v_j) \|_{C^1} \leq 4 L\sqrt{\pi}\|\psi\|_{L^\infty}(1+k) \|g_j\|_{L^2( {\mathbb S}^{2})}.
$$
Therefore  from \eqref{eq:xalpha} in Lemma \ref{lem:1} we have \eqref{3eq:deltaf1j}. 
%\begin{align}\label{3eq:deltaf1j}
%	\left|\int_{S_h}  \delta F_{3j} (x) u_0(sx') {\rm d} x' \right | & \leq \frac{8  L\sqrt{\pi}\|\psi\|_{C^\infty}(\theta_M- \theta_m) \Gamma(2 \alpha+4) }{ \delta_W^{2\alpha+4   }  } k^2 {\rm diam}(S_h)^{1-\alpha }\nonumber \\
%	& \quad \cdot (1+k) \|g_j\|_{L^2( {\mathbb S}^{n-1})}   s^{-\alpha-2   }.
%\end{align}
Using similar arguments, we can deduce \eqref{3eq:deltaf1} and \eqref{3eq:deltaf2}. 
%\begin{align}\label{3eq:deltaf2}
%	\left|\int_{S_h}  \delta F_{1 } (x) u_0(sx') {\rm d} x' \right | &\leq  \frac{2\|F_{1}\|_{C^\alpha } (\theta_M-\theta_m )\Gamma(2\alpha+4) }{ \delta_W^{2\alpha+4}} s^{-\alpha-2}  , \\
%		\left|\int_{S_h}  \delta F_{2} (x) u_0(sx') {\rm d} x' \right | &\leq \frac{2\|F_{2 }\|_{C^\alpha } (\theta_M-\theta_m )\Gamma(2\alpha+4) }{ \delta_W^{2\alpha+4}} s^{-\alpha-2}. \nonumber
%\end{align}

Since $v-w \in H^2(S_{h}  \times (-M,M))$, which   implies that    ${\mathcal R}  (v-w) \in H^2(S_{h} )$ by using Lemma \ref{lem h2}, from \eqref{eq:I3} in Lemma \ref{lem:int1}, we can prove \eqref{3eq:I3}. 

The proof is complete. 
%\begin{equation}\label{3eq:I3}
%	\left| I_3 \right| \leq C e^{-c' \sqrt s}
%\end{equation}
%where $c'>0$ as $s \rightarrow \infty$. } 
 \end{proof}
 
 \begin{lemma}\label{lem:3 delta}
	Under the same setup in Lemma \ref{lem:int2}, we assume that the transmission eigenfunction $v$ to \eqref{eq:3d in eig} can be approximated by a sequence of the Herglotz wave functions $\{v_j\}_{j=1}^{+\infty} $ with the form \eqref{eq:herg3} in $H^1(S_h \times (-M,M))$  satisfying \begin{equation}\label{3eq:ass1}
		\|v-v_j\|_{H^1(S_h \times (-M,M))} \leq j^{-1-\Upsilon},\quad  \|g_j\|_{L^2({\mathbb S}^{2})} \leq C j^{{1+\varrho}}, 
	\end{equation}
	for some positive constant $C$, $\Upsilon>0$ and ${\varrho>0}   $. Let $\delta_j(s)$ and $\epsilon_j^\pm(s)$ be defined in \eqref{3eq:intnotation}.  Then we have the following estimate,
		\begin{align}
			|\delta_j(s)| & \leq \frac{ k^2 \|\psi \|_{L^\infty}  \sqrt{C(L,h)  ( \theta_M-\theta_m )}  e^{-\sqrt{s \Theta } \delta_W } h  }{\sqrt 2 } j^{-1-  \Upsilon}. \label{3eq:deltajnew3}
		\end{align}
Furthermore, assuming that the boundary parameter $\eta$ in \eqref{eq:3d in eig} fulfills $\eta ( x) \in C^{\alpha}(\overline{\Gamma_h^\pm } \times [-M,M] )$, which indicates that $\eta(x')\in C^\alpha(\overline{\Gamma_{ h}^\pm})$ and has the expansion  
\begin{equation}\label{3eq:eta}
	\eta(x')=\eta(0)+\delta \eta(x'),\quad |\delta \eta(x')| \leq \|\eta \|_{C^\alpha  } |x'|^\alpha,
\end{equation}
then under the assumption \eqref{3eq:ass1}, we have
\begin{align}
	|\epsilon_j^\pm (s) | & \leq C {\bl  \| \psi \|_{L^\infty} } \left( |\eta(0)|  \frac{\sqrt { \theta_M-\theta_m }  e^{-\sqrt{s \Theta } \delta_W } h } {\sqrt 2 } \right. \nonumber\\
	&\left. \quad + \|\eta\|_{C^\alpha } s^{-(\alpha+1 )} \frac{\sqrt{2(\theta_M-\theta_m) \Gamma(4\alpha+4) } }{(2\delta_W)^{2\alpha+2  } } \right) j^{-1-\Upsilon}.  \label{3eq:xij} 
\end{align}
	\end{lemma}
	
	\begin{proof}
	We first prove \eqref{3eq:deltajnew3}. Indeed, by using the Cauchy-Schwarz inequality,  we have
\begin{align}
	\|{\mathcal R}(  v)-{\mathcal R}(  v_j) \|_{L^2(S_h)}^2&=\int_{S_h } \left| \int_{-L}^L  \psi(x_3) (v(x',x_3)-v_j(x',x_3))  {\rm d} x_3 \right|^2 {\rm d} x'\nonumber \\
	&\leq C(L,h) \|\psi\|^2_{L^\infty} \|v-v_j\|_{L^2(S_h \times (-L,L))}^2, \label{eq:312}
\end{align}
where  $C(L,h) $ is a positive constant depending on $L$ and $h$. Since the $L^2$-norm of $u_0$ in $S_h$ can be estimated by \eqref{eq:u0L2} in Lemma \ref{lem:23}, recalling that $\delta_j(s)$ is defined in \eqref{3eq:intnotation}, and using Cauchy-Schwarz inequality again,  by virtue of \eqref{3eq:ass1},  
we can deduce \eqref{3eq:deltajnew3}. 
%\begin{align}\label{3eq:deltajnew3}
%	|\delta_j(s)|&\leq k^2 \| {\mathcal R}( v)-{\mathcal R}(v_j) \|_{L^2(S_h )}  \|u_0(sx)\|_{L^2(S_h)}\\
%	&\leq \frac{ k^2 \|\psi \|_\infty  \sqrt{C(L,h)  ( \theta_M-\theta_m )}  e^{-\sqrt{s \Theta } \delta_W } h  }{\sqrt 2 } j^{-1-  \Upsilon}, \nonumber
%\end{align}
%where $\Theta \in [0,h]$ and $\delta_W$ is defined in \eqref{eq:xalpha}. 

  Note that $\epsilon_{j}^\pm(s)$ is defined in \eqref{3eq:intnotation} and $\eta$ has the expansion \eqref{3eq:eta}.  Using the Cauchy-Schwarz inequality and the trace theorem, we have
\begin{align}\label{3eq:19}
	|\epsilon_{j}^\pm(s)|& \leq |\eta(0)|   \int_{\Gamma_{h } ^\pm } | u_0(sx') | | {\mathcal R}   (v(x',x_3)- v_j (x',x_3) ) |   {\rm d} \sigma \\
	&\quad + \|\eta\|_{C^\alpha }   \int_{\Gamma_{h } ^\pm }|x'|^\alpha  | u_0(sx') | | {\mathcal R}   (v(x',x_3)- v_j (x',x_3) ) |   {\rm d} \sigma \notag \\
	&\leq |\eta(0)|   \|{\mathcal R}  (v-v_j) \|_{H^{1/2}(\Gamma_h^\pm ) } \| u_0(sx')\|_{H^{-1/2}(\Gamma_h^\pm ) }  \notag \\
	&\quad + \|\eta\|_{C^\alpha }   \|{\mathcal R}  (v-v_j) \|_{H^{1/2}(\Gamma_h^\pm ) } \| {|x'|}^\alpha u_0(sx')\|_{H^{-1/2}(\Gamma_h^\pm ) } \notag \\
	&\leq |\eta(0)|   \|{\mathcal R}  (v-v_j) \|_{H^{1}(S_h ) } \| u_0(sx')\|_{L^2 (S_h) } \notag \\
	&\quad + \|\eta\|_{C^\alpha }   \|{\mathcal R}  (v-v_j) \|_{H^{1}(S_h)  } \| {|x'|}^\alpha u_0(sx')\|_{L^2 (S_h ) } \notag \\
		&\leq  C \| \psi \|_{L^\infty} \|v-v_j \|_{H^{1}(S_h \times (-L,L)) }(|\eta(0)| \| u_0(sx')\|_{L^{2}(S_h ) } + \|\eta\|_{C^\alpha }\| {|x'|}^\alpha u_0(sx')\|_{L^{2}(S_h ) } ),  \nonumber
\end{align}
where $C$ is a positive constant and the last inequality comes from Lemma \ref{lem h2}.   Substituting   \eqref{eq:u0L2},  \eqref{eq:22}  and \eqref{3eq:ass1} into \eqref{3eq:19}, we obtain \eqref{3eq:xij}. 
	\end{proof}

\begin{lemma}
	Let $j_\ell (t)$ be the $\ell$-th spherical Bessel function with the form
	  \begin{equation}\label{eq:bess sph}
	 	 j_\ell (t)=\frac{t^\ell }{ (2\ell+1)!!}\left  (1-\sum_{l=1}^\infty  \frac{(-1)^l t^{2l }}{ 2^l l! N_{\ell,l} }\right  ),
	 \end{equation}
	 where $N_{\ell,l}=(2\ell+3)\cdots  (2\ell+2l+1)$ and $\mathcal R$ be the dimension reduction operator defined in \eqref{dim redu op}. Then
	 \begin{subequations}
	 	\begin{align}
	 		{\mathcal R} (j_0)(x')	&=C(\psi )\Bigg[ 1-   \sum_{l=1}^\infty  \frac{(-1)^l k^{2l} } { 2^l l! (2l+1)!! }  \left(|x'|^2+a_{0,l} ^2\right )^l\Bigg],\label{eq:j0} \\
	 		{\mathcal R} (j_\ell)(x')	&={ \frac{k^\ell   (|x'|^2+a_{\ell }^2)^{(\ell-1) /2 } }{ (2\ell+1)!!} \Bigg[ 1-  \sum_{l=1}^\infty  \frac{(-1)^l k^{2l} (|x'|^2+a_{\ell,l }^2)^{l }}{ 2^l l! N_{\ell,l}  }   \Bigg] C_1(\psi )|x'|^2  }, \label{eq:jell}
	 	\end{align}
	 \end{subequations}
	 where $j_0=j_0(k|x|)$, $j_\ell=j_\ell(k|x|)$,  $\ell \in  \mathbb N  $, $a_{0,l} \in [-L,L]$, $a_\ell, \, a_{\ell,l} \in [-L,L]$, and
	 \begin{equation}\label{eq:329 cpsi}
	 	C(\psi ) = \int_{-L}^L \psi(x_3) {\rm d} x_3, \quad C_1(\psi )= \int_{-\arctan L/|x'|}^{\arctan L/|x'|}   \psi(|x'| \tan \varpi ) \sec^3 \varpi  {\rm d}   \varpi.
	 \end{equation}
	 Furthermore, there holds that
	 \begin{equation}\label{eq:C1psi}
	0<C_1(\psi) < 2^{5/2}  \| \psi \|_{L^\infty}  \arctan L . 
\end{equation}
\end{lemma}

\begin{proof}
	 From the definition of the dimension reduction operator \eqref{dim redu op} and the integral mean value theorem, we know that
\begin{align}
{\mathcal R} (j_0)(x')	&= \int_{-L}^L  \psi(x_3) j_0(k |x|) {\rm d} x_3 \notag  \\
&= \int_{-L}^L  \psi(x_3) {\rm d} x_3
-  \sum_{l=1}^\infty  \frac{(-1)^l k^{2l} } {2^l l! (2l+1)!! } \int_{-L}^L \psi(x_3)\left  (|x'|^2+x_3^2\right )^l {\rm d} x_3, \nonumber
% \\
%&=C(\psi )\Bigg[ 1-   \sum_{l=1}^\infty  \frac{(-1)^l k^{2l} } { 2^l l! (2l+1)!! }  \left(|x'|^2+a_{0,l} ^2\right )^l\Bigg] , \nonumber
\end{align}
from which we obtain \eqref{eq:j0}. 

For ${\mathcal R}(j_{\ell})(x')=\int_{-L}^L  \psi(x_3) j_\ell (k |x|) {\rm d} x_3$,  {using the integral mean value theorem, we can deduce that for $\ell\in\mathbb{N}$},
\begin{align}\label{eq:mean}
	\int_{-L}^L  \psi(x_3) (|x'|^2+x_3^2)^{\ell/2 }{\rm d} x_3&= (|x'|^2+a_\ell^2)^{(\ell-1)/2 }\int_{-L}^L  \psi(x_3) (|x'|^2+x_3^2)^{1/2 }{\rm d}  x_3\\
	%&=|x'|^2  (|x'|^2+a_\ell^2)^{(\ell-1)/2 } \int_{-\arctan L/|x'|}^{\arctan L/|x'|}   \psi(|x'| \tan \varpi ) \sec^3 \varpi  {\rm d}   \varpi %\quad  (x_n=|x'| \tan \varpi ) 
%	\nonumber\\
	&=C_1(\psi )|x'|^2  (|x'|^2+a_\ell^2)^{(\ell-1)/2 }, \nonumber
		\end{align} 
where $a_\ell \in [-L,L]$. Thus for $\ell\in\mathbb{N}$, from \eqref{eq:bess sph}, we have 
\begin{align}
{\mathcal R} (j_\ell)(x')	&=\int_{-L}^L \psi(x_3) j_\ell (k |x|) {\rm d} x_3 \notag \\
&= \frac{k^\ell  }{ (2\ell+1)!!}  \int_{-L}^L  \psi(x_3) (|x'|^2+x_3^2)^{\ell/2 }\left  (1-\sum_{l=1}^\infty  \frac{(-1)^l k^{2l } (|x'|^2+x_3^2)^{l }}{ 2^l l! N_{\ell,l} }\right  )  {\rm d} x_3.  \label{add3}
%\nonumber\\
%&={ \frac{k^\ell   (|x'|^2+a_{\ell }^2)^{(\ell-1) /2 } }{ (2\ell+1)!!} \Bigg[ 1-  \sum_{l=1}^\infty  \frac{(-1)^l k^{2l} (|x'|^2+a_{\ell,l }^2)^{l }}{ 2^l l! N_{\ell,l}  }   \Bigg] C_1(\psi )|x'|^2  },  \nonumber
\end{align}
Substituting \eqref{eq:mean} into \eqref{add3}, we can obtain \eqref{eq:jell}.

Clearly, if $L<|x'|$, we know that $0<\sec \varpi<\sqrt {\frac{L^2}{|x'|^2}+1}$, where  
$$
\varpi \in [-\arctan L/|x'|, \arctan L/|x'| ].
$$ Therefore we can deduce \eqref{eq:C1psi}.  
 \end{proof}

Using the Jacobi-Anger expansion (cf. \cite[Page 75]{CK}), for the Herglotz wave function $v_j$ given in \eqref{eq:herg3}, we have
\begin{equation}\label{3eq:vjex}
	v_j(x)= v_j(0) j_0(k |x| )+ \sum_{\ell =1}^\infty  \gamma_{\ell j}   i^\ell   (2\ell +1 ) j_\ell  ( k |x| ), \quad x\in \R^3 , 
	\end{equation}
	where
	\begin{align*}
		v_j(0)=  \int_{{\mathbb S}^{2}}  g_j(d ) {\rm d} \sigma(d),\quad \gamma_{\ell j}= \int_{{\mathbb S}^{2}}  g_j(d ) P_\ell  (\cos( \varphi )) {\rm d} \sigma(d ), \quad d \in {\mathbb S}^{2}, 
	\end{align*}
	and 
	 $j_\ell (t)$ is the $\ell$-th spherical Bessel function  \cite{Abr} and $\varphi$ is the angle between $x$ and $d$. 
	 
	 %Moreover,  we have the explicit expression of  $j_\ell (t)$  as \eqref{eq:bess sph}. 
	 
	 In the next lemma, we characterize the integrals $I_2^\pm$  defined by \eqref{3eq:intnotation}, which shall play a critical role in the proof of our main Theorem \ref{Th:3.1} in what follows.
	  
\begin{lemma}\label{lem:36}
Let  $\Gamma_h^\pm$ be defined in \eqref{eq:sh} and $u_0(sx)$ be the CGO solution defined in \eqref{eq:u0} with the parameter $s\in \mathbb R_+$, and $I_2^\pm$ be defined by \eqref{3eq:intnotation}.  Recall that the Herglotz wave function  $v_j$ is given in the form \eqref{eq:herg3}.  Suppose that $\eta ( x) \in C^{\alpha}(\overline{\Gamma_h^\pm } \times [-M,M] )$ ($0< \alpha<1$) satisfying \eqref{3eq:eta} and let
\begin{align}\label{3eq:Ieta}
\begin{split}
I^-_{\eta,1}& =\int_{\Gamma^-_h } \delta \eta(x')  u_0(sx')  {\mathcal R}  (j_0)  (x') {\rm d} \sigma,\ \  I^+_{\eta,1}=\int_{\Gamma^+_h } \delta \eta(x')  u_0(sx')  {\mathcal R}  (j_0)  (x') {\rm d} \sigma,\\
  I^-_{\eta,2} &=\sum_{\ell =1}^\infty \gamma_{\ell j } i^\ell (2\ell +1) \int_{\Gamma^-_h } \delta \eta(x')  u_0(sx')  {\mathcal R}  (j_\ell )  (x') {\rm d} \sigma, \\
 I^+_{\eta,2} &=\sum_{\ell =1}^\infty \gamma_{\ell j } i^\ell (2\ell +1) \int_{\Gamma^+_h } \delta \eta(x')  u_0(sx')  {\mathcal R}  (j_\ell )  (x') {\rm d} \sigma,  \\
	I_\eta^-&= v_j(0){I}_{\eta,1}^-+ { I}_{\eta,2}^- ,\quad I_\eta^+= v_j(0){I}_{\eta,1}^++ { I}_{\eta,2}^+. 
	\end{split}
\end{align}
 Assume that for a fixed $k\in 
 \mathbb R_+$, $h$ is sufficiently small such that $k^2(h^2+L^2)<1$ and 
 \begin{equation}\label{eq:lem36 cond}
 	kL<1,
 \end{equation}
 where $2L$ is the length of the interval of the dimensional reduction operator $\mathcal R$ in \eqref{dim redu op},   and $-\pi< \theta_m < \theta_M <\pi $, where $\theta_m$ and $\theta_M$ are defined in \eqref{eq:W}. Then 
\begin{align}\label{3eq:I2-final}
	I_2^-&=2\eta(0)v_j(0)s^{-1}\left( \mu(\theta_m )^{-2}-   \mu(\theta_m )^{-2} e^{ -\sqrt{sh}\mu(\theta_m ) } -  \mu(\theta_m )^{-1} \sqrt{sh} e^{ -\sqrt{sh}\mu(\theta_m ) } \right  )C(I_{311}^-)   \nonumber\\
	&\quad +v_j(0) \eta(0) I_{312}^-+ \eta(0) I_{32}^-+I_\eta^-, 
\end{align}
where $\mu(\theta_m )$ is defined in \eqref{eq:omegamu},
\begin{equation}\label{eq:335 I32}
	\begin{split}
	C(I_{311}^-)&  =C(\psi )  \Bigg[ 1-   \sum_{l=1}^\infty  \frac{(-1)^l k^{2l}} { (2l+1)!! }  a_{0,l} ^{2l} \Bigg] ,\\
		I_{32}^-&=\sum_{\ell =1}^\infty \gamma_{\ell j} i^\ell (2 \ell+1) \int_{\Gamma_h^- } u_0(sx') {\mathcal R} (j_\ell ) (x')  {\rm d} \sigma,\\
	I_{312}^-&=	-    C( \psi ) \sum_{l=1}^\infty  \frac{(-1)^l k^{2l}} { (2l+1)!! } \left( \sum_{i_1=1}^l  C(l,{i_1 })a_{0,l} ^{2(l-i_1) }  \int_{0}^h r^{2 i_1}   e^{-\sqrt{sr} \mu (\theta_m)}   {\rm d} r \right),
	\end{split}
\end{equation}
with $C(l,{i_1 })=\frac{l!}{i_1! (l-i_1)!}$ being the combinatorial number of the order $l$. Here  $a_{0,l}$, $\mathcal R(j_\ell)$ and $C(\psi )$ are defined in  \eqref{eq:j0},  \eqref{eq:jell} and \eqref{eq:329 cpsi}, respectively. 
It holds as $s\rightarrow +\infty$ that 
\begin{subequations}
	\begin{align}%\label{eq:336 bound}
		I_{312}^- &\leq \Oh(s^{-3}), \label{eq:I312} \\
	 I_{32}^- &\leq \Oh (\|g_j\|_{L^2 ({\mathbb S} ^{2})} s^{-3}) , \label{3eq:I32} \\
	\left| I_{\eta,1}^- \right|  &\leq \Oh ( s^{-1-\alpha }),\label{3eq:I1} \\
		\left| I_{\eta,2}^- \right| & \leq \Oh (\|g_j\|_{L^2 ({\mathbb S} ^{2})} s^{-{3}-\alpha }) ,\label{3eq:I2}\\
		  |I_\eta^-| &\leq  \|\eta \|_{C^\alpha } \left( |v_j(0)|  |I_{\eta,1}^-| + |I_{\eta,2}^-|  \right).  \label{eq:336 bound}
	\end{align}
	\end{subequations}

	Similarly, we have
	 \begin{align}\label{3eq:I2+final}
	I_2^+&=2\eta(0)v_j(0)s^{-1}\left( \mu(\theta_M )^{-2}-   \mu(\theta_M )^{-2} e^{ -\sqrt{sh}\mu(\theta_M ) }-  \mu(\theta_M )^{-1} \sqrt{sh}  e^{ -\sqrt{sh} \mu(\theta_M )} \right  )C(I_{311}^+)\nonumber\\
	&\quad  +v_j(0)\eta(0) I_{312}^+ +\eta(0) I_{32}^+ +I_\eta^+, 
\end{align}
where
\begin{equation}\label{eq:CI311+}
\begin{split}
C(I_{311}^+)&=C(\psi )  \Bigg[ 1-   \sum_{l=1}^\infty  \frac{(-1)^l k^{2l}} { (2l+1)!! }  a_{0,l,+} ^{2l} \Bigg],\quad a_{0,l,+} \in [-L,L],  \\
 I_{312}^+&=-C(\psi)\sum_{l=1}^\infty  \frac{(-1)^l k^{2l} } { (2l+1)!! }  \sum_{i_1=1}^l  C(l,{i_1 }) a_{0,l,+} ^{2(l-i_1) }  \int_{0}^h r^{2 i_1}   e^{-\sqrt{sr} \mu(\theta_M)}   {\rm d} r,   \\
   I_{32}^+&=\sum_{\ell =1}^\infty \gamma_{\ell j} i^\ell (2 \ell+1) \int_{\Gamma_h^+ } u_0(sx') {\mathcal R} (j_\ell ) (x')  {\rm d} \sigma.   
%I_{\eta}^+&=\int_{\Gamma^+_h } \delta \eta(x')  u_0(sx') {\mathcal R}  (v_j) (x') {\rm d} \sigma,  ,\\
%		I^+_{\eta,1} &=\int_{\Gamma^+_h } \delta \eta(x')  u_0(sx')  {\mathcal R}   (j_0)  (x') {\rm d} \sigma,\, 	,  
%		\\
%I^+_{\eta,2}& =\sum_{\ell =1}^\infty \gamma_{\ell j } i^\ell (2\ell +1) \int_{\Gamma^+_h } \delta \eta(x')  u_0(sx')  {\mathcal R}   (j_\ell )  (x') {\rm d} \sigma, \, 
%	. 
\end{split} 
\end{equation}
 There hold as $s \rightarrow +\infty$ that
\begin{equation}\label{eq:338 bounds}
	\begin{split}
	| I_{312}^+| &\leq \Oh(s^{-3}), 	\quad  
	| I_{32}^+| \leq \Oh (\|g_j\|_{L^2 ({\mathbb S} ^{2})} s^{-3}),\\
	\left|  I_{\eta,1}^+ \right|  &\leq \Oh ( s^{-1-\alpha }),\quad \left|  I_{\eta, 2}^+  \right|  \leq \Oh (\|g_j\|_{L^2 ({\mathbb S} ^{2})} s^{-{3}-\alpha }),\\
	|I_\eta^+|  &\leq  \|\eta \|_{C^\alpha } \left(| v_j(0)| |  I_{\eta,1}^+| + |I_{\eta,2}^+ | \right).
	\end{split}
\end{equation} 
\end{lemma}

\begin{proof} We first investigate the boundary integral $I_2^- $ which is given by 
\eqref{3eq:intnotation}.  In this situation, the polar coordinates $x=(r\cos \theta, r \sin \theta )$ satisfy $r \in (0, h)$ and $\theta=\theta_m$ or $\theta=\theta_M$ when $x\in \Gamma_h^-$ or $x\in \Gamma_h^+$, respectively.  Since $\eta \in C^\alpha(\overline{\Gamma}_h^\pm  \times [-M,M]  )$, we have the expansion \eqref{3eq:eta}.  Substituting  \eqref{3eq:eta}  into the expression of $I_2^-$, we have\begin{align}\label{3eq:I2-}
	I_2^-   &=\eta(0) I_{21}^-+ I_\eta^-,
\end{align}
where
\begin{align*}
I_{21}^-&=	 \int_{\Gamma_h^- } u_0(sx') {\mathcal R} (v_j) (x')  {\rm d} \sigma ,\quad 
I_{\eta}^- =\int_{\Gamma^-_h } \delta \eta(x')  u_0(sx')  {\mathcal R} (v_j)  (x') {\rm d} \sigma. 
\end{align*}
By virtue of \eqref{3eq:vjex}, it can be verified that $I_\eta^-= v_j(0){I}_{\eta,1}^-+ { I}_{\eta,2}^-$.

Let
$$
{ I}_{31}^-=\int_{\Gamma_h^- } u_0(sx') {\mathcal R} (j_0) (x')  {\rm d} \sigma.
%\quad { I}_{32}^-=\sum_{\ell =1}^\infty \gamma_{\ell j} i^\ell (2 \ell+1) \int_{\Gamma_h^- } u_0(sx') {\mathcal R} (j_\ell ) (x')  {\rm d} \sigma. 
$$
Substituting \eqref{eq:jell} and \eqref{3eq:vjex}  into the expression of $I_{21}^-$ defined in \eqref{3eq:I2-}, we can deduce that
\begin{align}\label{add4}
	I_{21}^-
	%&=	 v_j (0)\int_{\Gamma_h^- } u_0(sx') {\mathcal R} (j_0) (x')  {\rm d} \sigma + \sum_{\ell =1}^\infty \gamma_{\ell j} i^\ell (2 \ell+1) \int_{\Gamma_h^- } u_0(sx') {\mathcal R} (j_\ell ) (x')  {\rm d} \sigma   \\
	&=v_j(0){ I}_{31}^-+ { I}_{32}^-,
\end{align}
where ${ I}_{32}^-$ is defined in \eqref{eq:335 I32}. 
Substituting the expansion \eqref{eq:j0}  into ${ I}_{31}^-$ and recalling that $\mu(\theta )=-\cos(\theta/2+\pi) -i \sin( \theta/2+\pi )$, we have
\begin{align}
	{ I}_{31}^-&= C( \psi )  \int_{0}^h \Bigg[ 1- \sum_{l=1}^\infty  \frac{(-1)^l  k ^{2l}} { (2l+1)!! }  \left(r^2+a_{0,l} ^2\right )^l\Bigg] e^{-\sqrt{sr} \mu(\theta_m)}   {\rm d} r
%	&=C( \psi )  \Bigg[ 1-   \sum_{l=1}^\infty  \frac{(-1)^l  k ^{2l}} { (2l+1)!! }  a_{0,l} ^{2l} \Bigg]   \int_{0}^h e^{-\sqrt{sr} \mu(\theta_m)}   {\rm d} r\\
%	&\quad -    C( \psi ) \sum_{l=1}^\infty  \frac{(-1)^l k^{2l}} { (2l+1)!! } \left( \sum_{i_1=1}^l  C(l,{i_1 })a_{0,l} ^{2(l-i_1) }  \int_{0}^h r^{2 i_1}   e^{-\sqrt{sr} \mu (\theta_m)}   {\rm d} r \right) \\
	:={ I}_{311}^-+{I}_{312}^-,\notag\\
{ I}_{311}^-&=	C( \psi )  \Bigg[ 1-   \sum_{l=1}^\infty  \frac{(-1)^l  k ^{2l}} { (2l+1)!! }  a_{0,l} ^{2l} \Bigg]   \int_{0}^h e^{-\sqrt{sr} \mu(\theta_m)}   {\rm d} r,\label{add5}
\end{align}
where $a_{0,l}$, $C( \psi )$ and ${I}_{312}^-$ are defined in \eqref{eq:j0},  \eqref{eq:329 cpsi} and \eqref{eq:335 I32}, respectively.
%Therefore for a fixed $k$,  if $L \in (0,\frac{1}{ \sqrt{2} k } )$  we can verify that
%\begin{equation*}
%	1-   \sum_{l=1}^\infty  \frac{(-1)^l k^{2l} } { (2l+1)!! }  a_{0,l} ^{2l} >0. 
%\end{equation*}
Moreover, using \eqref{eq:I311} in Lemma \ref{lem:u0 int},  we obtain that 
\begin{equation}\label{3eq:I311}
	{ I}_{311}^-=2s^{-1}\left( \mu(\theta_m )^{-2}-   \mu(\theta_m )^{-2} e^{ -\sqrt{sh}  \mu(\theta_m ) }\ -  \mu(\theta_m )^{-1} \sqrt{sh}   e^{ -\sqrt{sh}  \mu(\theta_m ) } \right  )C(I_{311}^-) ,
\end{equation}
where $
%\begin{equation}\label{eq:337}
	C(I_{311}^-)  =C(\psi )  \Bigg[ 1-   \sum_{l=1}^\infty  \frac{(-1)^l k^{2l}} { (2l+1)!! }  a_{0,l} ^{2l} \Bigg] . 
%\end{equation}
$ 

Finally, substituting \eqref{add4}, \eqref{add5} and \eqref{3eq:I311}
%\eqref{3eq:I1}, \eqref{3eq:I2}, \eqref{3eq:I311} , \eqref{eq:I312}, and \eqref{3eq:I32} 
into \eqref{3eq:I2-}, we have the  integral equality  \eqref{3eq:I2-final}.

%   $C(\psi ) = \int_{-L}^L \psi(x_n) {\rm d} x_n>0$ since  $\psi$ is a nonnegative function. 
 
%  it is not difficult to see that
% $$
% \left|  \sum_{l=1}^\infty  \frac{(-1)^l k^{2l}} { (2l+1)!! }  a_{0,l,+} ^{2l} \right | \leq \frac{(kL)^2}{1-(kL)^2},  
% $$

 Following a similar arguments for deriving the integral equality \eqref{3eq:I2-final} of $I_2^-$, one can derive the  integral equality \eqref{3eq:I2+final} for $I_2^+$.

In the following, we derive the estimate for $I_\eta^-$ in \eqref{eq:336 bound} by investigating \eqref{3eq:I1} and \eqref{3eq:I2}. Substituting   \eqref{eq:j0} into $I_{\eta,1}^-$, we can derive that
\begin{align*}
	|I_{\eta,1}^-| &\leq |C(\psi )|  \|\eta \|_{C^\alpha }  \int_{0}^h r^\alpha   \Bigg| 1-   \sum_{l=1}^\infty  \frac{(-1)^l k^{2l}} { (2l+1)!! }  \left(r^2+a_{0,l} ^2\right )^l\Bigg |  e^{-\sqrt{sr} \omega(\theta_m)} {\rm d} r\\
	&=2 L\| \psi \|_{L^\infty}  \|\eta \|_{C^\alpha }   \Bigg| 1-  \sum_{l=1}^\infty  \frac{(-1)^l k^{2l }} { (2l+1)!! }  \left(\beta_{0,l}^2+a_{0,l} ^2\right )^l\Bigg |  \int_{0}^h r^\alpha e^{-\sqrt{sr} \omega(\theta_m)} {\rm d} r, 
\end{align*}
where $\beta_{0,l} \in [0,h]$ such that $k^2(\beta_{0,l}^2+a_{0,l} ^2 ) \leq  k^2( h^2+L^2) <1$ for sufficiently  small $h$ and $L$.   From \eqref{eq:zeta} in Lemma \ref{lem:1}, we obtain \eqref{3eq:I1}. 
%\begin{equation}\label{3eq:I1}
%	|I_{\eta,1}^-|  \leq \Oh (s^{-\alpha -1})
%\end{equation}
%as $s\rightarrow +\infty$. 
 Substituting \eqref{eq:jell} into $I_{\eta,2}^-$, and using \eqref{eq:C1psi}, we can deduce that
\begin{align}\notag
	I_{\eta,2}^-& \leq  C_1(\psi)  \|\eta \|_{C^\alpha }\nonumber  \\
	& \quad \times \sum_{\ell =1}^\infty |\gamma_{\ell j}| \int_{0}^h r^\alpha   \frac{k^\ell  (r^2+a_{\ell }^2)^{(\ell-1)/2 } }{ (2\ell-1)!!}   \Bigg |1 -  \sum_{l=1}^\infty  \frac{ k^{2l} (r^2+a_{\ell,l }^2)^{ l }}{ 2^l l! N_{\ell,l} }   \Bigg|   r^2 e^{-\sqrt{sr} \omega(\theta_m)} {\rm d} r\nonumber \\
	&\leq C_1(\psi)  \|\eta \|_{C^\alpha }  \|g_j\|_{L^2({\mathbb S}^{2})} \int_{0}^h r^{2+\alpha}   e^{-\sqrt{sr} \omega(\theta_m)} {\rm d} r   \nonumber \\
    &\quad \quad \times \sum_{\ell =1}^\infty  \frac{k^\ell   (\beta_\ell^2+a_\ell ^2)^{(\ell-1)/2 } }{ (2\ell-1)!!}  \Bigg |1- 
\sum_{l=1}^\infty  \frac{ k^{2l} (\beta_{\ell,l}^2+a_{\ell,l }^2)^{l }}{ 2^l l! N_{\ell,l} }   \Bigg |   \nonumber   \\
&
	=\Oh(  \|g_j\|_{L^2({\mathbb S}^{2})} s^{-\alpha -3 } ) , \notag
\end{align}
where  $\beta_\ell, \beta_{\ell,l} \in [0,h]$ such that  $k^{2}(\beta_{\ell}^2+a_{\ell}^2) \leq k^{2} (h^2+L^2) <1$ and $k^2(\beta_{\ell,l}^2+a_{\ell,l}^2) \leq k^2 (h^2+L^2) <1$ for sufficiently  small $h$ and $L$, by utilizing the claim that
$$
|\gamma_{\ell j} | \leq \|g_j\|_{L^2({\mathbb S}^{2}) },
$$
 where we use the fact that  $|P_\ell(t)| \leq 1$ when $|t|\leq 1$. Consequently, \eqref{eq:336 bound} can be derived.

For $I_{312}^-$, we can deduce that
 \begin{align}\notag
 	\left| { I_{312}^-}\right| &\leq     |C( \psi )|  \sum_{l=1}^\infty  \frac{ k^{2l}} { (2l+1)!! }  \sum_{i_1=1}^l  C(l,i_1)h^{2 (i_1-1) } L ^{2(l-i_1) }\int_{0}^h r^2   e^{-\sqrt{sr} \omega(\theta_m)}   {\rm d} r\\
 	&=   |C( \psi )|  \sum_{l=1}^\infty  \frac{ k^{2l} } { (2l+1)!! h^2 }  \sum_{i_1=1}^l  C(l,i_1)h^{2 i_1 } L ^{2(l-i_1) }\int_{0}^h r^2   e^{-\sqrt{sr} \omega(\theta_m)}   {\rm d} r \nonumber \\
 	&=   |C( \psi )|  \sum_{l=1}^\infty  \frac{ k^{2l}} { (2l+1)!! h^2 }  ((h^2+L^2)^l-L^{2l} ) \int_{0}^h r^2   e^{-\sqrt{sr} \omega(\theta_m)}   {\rm d} r \nonumber \\
 	&\leq 2L   \|\psi  \|_{L^\infty}   \sum_{l=1}^\infty  \frac{ k^{2l}} { (2l+1)!! h^2 }  ((h^2+L^2)^l-L^{2l} ) \cdot   \Oh(s^{-3})  \nonumber \\
 	&=\Oh(s^{-3}), \nonumber
 \end{align}
 where we choose $h$ and $L$ such that $ k^{2}( h^2+L^2)<1$ and $k L<1$. 
 
 Substituting the expansion \eqref{eq:jell} of $j_\ell $ into $I_{32}^- $, we have  
 \begin{align}\notag
 	| I_{32}^-|&\leq C_1(\psi ) \|g_j\|_{L^2 ({\mathbb S} ^{2})}  \nonumber \\
	& \quad \times \sum_{\ell =1}^\infty \int_0^h r^2 e^{-\sqrt{sr} \omega(\theta_m) } \frac{  k^\ell (|r|^2+a_{\ell }^2)^{(\ell-1)/2 } }{ (2\ell- 1)!!}   \Bigg| 1-\sum_{l=1}^\infty  \frac{(-1)^l k^{2l} (|r|^2+a_{\ell,l }^2)^{ l }}{ 2^l l! N_{\ell,l}  }   \Bigg| {\rm d} r\nonumber\\
 	&=C_1(\psi)  \|g_j\|_{L^2 ({\mathbb S} ^{2})}  \sum_{\ell =1}^\infty \frac{  k^\ell (|\beta_\ell|^2+a_{\ell }^2)^{(\ell-1)/2 } }{ (2\ell- 1)!!} \Bigg| 1-\sum_{l=1}^\infty  \frac{(-1)^l k^{2l} (|\beta_{\ell,l}|^2+a_{\ell,l }^2)^{ l }}{ 2^l l! N_{\ell,l}}   \Bigg| \nonumber  \\
 	&\quad \times   \int_0^h r^2 e^{-\sqrt{sr} \omega(\theta_m) }  {\rm d} r \nonumber \\
% 	&=2 L \|\psi\|_\infty \|g_j\|_{L^2 ({\mathbb S} ^{n-1})}  \sum_{\ell =1}^\infty \frac{  k^\ell (|\beta_\ell|^2+a_{\ell }^2)^{(\ell-1)/2 } }{ (2\ell- 1)!!} \Bigg| 1-\sum_{l=1}^\infty  \frac{(-1)^l k^{2l} (|\beta_{\ell,l}|^2+a_{\ell,l }^2)^{2l }}{ 2^l l! (2\ell+3)\cdots  (2\ell+2l+1) }   \Bigg| \nonumber  \\
% 	&\quad \times   \left( \int_0^h r e^{-\sqrt{sr} \omega(\theta_m) }  {\rm d} r  + \Oh \left( \int_0^h r^3 e^{-\sqrt{sr} \omega(\theta_m) }  {\rm d} r \right) \right)  \nonumber \\
 	 	&=\Oh (\|g_j\|_{L^2 ({\mathbb S} ^{2})} s^{-3}), \notag
 \end{align}
 where $\beta_\ell, \beta_{\ell, l} \in [0,h]$ such that $k^{2} (\beta_\ell^2+a_\ell^2) \leq k^{2} (h^2+L^2)<1$ and $k^2(|\beta_{\ell,l}|^2+a_{\ell,l}^2 )\leq k^2( h^2+L^2) <1$ for sufficiently  small $h$ and $L$. 
 
 The asymptotic analysis \eqref{eq:338 bounds} for $I_{312}^+$, $I_{32}^+$, $I_{\eta,1}^+$  and $I_{\eta,2}^+$  can be analyzed in a similar way, which is omitted. 
\end{proof}

In the next proposition, we deduce the lower and upper bounds for $C(I_{311}^-)$ and $C(I_{311}^+)$, where $C(I_{311}^-)$ and $C(I_{311}^+)$ are defined in  \eqref{eq:335 I32} and  \eqref{eq:CI311+}, respectively, which shall be used to prove Lemma \ref{lem:37 coeff} in the following. 
\begin{proposition}\label{pro:31}
	Let $C(I_{311}^-)$ and $C(I_{311}^+)$ be defined in  \eqref{eq:335 I32} and  \eqref{eq:CI311+}, respectively.  Assume that the condition \eqref{eq:lem36 cond} is fulfilled for a succificiently small $L\in \mathbb R_+$,  then
 \begin{equation}\label{eq:341}
 	 0<	C(I_{311}^-)\leq \frac{C(\psi )}{1-(kL)^2}, \quad  0<	C(I_{311}^+)\leq \frac{C(\psi )}{1-(kL)^2}, 
 \end{equation}
  where $C(\psi )$ is defined in \eqref{eq:339}. 
\end{proposition}
\begin{proof}
Recall that  $|a_{0,l }|\leq L$ is defined in \eqref{eq:j0}. By \eqref{eq:lem36 cond}, we know that
\begin{equation}\label{eq:336}
	\left| \sum_{l=1}^\infty  \frac{(-1)^l k^{2l}} { (2l+1)!! }  a_{0,l} ^{2l} \right|  \leq \sum_{l=1}^\infty (k L)^{2l}=\frac{(k L)^2}{1-(k L)^2}.
\end{equation}
From \eqref{eq:336}, we know that \begin{equation}\label{eq:339}
  0<\frac{C(\psi )(1-2(kL)^2)}{1-(kL)^2}\leq 	C(I_{311}^-)\leq \frac{C(\psi )}{1-(kL)^2}, 
 \end{equation}
 where  $C(\psi ) = \int_{-L}^L \psi(x_3) {\rm d} x_3>0$, since  $\psi \not\equiv 0$ is a nonnegative function .

  For sufficiently small $L$,   similar to \eqref{eq:339}, we know that
 \begin{equation}\label{add6}
  0<\frac{C(\psi )(1-2(kL)^2)}{1-(kL)^2}\leq 	C(I_{311}^+)\leq \frac{C(\psi )}{1-(kL)^2}, 
 \end{equation}
 where $C(\psi )$ is defined in \eqref{eq:339}. 
\end{proof}

\begin{lemma}\label{lem:37 coeff}
Let $\theta_m$ and $\theta_M$ be defined in \eqref{eq:sh}. Assume that $\theta_m$ and $\theta_M$  fulfil the condition \eqref{eq:lem 29 cond}, and moreover \eqref{eq:lem36 cond} is satisfied for a sufficiently small $L\in \mathbb R_+$, 
	then 
\begin{equation}\label{eq:348}
	 C(I_{311}^-)\mu(\theta_m)^{-2}+C(I_{311}^+)\mu(\theta_M)^{-2} \neq 0,
\end{equation}
 where $C(I_{311}^-)$ and $C(I_{311}^+)$ are defined in  \eqref{eq:335 I32} and  \eqref{eq:CI311+}, respectively. 
\end{lemma}
\begin{proof}
It can be calculated that
\begin{align*}
	&C(I_{311}^-) \mu(\theta_m)^{-2}+C(I_{311}^+) \mu(\theta_M)^{-2}\\
	&=\frac{(C(I_{311}^+)\cos \theta_m+C(I_{311}^-)\cos \theta_M)+i (C(I_{311}^+)\sin \theta_m+C(I_{311}^-)\sin \theta_M )}{ (\cos \theta_m+i \sin \theta_m )(\cos \theta_M+i \sin \theta_M )}.
\end{align*}
Therefore, under  the assumption \eqref{eq:lem 29 cond}, we know that
$$
\cos \theta_m+\cos \theta_M \mbox{ and } \sin \theta_m+\sin \theta_M
$$
can not be zero simultaneously. Without loss of generality, we assume that $\cos \theta_m+\cos \theta_M \neq   0$.   Then we consider the following two cases:
\begin{itemize}
	\item[(i)] Case A: $\cos \theta_m+\cos \theta_M >  0$,
	\item[(ii)] Case B: $\cos \theta_m+\cos \theta_M < 0$. 
\end{itemize}
For  Case A, we first consider that $\cos \theta_m$ and $\cos \theta_M$ have the same sign.   From  \eqref{eq:341} in Proposition \ref{pro:31}, it is not difficult to see that the real  part of  the {\bl numerator} of $C(I_{311}^-)\mu(\theta_m)^{-2}+C(I_{311}^+)\mu(\theta_M)^{-2} $  can not be zero. Therefore, \eqref{eq:348} is proved when $\cos \theta_m$ and $\cos \theta_M$ have the same sign.

  In the following, we assume that $\cos \theta_m$ and $\cos \theta_M$ have different signs. Then it implies that $\cos \theta_m \leq  0$ and $\cos \theta_M >  0$.  From \eqref{eq:339} and \eqref{add6}, we can deduce that 
\begin{align*}%\label{eq:upper bound}
	\frac{C(\psi )}{1-(kL)^2}
	(  \cos \theta_m+   (1-2(kL)^2)   \cos \theta_M  )  
 & \leq C(I_{311}^+)\cos \theta_m +C(I_{311}^-)\cos \theta_M \\
 &\leq \frac{C(\psi )}{1-(kL)^2}
	( (1-2(kL)^2)  \cos \theta_m+ \cos \theta_M  ) .  \nonumber 
\end{align*}
Since $L$ is flexible, for a given $0< \varepsilon   <1$, we can choose $L$ such that  $0<k L<\sqrt { \varepsilon /2 }$, from which  we can derive the bounds as follows
\begin{align}\label{eq:351}
	\frac{C(\psi )}{1-(kL)^2}
\left 	( \cos \theta_m  +  (1- \varepsilon )    \cos \theta_M  \right  )  
 & \leq C(I_{311}^+)\cos \theta_m  +C(I_{311}^-) \cos \theta_M \nonumber \\
 & \leq \frac{C(\psi )}{1-(kL)^2 }
	\left ( (1- \varepsilon ) \cos \theta_m  +  \cos \theta_M \right   ) .  
\end{align}
Since  $\cos \theta_m+\cos \theta_M > 0$, we can consider the lower bound in  \eqref {eq:351}. Denote $\varepsilon_0=\min \{\frac{\cos \theta_m+\cos \theta_M }{2\cos \theta_M}, 1\}$ and choose $\varepsilon\in (0, \varepsilon_0)$. It can be verified that 
\begin{align*}
C(I_{311}^+)\cos \theta_m +C(I_{311}^-) \cos \theta_M \geq \frac{C(\psi )}{1-(kL)^2 }
\left ( \cos \theta_m+ (1-\varepsilon)\cos \theta_M \right   ) >0, 
\end{align*}
 which means that \eqref{eq:348} still holds.

For Case B, if $\cos \theta_m=0$ or $\cos \theta_M=0$ is satisfied,  from the upper bound of \eqref{eq:351} we can easily show that 
\begin{equation}\label{eq:boundtemp}
	C(I_{311}^+)\cos \theta_m +C(I_{311}^-) \cos \theta_M <0. 
\end{equation}
Otherwise, if $|\cos \theta_m  | \leq |\cos \theta_M |$, from the fact that $(1- \varepsilon  ) |\cos \theta_m  | \leq |\cos \theta_M |$, we know that \eqref{eq:boundtemp} still holds  from the upper bound of \eqref{eq:351}.  If $|\cos \theta_m  | > |\cos \theta_M |$, we can choose $\varepsilon$ such that $ \varepsilon >1- |\cos \theta_M | / |\cos \theta_m | >0 $ to make \eqref{eq:boundtemp} also be fufilled  from the upper bound of \eqref{eq:351}. Therefore, for Case B, we know that \eqref{eq:348} is  fulfilled. 

The proof is complete. 
\end{proof}

We are in a position to present another main result of this paper on the vanishing of conductive transmission eigenfunctions at an edge corner in 3D.} 

\begin{theorem}\label{Th:3.1}
{\bl Let $\Omega \Subset \mathbb R^3$ be a bounded Lipschitz domain with $0\in \partial \Omega$ and} ${S_h}\subset \R^{2}$ be defined in \eqref{eq:sh}, $M>0$, $0<\alpha<1$. For any fixed {\bl $x_3^c \in (-M,M)$} and $L>0$ defined in Definition \ref{Def}, we suppose that  $L $ is sufficiently small such that $(x_3^c-L,x_3^c+L) \subset (-M,M) $ {\bl and
$$
(B_h  \times (-M,M) ) \cap \Omega=S_h \times (-M,M),
$$
where $B_h\Subset \mathbb R^2 $ is  the central ball of radius $h \in \mathbb R_+$. Assume that $v,w\in H^1(\Omega )$ are the transmission eigenfunctions to \eqref{eq:in eig}}  and there exists a sufficiently  small neighbourhood {\bl $S_h \times (-M,M)$ (i.e. $h>0$ is sufficiently  small) of $(0,x_3^c)$ with $x_3^c\in (-M,M)$} such that $qw \in C^{\alpha}(\overline {S}_h \times [-M,M] )$ and $\eta \in C^\alpha(\overline{\Gamma}_h^\pm  \times [-M,M]  )$   for $0<\alpha<1$,   and  $v-w\in H^2({S_h}\times (-M,M) )$,    where   $0$ is the vertex of  $S_h$  defined in \eqref{eq:sh}. Write {\bl $x=(x', x_3) \in \R^{3}, \, x'\in \R^{2}$}.  If the following conditions are fulfilled: 
	\begin{itemize}
		\item[(a)] the transmission eigenfunction $v$ can be approximated in $H^1(S_h \times (-M,M))$ by the Herglotz functions $v_j$, $j=1,2,\ldots$,  with kernels $g_j$ satisfying the approximation property  \eqref{3eq:ass1}, where the parameter $\varrho$ in \eqref{3eq:ass1} fulfills that $0<\varrho<\alpha$, 
%	\begin{equation}\label{3eq:ass1}
%		\|v-v_j\|_{H^1(S_h \times (-M,M))} \leq j^{-1-\Upsilon},\quad  \|g_j\|_{L^2({\mathbb S}^{n-1})} \leq C j^{{1+\varrho}}, 
%	\end{equation}
%	for some positive constant $C$, $\Upsilon>0$ and ${0< \varrho<\alpha}   $,
	\item[(b)] the function $\eta=\eta(x')$ is independent  of $x_3$ and  
	\begin{equation}\label{3eq:ass2}
		\eta(0) \neq 0, 
	\end{equation}
	\item[(c)]  the angles $\theta_m$ and $\theta_M$ of  $S_h$ satisfy 
	\begin{equation}\label{3eq:ass3} %($-\pi < \theta_m < \theta_M < \pi$) 
		-\pi < \theta_m < \theta_M < \pi \mbox{ and } \theta_M-\theta_m \neq \pi,
	\end{equation}
%	\begin{equation}\label{3eq:ass3}
%		\theta_M-  	\theta_m <\pi \mbox{ or } \theta_M-  	\theta_m >\frac{3\pi}{2}, 
%	\end{equation}
	\end{itemize}
then for every edge point  {\bl $(0,x_3^c) \in \R^3 $} of $S_h \times (-M,M)$ where {\bl $x_3^c\in (-M,M)$}, one has
{\bl 
\[
%\begin{equation}\label{eq:36}
	\lim_{ \rho \rightarrow +0 }\frac{1}{m(B((0,x_3^c), \rho  )\cap \Omega ) } \int_{m(B((0, x_3^c), \rho  ) \cap \Omega ) } |v(x)| {\rm d} x=0,
%\end{equation}
\]
}	where $m(B((0,x_3^c), \rho  )\cap\Omega)$ is the volume of $B((0,x_3^c) ,\rho )\cap \Omega$.  

\end{theorem}

\begin{proof}  
One can easily see that the transmission eigenfunctions $v,w\in H^1(\Omega )$   to \eqref{eq:in eig} fulfill the PDE system \eqref{eq:3d in eig}. Using the dimensional reduction operator $\mathcal R$ given by \eqref{dim redu op}, we can show that $\mathcal R(v)$ and $\mathcal R(w)$ satisfy \eqref{eq:eig reduc}. By virtue of Lemma \ref{lem:32} and Lemma \ref{lem:int2}, we know that the integral equality \eqref{3eq:int identy} holds under the assumption that $qw \in C^{\alpha}(\overline {S}_h \times [-M,M] )$  for $0<\alpha<1$   and  $v-w\in H^2({S_h}\times (-M,M) )$.  Since $\eta \in C^\alpha(\overline{\Gamma}_h^\pm  \times [-M,M]  )$  for $0<\alpha<1$, from Lemma \ref{lem:36}, we can obtain \eqref{3eq:I2-final} and \eqref{3eq:I2+final}. Recall that $\mu(\theta_m )$ and  $\mu(\theta_M)$ are defined in \eqref{eq:omegamu}.  Therefore, substituting \eqref{3eq:I2-final} and \eqref{3eq:I2+final} into  \eqref{3eq:int identy}, after rearranging terms and multiplying $s$ on the both sides of \eqref{3eq:int identy},  we deduce that
\begin{equation}
\begin{split}\label{3eq:45}
&	2v_j(0)\eta(0)\Big[ \left( \mu(\theta_M )^{-2}-   \mu(\theta_M )^{-2} e^{ -\sqrt{sh}  \mu(\theta_M ) } -  \mu(\theta_M )^{-1} \sqrt{sh} e^{ -\sqrt{sh} \mu(\theta_M ) }  \right  )C(I_{311}^+) \\
	&\quad + \left( \mu(\theta_m )^{-2}-   \mu(\theta_m )^{-2} e^{ -\sqrt{sh} \mu(\theta_m ) } -  \mu(\theta_m )^{-1} \sqrt{sh}  e^{ -\sqrt{sh}  \mu(\theta_m ) } \right  ) C(I_{311}^-)\Big]  \\
	&=s\Big[ I_3-(F_{1 } (0)+F_{2 } (0)+F_{3j } (0)) \int_{S_h} u_0(sx') {\rm d} x' - \delta_j(s)  \\
	&\quad - \eta(0) (I_{32}^+ +  I_{32}^- ) -I_\eta^+  - I_\eta^- -\int_{S_h}  \delta  F_{1 } (x') u_0(sx') {\rm d} x' -\int_{S_h} \delta  F_{2 } (x')u_0(sx') {\rm d} x' \\
	&\quad -\int_{S_h} \delta  F_{3j } (x')u_0(sx') {\rm d} x'- v_j(0) \eta(0)\left( I_{312}^-+ I_{312}^+ \right) - \epsilon_j^\pm(s)\Big], 
\end{split}
\end{equation}
where $C(I_{311}^-)$ and $C(I_{311}^+)$ are defined in  \eqref{eq:335 I32} and  \eqref{eq:CI311+}, respectively. Here $\delta F_{3j}$, $\delta F_1$ and  $\delta F_2$ are given by \eqref{eq:F3j} and \eqref{eq:326F3j}.  

When $s=j$,  under the assumption \eqref{3eq:ass1},  using \eqref{eq:I312}, \eqref{3eq:I32}, \eqref{eq:336 bound} and \eqref{eq:338 bounds} in Lemma \ref{lem:36} we know that
\begin{equation}
\begin{split}\label{3eq:46}
&j |  I_{32}^-|  \leq \Oh( j^{-2} \|g_j\|_{L^2({\mathbb S}^2 )} ) \leq \Oh( j^{-1+\varrho}  ) ,\quad {  j |  I_{32}^+ |  \leq \Oh( j^{-2} \|g_j\|_{L^2({\mathbb S}^2  )} )} \leq \Oh( j^{-1+\varrho}  ) ,  \\
& j  I_{312}^- \leq \Oh(j^{-2}),\quad j I_{312}^+ \leq \Oh(j^{-2}),
\end{split}
\end{equation}
and
\begin{equation}\label{3eq:46 a}
\begin{split}	
&j |I_\eta^-| \leq  \|\eta \|_{C^\alpha } \left( |v_j(0)| \Oh (j^{-\alpha })+ \Oh (\|g_j\|_{L^2({\mathbb S}^2 )} j^{-{2}-\alpha } ) \right) \\
&\hspace{0.8cm}\leq  \|\eta \|_{C^\alpha } \left( |v_j(0)| \Oh (j^{-\alpha })+ \Oh (j^{-{1}-(\alpha -\varrho ) } ) \right),   \\
& j|I_\eta^+|  \leq  \|\eta \|_{C^\alpha } \left(| v_j(0)| \Oh (j^{-\alpha })+ \Oh (\|g_j\|_{L^2({\mathbb S}^2 )} j^{-{2}-\alpha } ) \right), \\
&\hspace{0.8cm}\leq  \|\eta \|_{C^\alpha } \left( |v_j(0)| \Oh (j^{-\alpha })+ \Oh (j^{-{1}-(\alpha -\varrho ) } ) \right),   
\end{split}
\end{equation}
as $j\rightarrow +\infty$.  
Clearly, when $s=j$, under the assumption \eqref{3eq:ass1}, from \eqref{eq:u0w}, \eqref{eq:1.5},   \eqref{3eq:deltajnew3} and \eqref{3eq:xij} in Lemma \ref{lem:3 delta},   \eqref{3eq:I3},  \eqref{3eq:deltaf1j} and \eqref{3eq:deltaf2} in Lemma \ref{lem:int2}, we can derive that
\begin{equation}
\begin{split}
\label{3eq:47}
	&j |I_3| \leq  C j e^{-c' \sqrt j},\ j \left| \int_{S_h} u_0(jx) {\rm d} x \right| \leq \frac{ 6 |e^{-2\theta_M i }-e^{-2\theta_m i }  | } {j} +  \frac{6(\theta_M-\theta_m )e^{-\delta_W \sqrt{h j}/2}}{j\delta_W^4}  ,\\
&j \left|\int_{S_h}  \delta F_{3j  } (x) u_0(sx') {\rm d} x' \right | \leq \frac{8 L\sqrt{\pi}\|\psi\|_{L^\infty}(\theta_M- \theta_m) \Gamma(2 \alpha+4) }{ \delta_W^{2\alpha+4   }  } k^2 {\rm diam}(S_h)^{1-\alpha } \\
&\hspace{5cm} \times   (1+k) \|g_j\|_{L^2( {\mathbb S}^{2})}   j^{-\alpha-1   } \leq \Oh \left (j^{-(\alpha-\varrho )} \right  ),   \\
	& j \left|\int_{S_h}  \delta F_{1 } (x) u_0(sx') {\rm d} x' \right | \leq  \frac{2\|F_{1}\|_{C^\alpha } (\theta_M-\theta_m )\Gamma(2\alpha+4) }{ \delta_W^{2\alpha+4}} j^{-\alpha-1}  ,  \\
	&j	\left|\int_{S_h}  \delta F_{2 } (x) u_0(sx') {\rm d} x' \right | \leq \frac{2\|F_{2}\|_{C^\alpha } (\theta_M-\theta_m )\Gamma(2\alpha+4) }{ \delta_W^{2\alpha+4}} j^{-\alpha-1},
	\end{split}
\end{equation}
and 
	\begin{equation}\label{3eq:47 a}
\begin{split}
	&j |\epsilon_j^\pm (j) |\leq C \|\psi\|_{L^\infty} \left( |\eta(0)|  \frac{\sqrt { \theta_M-\theta_m }  e^{-\sqrt{j \Theta } \delta_W } h } {\sqrt 2 } j \right. \\
	&\left. \hspace{4cm} + \|\eta\|_{C^\alpha } j^{-\alpha } \frac{\sqrt{2(\theta_M-\theta_m) \Gamma(4\alpha+4) } }{(2\delta_W)^{2\alpha+2  } } \right) j^{-1-\Upsilon } ,  \\
	&j|\delta_j(j)| \leq   \frac{ k^2 \|\psi \|_{L^\infty}  \sqrt{C(L,h)  ( \theta_M-\theta_m )}  e^{-\sqrt{s \Theta } \delta_W } h  }{\sqrt 2 } j^{-\Upsilon },  \quad  \Theta  \in [0,h ], 
\end{split}
\end{equation}
as $j\rightarrow +\infty$, where $c'>0$ and $\delta_W$ are defined in \eqref{3eq:I3} and \eqref{eq:xalpha}, respectively.

The coefficient of $v_j(0)$ of \eqref{3eq:45} with respect to the zeroth order of $s$ is 
$$
2 \eta(0)\left  (C(I_{311}^-)\mu(\theta_m)^{-2}+C(I_{311}^+)\mu(\theta_M)^{-2} \right ). 
$$

% from \eqref{eq:351},  we can conclude that  
%$$
%C(I_{311}^+)\cos \theta_m +C(I_{311}^-) \cos \theta_M <  \frac{C(\psi )}{1-(kL)^2 }
%	\left ( (1- \varepsilon ) \cos \theta_m  +  \cos \theta_M \right   ) <0. 
%$$ 
%Hence \eqref{eq:348} is still fulfilled. 

 In \eqref{3eq:45}, we take $s=j$ and let $j\rightarrow \infty$, combining with \eqref{3eq:46}, \eqref{3eq:46 a},  \eqref{3eq:47} and \eqref{3eq:47 a}, we can prove that
 \begin{equation}\label{eq:3v(0)}
\lim_{j \rightarrow \infty} \eta(0)\left(C(I_{311}^-)\mu(\theta_m)^{-2}+C(I_{311}^+)\mu(\theta_M)^{-2} \right)  v_j(0)=0. 
 \end{equation}
Under the assumption \eqref{3eq:ass3},  from Lemma \ref{lem:37 coeff}, we have  $C(I_{311}^-)\mu(\theta_m)^{-2}+C(I_{311}^+)\mu(\theta_M)^{-2} \neq 0$. Therefore, from \eqref{3eq:ass2} and \eqref{eq:3v(0)}, we prove that 
$$
\lim_{j \rightarrow \infty} v_j(0)=0. 
$$
Using the a similar argument of \eqref{eq:250}, we finish the proof of this theorem. 
%Using the fact that
%\begin{align*}
%	\lim_{ \rho \rightarrow +0 }\frac{1}{m(B(x_c, \rho  ))} \int_{B(x_c, \rho )} |v(x)| {\rm d} x &\leq \lim_{j \rightarrow \infty}  \Big( \lim_{ \rho \rightarrow +0 }\frac{1}{m(B(x_c, \rho  ))} \int_{B(x_c, \rho )} |v(x)-v_j(x)| {\rm d} x\\
%	&\quad +\lim_{ \rho \rightarrow +0 }\frac{1}{m(B(x_c, \rho  ))} \int_{B(x_c, \rho )} |v_j(x)| {\rm d} x\Big) 
%\end{align*}
%we finish the proof of this theorem. \qed 
\end{proof}

{
\begin{remark}
Similar to Remark~\ref{rem:th1.1}, Theorem~\ref{Th:3.1} can be localized. Moreover, we would like to mention that in contrast to the regularity assumption on $v-w$ near the corner in 2D of Theorem \ref{Th:1.1}, we impose that $v-w \in H^2(S_h \times(-M,M))$ in Theorem \ref{Th:3.1}, where we need to use the $C^\alpha$-continuity of $\mathcal R(v-w)$ to investigate the asymptotical order of $s$ with respect to $s \rightarrow \infty $ for the volume integral of $F_1(x')$ over $S_h$ in \eqref{3eq:int identy}. 
\end{remark}
}

Similar to Corollary \ref{cor:2.1}, we consider the vanishing property of the interior transmission eigenfunctions $v \in H^1(W \times (-M,M))$ and $w \in H^1(W \times (-M,M)) $ to \eqref{eq:in eig reduce} on the edge point under the assumptions \eqref{3eq:ass3} and \eqref{3eq:ass1 int}.  

%In the next corollary, we will consider the case $ \eta(x) \equiv 0 $ near the corner, which means that \eqref{eq:in eig} degenerates to \eqref{eq:in eig reduce} near the corner. The interior transmission eigenfunctions $v \in H^1(\Omega )$ and $w\in H^1(\Omega ) $ to \eqref{eq:in eig reduce} still have the vanishing property near the corner provided that $v$ can be approximated by  the Herglotz waves $v_j$ under the assumption \eqref{eq:ass1 int} and  the interior angle  of the sector $W$ containing the corner is not $\pi$. 
\begin{corollary}\label{cor:3.1}
	{\bl Let $\Omega \Subset \mathbb R^3$ be a bounded Lipschitz domain with $0\in \partial \Omega$ and} ${S_h}\subset \R^{2}$ be defined in \eqref{eq:sh}, $M>0$, $0<\alpha<1$. For any fixed {\bl $x_3^c \in (-M,M)$} and $L>0$ defined in Definition \ref{Def}, we suppose that  $L $ is sufficiently small such that $(x_3^c-L,x_3^c+L) \subset (-M,M) $ {\bl and
$$
(B_h \times (-M,M) ) \cap \Omega=S_h \times (-M,M),
$$
where $B_h\Subset \mathbb R^2 $ is  the central ball of radius $h \in \mathbb R_+$.} Suppose {\bl  $v \in H^1(\Omega )$ and $w \in H^1(\Omega ) $ are}  the interior transmission eigenfunctions  to \eqref{eq:in eig reduce} in $\R^3$.   Suppose that   there exists a sufficiently  small neighbourhood {\bl $S_h \times (-M,M)$ (i.e. $h>0$ is sufficiently  small) of $(0,x_3^c)$ with $x_3^c\in (-M,M)$} such that $qw \in C^{\alpha}(\overline {S_h} \times [-M,M] )$  for $0< \alpha  <1$,  and $v-w\in H^2({S_h}\times(-M,M))$,  where $0 $ is the vertex of  $S_h$  defined in \eqref{eq:sh}.    If the following conditions are fulfilled: 
	\begin{itemize}
		\item [(a)] the transmission eigenfunction $v$ can be approximated in $H^1(S_h \times (-M,M))$ by the Herglotz waves $v_j$, $j=1,2,\ldots$,  with kernels $g_j$ satisfying 
	\begin{equation}\label{3eq:ass1 int}
		\|v-v_j\|_{H^1(S_h \times (-M,M) )} \leq j^{-2-\Upsilon},\quad  \|g_j\|_{L^2({\mathbb S}^{2})} \leq C j^{\varrho}, 
	\end{equation}
	for some positive constants $C$, $\Upsilon >0$ and $0< \varrho<\alpha  $,
	\item[(b)]  the angles $\theta_m$ and $\theta_M$ of $S_h$ satisfy 
	\begin{equation}\label{3eq:ass2 int} %($-\pi < \theta_m < \theta_M < \pi$) 
		-\pi < \theta_m < \theta_M < \pi \mbox{ and } \theta_M-\theta_m \neq \pi,
	\end{equation}
	\end{itemize}
	then we have
	$$
	{\bl 
	\lim_{ \rho \rightarrow +0 }\frac{1}{m(B((0,x_3^c), \rho  )\cap P(\Omega) )} \int_{B((0,x_3^c), \rho ) \cap P(\Omega)} {\mathcal R}(V w)(x')  {\rm d} x'=0,} 
	$$
  where {\bl $ P(\Omega )$ is the projection set of $\Omega$ on $\mathbb R^2$ and } $q(x',x_3)=1+V(x',x_3)$. 
\end{corollary}
\begin{proof} {\bl It is clear that the transmission eigenfunctions $v \in H^1(\Omega )$ and $w \in H^1(\Omega )$ to \eqref{eq:in eig reduce} fulfill \eqref{eq:3d in eig}   for $\eta \equiv 0$. Using the dimensional reduction operator $\mathcal R$ given by \eqref{dim redu op}, we can show that $\mathcal R(v)$ and $\mathcal R(w)$ satisfy \eqref{eq:eig reduc}  for $\eta \equiv 0$. By virtue of Lemmas \ref{lem:32} and \ref{lem:int2},  due to} $ \eta(x) \equiv 0 $, from \eqref{3eq:int identy} we have the following integral equality
\begin{align}\label{int3eq:int identy}
	&(F_{1 } (0)+F_{2 } (0)+F_{3j } (0)) \int_{S_h} u_0(sx') {\rm d} x'+\delta_j (s)\\
	 &= I_3-\int_{S_h}  \delta  F_{1} (x') u_0(sx') {\rm d} x' - \int_{S_h} \delta  F_{2} (x')u_0(sx') {\rm d} x' -\int_{S_h} \delta  F_{3j} (x')u_0(sx') {\rm d} x' .  \notag 
\end{align}
where $\delta_j(s)$ is defined in \eqref{eq:313delta},  $\delta F_1(x')$,  $\delta F_2(x')$ and  $\delta F_{3j}(x')$ are defined in \eqref{eq:326F3j}, $I_3$ is given in  \eqref{3eq:int identy}. Since $v=w$ on $\Gamma_h^\pm \times (-M,M)$, it is easy to see that
$$
F_{1} (0)=
	\int_{-L}^{L}\psi''(x_3)(v(0, x_3)-w(0, x_3)) {\rm d} x_3=0. 
$$
Therefore, using \eqref{eq:u0w}, from \eqref{int3eq:int identy}, we deduce that
\begin{align}\label{int3eq:int identy1}
	& 6 i  (F_{2 } (0)+F_{3j } (0)) (e^{-2\theta_M i }-e^{-2\theta_m i }  ) s^{-2}- (F_{2 } (0)+F_{3j } (0)) \int_{W \backslash S_h} u_0(sx') {\rm d} x'\\
	 &= I_3-\int_{S_h}  \delta  F_{1} (x') u_0(sx') {\rm d} x' - \int_{S_h} \delta  F_{2} (x')u_0(sx') {\rm d} x' -\int_{S_h} \delta  F_{3j} (x')u_0(sx') {\rm d} x' -\delta_j (s) .  \notag 
\end{align}
In \eqref{int3eq:int identy1}, we take $s=j$ and multiply $j^2$ on the both sides of \eqref{int3eq:int identy1}, then we have
\begin{align}\label{int3eq:int identy2}
	& 6 i  (F_{2 } (0)+F_{3j } (0)) (e^{-2\theta_M i }-e^{-2\theta_m i }  ) =j^2\Big [ I_3 + (F_{2 } (0)+F_{3j } (0)) \int_{W \backslash S_h} u_0(sx') {\rm d} x' \notag \\
	 &\quad -\int_{S_h}  \delta  F_{1} (x') u_0(sx') {\rm d} x' - \int_{S_h} \delta  F_{2} (x')u_0(sx') {\rm d} x' -\int_{S_h} \delta  F_{3j} (x')u_0(sx') {\rm d} x' -\delta_j (s) \Big].  
\end{align}

{\bl Using  \eqref{3eq:deltajnew3} in Lemma \ref{lem:3 delta} and}  under the assumption \eqref{3eq:ass1 int}, it is easy to see that
\begin{equation}\label{eq:deltajnew int}
	j^2 |\delta_j(s)| \leq    \frac{ |k|^2 \|\psi \|_{L^\infty}  \sqrt{C(L,h)  ( \theta_M-\theta_m )}  e^{-\sqrt{s \Theta } \delta_W } h  }{\sqrt 2 }   j^{-\Upsilon }, 
\end{equation} 
where  $C(L,h)$ is a positive number defined in \eqref{eq:312},  $ \Theta  \in [0,h ]$ and $\delta_W$ is defined in \eqref{eq:xalpha}.  

Under the assumption \eqref{3eq:ass1 int}, by virtue of  \eqref{3eq:deltaf1j}, \eqref{3eq:deltaf1}, \eqref{3eq:deltaf2} and \eqref{3eq:I3} {\bl in Lemma \ref{lem:int2},}  we can obtain the following estimates
\begin{align}\label{eq:360est}
	&j^2 \left|\int_{S_h}  \delta F_{3j  } (x) u_0(sx') {\rm d} x' \right | \leq \frac{8 L\sqrt{\pi}\|\psi\|_{L^\infty}(\theta_M- \theta_m) \Gamma(2 \alpha+4) }{ \delta_W^{2\alpha+4   }  } k^2 {\rm diam}(S_h)^{1-\alpha }\nonumber \\
&\hspace{5cm} \times   (1+k) \|g_j\|_{L^2( {\mathbb S}^{2})}   j^{-\alpha  } \leq \Oh \left (j^{-(\alpha-\varrho )} \right  ),  \nonumber \\
	& j^2 \left|\int_{S_h}  \delta F_{1 } (x) u_0(sx') {\rm d} x' \right | \leq  \frac{2\|F_{1}\|_{C^\alpha } (\theta_M-\theta_m )\Gamma(2\alpha+4) }{ \delta_W^{2\alpha+4}} j^{-\alpha}  , \nonumber \\
	&j^2	\left|\int_{S_h}  \delta F_{2 } (x) u_0(sx') {\rm d} x' \right | \leq \frac{2\|F_{2}\|_{C^\alpha } (\theta_M-\theta_m )\Gamma(2\alpha+4) }{ \delta_W^{2\alpha+4}} j^{-\alpha}.
\end{align}

%we take $s=j$.  Since \eqref{eq:u0w}, multiplying $j^2$ on both sides of \eqref{int3eq:int identy}, using the  assumptions \eqref{3eq:ass1 int} and \eqref{3eq:ass3}, let $j \rightarrow \infty$, recalling that $F_2$ and $F_{3j}$ are given in \eqref{eq:F1xi} 
%

Under the  assumption \eqref{3eq:ass2 int}, it is easy to see that
$$
\left| e^{-2\theta_M i }-e^{-2\theta_m i }  \right| =\left|1-e^{-2(\theta_M -\theta_m) i } \right| \neq 0,
$$
since $\theta_M -\theta_m 
\neq \pi$. 
In \eqref{int3eq:int identy2}, by letting $j\rightarrow \infty$, from \eqref{eq:1.5}, \eqref{eq:deltajnew int} and  \eqref{eq:360est}, we prove that
$$
\lim_{j \rightarrow \infty} F_{3j}(0) =- F_2(0), 
$$ 
which implies
\begin{equation}\label{eq:cor31}
	\lim_{j \rightarrow \infty} {\mathcal R }(v_j)(0) ={\mathcal R }(qw)(0)
\end{equation}
through recalling that $F_2$ and $F_{3j}$ are given in \eqref{eq:F1xi}.  
From \eqref{3eq:bound1}, we have
{\bl 
\begin{equation*}%\label{eq:355bound}
\begin{split}
	& \lim_{ \rho \rightarrow +0 }\frac{1}{m(B(0, \rho  ) \cap P(\Omega ))} \int_{B(0, \rho ) \cap P(\Omega )}  {\mathcal R }(v)(x')  {\rm d} x '  \\
	 &= \lim_{ \rho \rightarrow +0 }\frac{1}{m(B(0, \rho  ) \cap P(\Omega ))} \int_{B(0, \rho ) \cap P(\Omega )} {\mathcal R } ( w)(x')  {\rm d} x'. 
	 \end{split} 
\end{equation*} } 
Since {\bl
\begin{align*}
	\lim_{j \rightarrow \infty} {\mathcal R }( v_j)(0)&=\lim_{j \rightarrow \infty}  \lim_{ \rho \rightarrow +0 }\frac{1}{m(B(0, \rho  )  \cap P(\Omega ) )} \int_{B(0, \rho ) \cap P(\Omega )} {\mathcal R } (v_j)(x')  {\rm d} x'\\
	&= \lim_{ \rho \rightarrow +0 }\frac{1}{m(B(0, \rho  ) \cap P(\Omega ))} \int_{B(0, \rho ) \cap P(\Omega )}{\mathcal R } (v )(x')  {\rm d} x',\\
	{\mathcal R } (qw )(0)&= \lim_{ \rho \rightarrow +0 }\frac{1}{m(B(0, \rho  ) \cap P(\Omega ))} \int_{B(0, \rho ) \cap P(\Omega )} {\mathcal R}(qw)(x')  {\rm d} x',
\end{align*}}
and from \eqref{eq:cor31}, we finish the proof of this corollary. 
\end{proof}

%we derive that
%$$
%0=\lim_{ \rho \rightarrow +0 }\frac{1}{m(B(0, \rho  ))} \int_{B(0, \rho )} {\mathcal R}(V w)(x')  {\rm d} x'=\lim_{ \rho \rightarrow +0 }\frac{1}{m(B(0, \rho  ))} \int_{B(0, \rho ) \times (-L,L)} \psi (x_n ) V (x',x_n) w (x',x_n )  {\rm d} x' {\rm d} x_n 
%$$
%
%
%and
%$$
%f_2(0)= \lim_{ \rho \rightarrow +0 }\frac{1}{m(B(0, \rho  ))} \int_{B(0, \rho )} qw(x)  {\rm d} x, 
%$$
%together  with  
%$$
% \lim_{ \rho \rightarrow +0 }\frac{1}{m(B(0, \rho  ))} \int_{B(0, \rho )} v(x)  {\rm d} x = \lim_{ \rho \rightarrow +0 }\frac{1}{m(B(0, \rho  ))} \int_{B(0, \rho )} w(x)  {\rm d} x, 
%$$
%	

{\color{black}
\begin{remark}
Corollary \ref{cor:3.1} states that the average value of the function $Vw$ over the cylinder centered at the edge point $(0,x_3^c)$ with the height $L$ vanishes in the distribution sense. In addition, if $V(x',x_3)$ is continuous near the edge point $(0,x_3^c)$ where $x_3^c \in (-M,M)$ and $ V(0,x_3^c ) \neq 0$, from the dominant convergence theorem	and the definition of the dimension reduction operator $\mathcal R$, we can prove that
$$
%	\lim_{ \rho \rightarrow +0 }\frac{1}{m(B(x_c, \rho  ))} \int_{B(x_c, \rho )} {\mathcal R}( w )(x')  {\rm d} x'=  %\times [x_n^c-L,x_n^c+L]
{\bl 
	\lim_{ \rho \rightarrow +0 }\frac{1}{m(B(0, \rho  )\cap P(\Omega ) )} \int_{B(0, \rho ) \cap P(\Omega) } \int_{x_3^c-L}^{x_3^c+L} \psi(x_3)  w (x',x_3)  {\rm d} x'{\rm d} x_3 =0 } 
	$$
under the assumptions  in  Corollary \ref{cor:3.1}, which also describes the vanishing property of the interior eigenfunctions $v$ and $w$ near the edge point in 3D. Furthermore, if $ \psi(x_3^c)\neq 0 $, one can prove that
$$
%	\lim_{ \rho \rightarrow +0 }\frac{1}{m(B(x_c, \rho  ))} \int_{B(x_c, \rho )} {\mathcal R}( w )(x')  {\rm d} x'=  %\times [x_n^c-L,x_n^c+L]
{\bl 	\lim_{ \rho \rightarrow +0 }\frac{1}{m(B(0, \rho  ) \cap P(\Omega) )} \int_{B(0, \rho ) \cap P(\Omega) } \int_{x_3^c-L}^{x_3^c+L}  w (x',x_3)  {\rm d} x'{\rm d} x_3 =0. } 
	$$

%
%
%$$
%	\lim_{ \rho \rightarrow +0 }\frac{1}{m(B(x_c, \rho  ))} \int_{B(x_c, \rho )}V(x) v(x)  {\rm d} x= V(x_c)
%	\lim_{ \rho \rightarrow +0 }\frac{1}{m(B(x_c, \rho  ))} \int_{B(x_c, \rho )}v(x)  {\rm d} x,
%	$$
%we can prove that the vanishing property near the corner $x_c$ of the interior transmission eigenfunctions $v   \in H^1(\Omega )$ and $w \in H^1(\Omega )$ under the assumptions  \eqref{eq:ass3} and \eqref{eq:ass1 int}. 

\end{remark}
}

In the following theorem, we impose a stronger regularity requirement for  the conductive transmission eigenfunction $v$ of \eqref{eq:3d in eig}, i.e., $v$ has $H^2$-regularity near the considering edge point. Using the dimension reduction operator given in Definition \ref{Def}, as well as the H{\"o}lder continuity of the considering functions, we can prove the following theorem in a similar way of proving Theorem \ref{Th:1.2}. The detailed proof of Theorem \ref{Th:3.2} is omitted here. 

%As considered in Theorem \ref{Th:3.1}, the domain of the conductive transmission eigenvalue problem \eqref{eq:3d in eig} in three dimension  is $W \times (-M,M)$, where $W$ is a sector in $\R^2$.  In contrast with Theorem \ref{Th:3.1}, the geometry of the sector $W$ here is more flexible, i.e., the open angle of the sector $W$ is not a straight angle, which coincides  with  the assumption regarding the angle requirements of the sector $W$ in Theorems \ref{Th:1.1} and \ref{Th:1.2} for two dimensional conductive eigenfunctions.   The detailed proof of the following theorem is omitted. 

%However,   in Theorem \ref{Th:3.1} we assume that $W$ is only convex. 

\begin{theorem}\label{Th:3.2}
	{\bl Let $\Omega \Subset \mathbb R^3$ be bounded domain with $0\in \partial \Omega$ and} ${S_h}\subset \R^{2}$ be defined in \eqref{eq:sh}, $M>0$, $0<\alpha<1$. For any fixed {\bl $x_3^c \in (-M,M)$} and $L>0$ defined in Definition \ref{Def}, we suppose that  $L $ is sufficiently small such that $(x_3^c-L,x_3^c+L) \subset (-M,M) $ {\bl and
$$
(B_h \times (-M,M) ) \cap \Omega=S_h \times (-M,M),
$$
where $B_h\Subset \mathbb R^2 $ is the central ball of radius $h \in \mathbb R_+$.} Let $v \in H^2 (\Omega )$ and $w \in H^1 (\Omega ) $ be the eigenfunctions to \eqref{eq:3d in eig}.  Moreover,  there exits a sufficiently smaller neighbourhood {\bl $S_h \times (-M,M)$ (i.e. $h>0$ is sufficiently  small) of $(0,x_3^c)$ with $x_3^c\in (-M,M)$}, such that  $q w\in C^\alpha(\overline {S}_h   \times [-M,M ] ) $ and $\eta \in C^\alpha(\overline{\Gamma}_h^\pm  \times [-M,M]  )$  for $0< \alpha  <1$ and $v-w\in H^2({{S}_h}\times(-M,M))$. Under the following assumptions: 
	\begin{itemize}
%		\item[(a)] the transmission eigenfunctions $v$ and $w$   satisfy $v \in H^2({S_h}\times(-L,L)  )$ and $w \in H^1({S_h}\times(-L,L)  )$ , 
	\item[(a)]  the function $\eta=\eta(x',x_3)$ is independent  of $x_3$ and does not vanish on the edge of  $ W \times (-M,M)$, i.e.,    
	\begin{equation*}%\label{eq:ass21}
		\eta(0 ) \neq 0, 
	\end{equation*}
 \item[(b)]  the angles $\theta_m$ and $\theta_M$ of $S_h$ satisfy
	\begin{equation*}%\label{eq:ass31}
-\pi < \theta_m < \theta_M < \pi \mbox{ and } \theta_M-\theta_m \neq \pi,	%\cos \theta_m+\cos \theta_M \mbox{ and } \sin \theta_m+\sin \theta_M
\end{equation*}
%can not be zero simultaneously, 
	\end{itemize}
	 	then we have $v$ and $w$ vanish at the edge point $(0,x_3^c) \in \R^3$ of $S_h \times (-M,M)$, where $x_3^c \in (-M,M)$. 
%	$
%	v(x_c) =0. 
%	$
%	where $B(x_c,\rho )$ is a ball centered at $x_c$ with radius $\rho$ and $m(B(x_c, \rho  ))$ is the volume of $B(x_c,\rho )$. 
\end{theorem}

\begin{remark}
	When $\eta \equiv 0$ near the edge point, under the $H^2$ regularity of the interior transmission eigenfunctions $v$ and $w$, the vanishing property of $v$ and $w$ is investigated in \cite{Bsource}. 
\end{remark}

\section{Unique recovery results for the inverse scattering problem}\label{sec:4}
{\bl 
In this section, we  apply the vanishing property of the conductive transmission eigenfunctions at a corner in 2D to investigate the unique recovery in the 
inverse problem associated with the corresponding  conductive scattering problem. Before that, we first describe the relevant physical background. 

The time-harmonic electromagnetic wave scattering from a conductive medium body arises in  the application of practical importance, for example the modeling of an electromagnetic object coated with a thin layer of a highly conducting material.  

In what follows, we let $\varepsilon$, $\mu$ and $\sigma$ denote the electric permittivity, the magnetic permeability and  the conductivity of a medium, respectively.
%the optical properties of a medium are specified the electric permittivity $\varepsilon$, the magnetic permeability $\mu$ and the conductivity $\sigma$. 
Let $\Omega$ be a bounded Lipschitz domain in $\mathbb{R}^2$ with a connected complement $\mathbb{R}^2\backslash\overline{\Omega}$. Consider a cylinder-like medium body $D:=\Omega\times\mathbb{R}$ in $\mathbb{R}^3$ with the cross section being $\Omega$ along the $x_3$-axis for ${x}=(x_j)_{j=1}^3\in D$. In the following discussions, with a bit abuse of notation, we shall also use ${x}=(x_1, x_2)$ in the 2D case, which should be clear from the context. Let $\Omega_\delta:=\{ {x}+h\nu({x}); {x}\in\partial\Omega\ \mbox{and}\ h\in (0, \delta)\}$ with $\delta\in\mathbb{R}_+$ being sufficiently small, where $\nu\in\mathbb{S}^1$ signifies the exterior unit normal vector to $\partial\Omega$. Set $D_\delta=\Omega_\delta\times\mathbb{R}$ to denote a layer of thickness $\delta$ coated on the medium body $D$. The material configuration associated with the above medium structure is given as follows:
\begin{equation}\label{eq:m1}
\varepsilon, \mu, \sigma=\varepsilon_1, \mu_0, \sigma_1\ \mbox{in}\ \ D;\ \varepsilon_2, \mu_0, \frac{\gamma}{\delta}\ \ \mbox{in}\ D_\delta;\ \varepsilon_0, \mu_0, 0\ \ \mbox{in}\ \mathbb{R}^3\backslash\overline{(D\cup D_\delta)},
\end{equation}
where $\varepsilon_j>0$, $j=0,1,2$, $\mu_0, \gamma>0$  and $\sigma_1\geq0$ are all constants. Consider a time-harmonic incidence: 
\begin{equation}\label{eq:inc1}
\nabla\times {E}^i-i\omega\mu_0 {H}^i=0,\quad \nabla\times {H}^i+i\omega\varepsilon_0 {E}^i=0\quad\mbox{in}\ \mathbb{R}^3,
\end{equation}
where $i:=\sqrt{-1}$, ${E}^i$ and ${H}^i$ are respectively the electric and magnetic fields and $\omega\in\mathbb{R}_+$ is the angular frequency. 
%The impingement of the incident field $({E}^i, {H}^i)$ on the medium body described in \eqref{eq:m1} generates the 
The electromagnetic scattering is generated by the impingement of the incident field $({E}^i, {H}^i)$ on the medium body described in \eqref{eq:m1} as follows
%, which is governed by the following Maxwell system:
\begin{equation}\label{eq:sca1}
\begin{cases}
& \nabla\times {E}-i\omega\mu {H}=0,\quad \nabla\times {H}+i\omega\varepsilon {E}=\sigma {E}, \quad\mbox{in}\ \mathbb{R}^3, \\[5pt]
& {E}={E}^i+{E}^s,\quad {H}= {H}^i+ {H}^s,\hspace*{2.55cm}\mbox{in}\ \mathbb{R}^3, \\[5pt]
& \displaystyle{\lim_{r\rightarrow\infty} r\left({H}^s\wedge\hat{{x}}-{E}^s \right)=0, }\hspace*{3.5cm} r:=|{x}|, \hat{{x}}:={x}/|{x}|,
\end{cases} 
\end{equation}
where 
%as usual one needs to impose the standard transmission conditions, namely 
the tangential components of the electric field ${E}$ and the magnetic field ${H}$ are continuous across the material interfaces $\partial D$ and $\partial D_\delta$. The last limit in \eqref{eq:sca1} is known as the Silver-M\"uller radiation condition. 

Under the transverse-magnetic (TM) polarisation, namely,
%\begin{equation}\label{eq:p1}
%\mathbf{E}^i=\begin{bmatrix}
%0\\0\\ u^i(x_1,x_2)
%\end{bmatrix},\  \mathbf{H}^i=\begin{bmatrix}
%H_1(x_1,x_2)\\
%H_2(x_1,x_2)\\
%0
%\end{bmatrix},
%\
%\mathbf{E}=\begin{bmatrix}
%0\\0\\ u(x_1,x_2)
%\end{bmatrix},\  \mathbf{H}=\begin{bmatrix}
%H_1(x_1,x_2)\\
%H_2(x_1,x_2)\\
%0
%\end{bmatrix},
%\end{equation}
%\begin{subequations}
\begin{align}\notag
{E}^i&=\begin{bmatrix}
0\\0\\ u^i(x_1,x_2)
\end{bmatrix},\  {H}^i=\begin{bmatrix}
H_1(x_1,x_2)\\
H_2(x_1,x_2)\\
0
\end{bmatrix}, %\label{eq:p1} %\\
\end{align}
and
\begin{align}
{E}&=\begin{bmatrix}
0\\0\\ u(x_1,x_2)
\end{bmatrix},\  {H}=\begin{bmatrix}
H_1(x_1,x_2)\\
H_2(x_1,x_2)\\
0
\end{bmatrix}, \label{eq:p2}
\end{align}
%\end{subequations}
it is rigorously verified in \cite{BonT} that as $\delta\rightarrow +0$, one has
	\begin{equation}\label{eq:model}
	\begin{cases}
	\Delta u^-+k^2 q u^-=0 & \mbox{ in }\ \Omega, \\[5pt] 
	\Delta u^+ +k^2  u^+=0 & \mbox{ in }\ \R^2 \backslash \Omega, \\[5pt] 
	u^+= u^-,\quad \partial_\nu u^+ + \eta  u^+=\partial_\nu u^- & \mbox{ on }\ \partial \Omega, \\[5pt]
	u^+=u^i+u^s & \mbox{ in }\ \R^2 \backslash \Omega, \\[5pt] 
	\lim\limits_{r \rightarrow \infty} r^{1/2} \left( \partial_r u^s-i k u^s \right)=0, & \ r=|\mathbf x|,
	\end{cases}
	\end{equation}
where 
\begin{equation}\label{eq:model12}
\mbox{$u^-=u|_{\Omega}$, $u^+=u|_{\mathbb{R}^2\backslash\overline{\Omega}}$\ \ and\ \ $k=\omega\sqrt{\varepsilon_0\mu_0}$, $\eta=i\omega\gamma\mu_0$}.
\end{equation}
The last limit in \eqref{eq:model} is known as the Sommerfeld radiation condition. According to \eqref{eq:m1}, as $\delta\rightarrow+0$, it is clear that the conductivity in the thin layer $D_\delta$ goes to infinity, or equivalently, its resistivity goes to zero. This in general would lead to the so-called perfectly electric conducting (PEC) boundary, which prevents the electric field from penetrating inside the medium body and instead generates a certain boundary current. However, it is noted in our case that the thickness of the coating layer $D_\delta$ also goes to zero, and this allows the electromagnetic waves to penetrate inside the medium body. Nevertheless, the thin highly-conducting layer effectively produces a transmission  boundary condition across the material interface $\partial D$ involving a conductive parameter $\eta$, which is referred to as the conductive transmission condition. It is known to us that a perfect conductor does not exist in nature, and hence the conductive medium body provides a more realistic means to model the electromagnetic scattering from an object coated with a thin layer of a highly conducting material; see \cite{HM, Senior} more relevant discussion about this aspect. }

%\begin{equation}\label{eq:model}
%\begin{cases}
%\Delta u^-+k^2 q u^-=0 & \mbox{ in }\ \Omega, \\[5pt] 
%\Delta u^+ +k^2  u^+=0 & \mbox{ in }\ \R^2 \backslash \Omega, \\[5pt] 
%u^+= u^-,\quad \partial_\nu u^+ + \eta u^+=\partial_\nu u^- & \mbox{ on }\ \partial \Omega, \\[5pt]
%u^+=u^i+u^s & \mbox{ in }\ \R^2 \backslash \Omega, \\[5pt] 
%\lim\limits_{r \rightarrow \infty} r^{1/2} \left( \partial_r u^s-i k u^s \right)=0, & \ r=|x|,
%\end{cases}
%\end{equation}
%where $u^i$ is an (nontrivial) entire solution to $(\Delta+k^2) u^i=0$ signifying an incident field, and the last limit is called the Sommerfeld radiation condition which holds uniformly with respect to {\bl  $ \hat x=x/|x| \in {\mathbb S}^{1}$}, and characterizes the out-radiating wave.  

 The well-posedness of the direct problem \eqref{eq:model} is known (cf. \cite{Bon}), and there exists a unique solution {\bl $u:=u^-\chi_\Omega+u^+\chi_{\mathbb{R}^2 \backslash\Omega}\in H_{loc}^1(\mathbb{R}^2)$.}  Moreover, there holds the following asymptotic expansion
{\bl  \begin{equation}\notag
 u^s(x)=\frac{e^{ik|x|}}{|x|^{1/2}}u^{\infty}(\hat{x})+\Oh \left(\frac{1}{|x|}\right), |x|\rightarrow +\infty
 \end{equation}
 uniformly in all directions $\hat{x}=x/|x| \in {\mathbb S}^{1}$.} The real-analytic function $u^{\infty}(\hat{x})$ is referred to as the \emph{far-field pattern} or the \emph{scattering amplitude} associated with $u^i$. The inverse scattering problem is concerned with the recovery of the scatterer $(\Omega; q, \eta)$ by knowledge of the far-field pattern $u^{\infty}(\hat{x}; u^i)$; that is
 \begin{equation}\label{eq:ip1}
 u^{\infty}(\hat{x}; u^i)\rightarrow(\Omega; q, \eta).
 \end{equation}
 In \eqref{eq:ip1}, if the far-field pattern is given corresponding to a single incident wave $u^i$, then it is referred to as a single far-field measurement, otherwise it is referred to as many far-field measurements. It is known that the inverse problem \eqref{eq:ip1} is nonlinear and ill-conditioned. 
%The problem  \eqref{eq:model}--\eqref{eq:model2}  is derived as the TM-mode (transverse magnetic) from the time-harmonic Maxwell system \cite{Ang}, where the scattering medium is covered by a thin layer with very high conductivity. The inverse scattering problem associated with \eqref{eq:model} is nonlinear and ill-posed, which makes it more difficult and challenging to deal with. 
For the reconstruction of the shape of the scatterer  $\Omega$ by using  the factorization method for \eqref{eq:ip1}, uniqueness issue has been studied in \cite{Bon}. The inverse spectral problem of gaining the information about the material properties associated to the conductive transmission eigenvalue problem has been studied in \cite{BHK}. In \cite{HK}, the method of uniquely recovering the conductive boundary parameter $\eta$ from the measured scattering data as well as the convergence of the conductive transmission eigenvalues as the conductivity parameters which tend to zero has also been studied. In all of the aforementioned literatures, the unique determination results are based on the far-field patterns of all incident plane waves at a fixed frequency, which means that infinitely many far-field measurements have been used. In what follows, we show that in a rather general and practical scenario, the polyhedral shape of the scatterer, namely $\Omega$, can be uniquely recovered by a single far-field measurement without knowing its material contents, namely $q$ and $\eta$. Moreover, if the surface conductive parameter $\eta$ is constant, then it can be recovered as well. 

 {{Our main unique recovery results for the inverse scattering problem \eqref{eq:ip1} are contained in Theorems \ref{Th:4.1} and \ref{Th: unique eta}. In Theorem \ref{Th:4.1}, we establish the unique recovery results by a single far-field measurement in determining a 2D polygonal conductive scatterer without knowing its contents. In Theorem \ref{Th: unique eta}, the surface conductive parameter $\eta$ of the scatterer can be further recovered if it is a constant.}} Before presenting the main results, we first show in Proposition \ref{Pro} that the conductive parameter $\eta$ in \eqref{eq:model} has a close relationship with the wave number $k$ from the practical point view of the TM-mode (transverse magnetic) for the time-harmonic Maxwell system \cite{Ang}.  This relationship helps us to show that our assumption in Theorem \ref{Th:4.1} can be fulfilled when the wave number $k$ is sufficiently small.

{\bl 
In view of \eqref{eq:model12}, we readily have the following observation. 
\begin{remark}\label{Pro}
	The conductive boundary parameter $\eta$ of \eqref{eq:model} satisfies 
\begin{equation}\label{eq:condn1}
	\eta= i\omega\gamma\mu_0 =\Oh( k ),
\end{equation}
where $k:=\omega \sqrt{\varepsilon_0 \mu_0  }$ is the wave number in \eqref{eq:model}. 
\end{remark} } 

Remark~\ref{Pro} basically indicates that when considering the conductive scattering problem \eqref{eq:model}, one may impose the
low-frequency dependence behaviour \eqref{eq:condn1} on the surface conductive parameter. As remarked earlier, Remark~\ref{Pro} only considers the simple model
\eqref{eq:sca1} for illustration of the low-frequency behaviour \eqref{eq:condn1}. For more complex Maxwell models, one can derive the conductive scattering system \eqref{eq:model} of a general form. 

We are in a position to consider the inverse problem \eqref{eq:ip1}. First, we introduce the admissible class of conductive scatterers in our study. {\bl  Let $W_{x_c}(\theta_W)$ be an open sector in $\mathbb R^2$ with the vertex $x_c$ and the open angle $\theta_W $. Denote
	\begin{equation}\label{eq:thm21}
	\begin{split}
		\Gamma_h^\pm(x_c)&: =\partial W_{x_c} (\theta_W )  \cap B_h(x_c),\quad S_{h} (x_c):=  \Omega\cap B_{h} (x_c)= \Omega \cap W_{x_c} (\theta_W ), \\
		S_{h/2} (x_c)&:=  \Omega\cap B_{h/2} (x_c)= \Omega \cap W_{x_c} (\theta_W ),\quad\Sigma_{\Lambda_h}(x_c):=S_{h}(x_c)\backslash S_{h/2} (x_c). 
	\end{split}
	\end{equation}} 

\begin{definition}~\label{def:adm}
Let {\bl $(\Omega;k,d, q, \eta)$}  be a conductive scatterer associated with {\bl the incident plane wave $u^i=e^{i kx\cdot d}$ with $d\in\mathbb{S}^{1}$ and $k\in \mathbb R_+$. Consider} the scattering problem \eqref{eq:model} and $u$ is  the total wave field therein. 
The scatterer is said to be admissible if it fulfils the following conditions:
\begin{itemize}
\item[(a)]  $\Omega$ is a bounded {\bl simply connected}  Lipschitz domain in $\mathbb{R}^2$, and $q\in L^\infty(\Omega)$, $\eta\in L^\infty(\partial\Omega)$. 

\item[(b)] Following the notations in Theorem \ref{Th:1.1}, if $\Omega$ possesses a corner {\bl $B_h(x_c) \cap \Omega= \Omega\cap W_{x_c}(\theta_W )$ where $x_c$ is the vertex of the sector $W_{x_c}(\theta_W )$ and the open angle $\theta_W$ of $  W_{x_c}(\theta_W )  $ satisfies $\theta_W \neq \pi $,  then $q \big |_{S_h(x_c)} $ is a constant,  $\eta\in C^\alpha(\overline{\Gamma_h^\pm(x_c}))$, where $S_{h}(x_c) $ and $\Gamma_h^\pm (x_c)$ are defined in \eqref{eq:thm21}.}    

%\item[(c)] In $\mathbb{R}^3$, following the notations in Theorem~\ref{Th:3.1}, if $\Omega$ possesses an edge $W\times (-M, M)$, then with with $v=u^i, w=u$, all the conditions stated in either Theorem~\ref{Th:3.1} or Theorem \ref{Th:3.2} are satisfied. 

\item[(c)] The total wave field $u$ is non-vanishing everywhere in the sense that for any $x\in\mathbb{R}^2$,
\begin{equation}\label{eq:nn2}
	\lim_{ \rho \rightarrow +0 }\frac{1}{m(B(x, \rho  ))} \int_{B(x, \rho )} |u(x)|  {\rm d} x\neq 0. 
	\end{equation}
\end{itemize}
\end{definition}

We would like to point out that the conditions stated in Definition~\ref{def:adm} can be fulfilled by the conductive scatterer $(\Omega; k,d, q, \eta)$ and the scattering problem \eqref{eq:model} in certain general and practical scenarios. 
%For example, as remarked in Remark~\ref{rem:hn1}, if $q=0$ in $\overline{S}_h$ then the conditions in (b)  can be easily fulfilled. If $q\neq 0$ in $\overline{S}_h$, but $\eta=0$ on $\overline{S}_h\cap\partial\Omega$, then $u\in H^2(S_h)$. Hence, the conditions in (b) can also be easily fulfilled. There might be more cases for which the conditions in (b) are fulfilled. 
In particular, the condition \eqref{eq:nn2} in (c) can be fulfilled at least when $k$ is sufficiently small. In fact, it has been shown in Proposition \ref{Pro} that if $\eta\neq 0$, then $\eta = \Oh (k)$ .  For the scattered field $u^s$ of \eqref{eq:model}, from \cite[Theorem 2.4]{Bon}, it is proved that
		$$
	\|u^s\|_{H^1(B)} \leq C( \|\eta u^i\|_{H^{-1/2}   ( \partial \Omega  ) }+k^2 \|q u^i \|_{L^2(\Omega  )} ) =\Oh( k  ) \|u^i\|_{L^2(\Omega)},
	$$
	where $C$ is a positive number and $B$ is a large ball containing $\Omega$. Hence, if the incident field $u^i$ is non-vanishing everywhere, say $u^i=e^{i kx\cdot d}$ with $d\in\mathbb{S}^{1}$ being a plane wave, and $k$ is sufficiently small, then \eqref{eq:nn2} is obviously fulfilled. Nevertheless, by Definition~\ref{def:adm}, we may include more general situations into our subsequent study of the inverse problem \eqref{eq:ip1}.

	\begin{figure}
  \centering
  \includegraphics[width=0.35\textwidth]{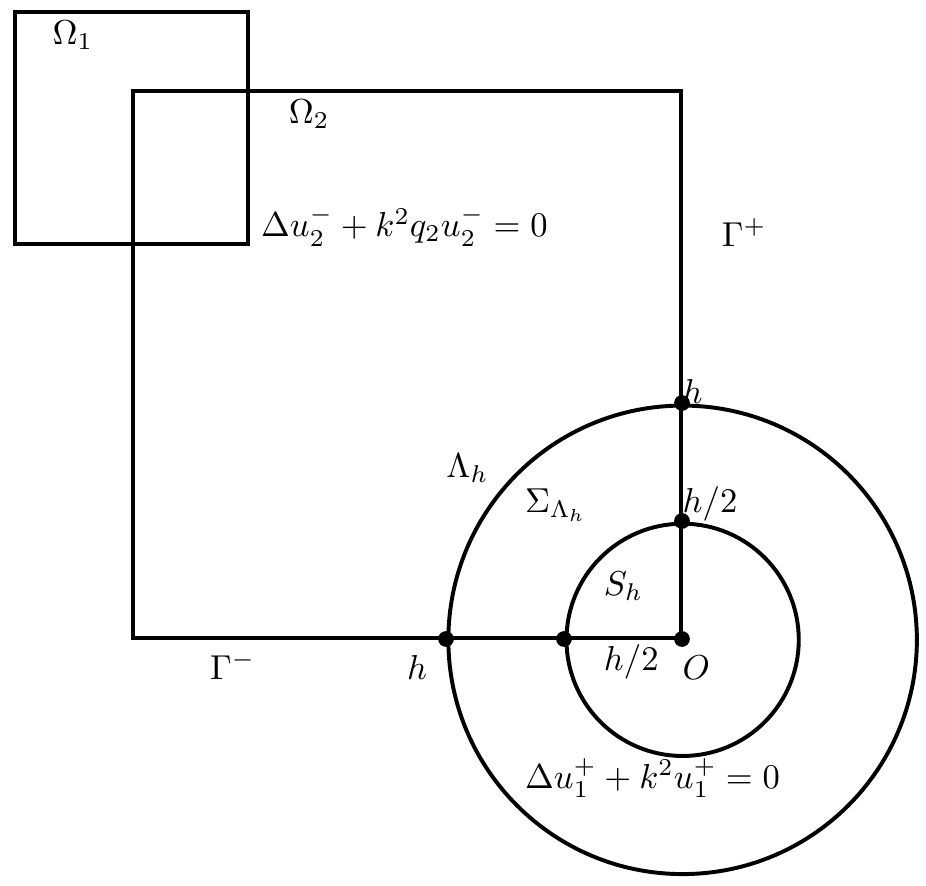}
  \caption{Schematic illustration of the geometry setup in the proof of Theorem~\ref{Th:4.1}.}
  \label{fig:geometry}
\end{figure}

{\bl 
The determination of the geometric shape of a conductive scatterer can be established in Theorem \ref{Th:4.1} by using Theorem~\ref{Th:1.2} and a contradiction argument. The technical requirement $qw\in C^\alpha(\overline S_h)$ in Theorem \ref{Th:1.2} can be easily fulfilled. Indeed,  this condition can be derived by using the classical results on the singular behaviours of the solutions to elliptic PDEs in a corner domain \cite{Dauge88,Grisvard,Cos,CN}. In fact, it is known that the solution can be decomposed into a singular part and a regular part, where the singular part is of a H\"older form that depends on the corner geometry as well as the boundary and the right-hand inputs. For our subsequent use, we first give the following result in a relatively simple scenario. 

In \eqref{eq:model}, by the standard PDE theory (see e.g. \cite{McLean}), we know that the solution $u$ is real-analytic away from the conductive interface. In  Lemma~\ref{lem41} as follows, we further establish the H\"older-regularity of the solution up to the conductive interface, especially to the vertex corner point. Denote
$$
	S_{2h}= W\cap B_{2h},\quad \Gamma_{2h}^{\pm }= \Gamma^{\pm } \cap B_{2h},
$$
where $W$ is the sector defined in \eqref{eq:W}, $B_{2h}$ is an open ball centered at $0$ in $\mathbb R^2$ with the radius $2h$ and $\Gamma^\pm$ are the boundaries of $W$. 
\begin{lemma}\cite[Lemma 3.4]{CDL3} \label{lem41}
Suppose that $u\in H^1(B_{2h})$ satisfies 
	\begin{equation}\label{eq:lem23}
\begin{cases}
\Delta u^-+k^2 q_- u^-=0 & \mbox{ in }\ S_{2h}, \\[5pt] 
\Delta u^+ +k^2  q_{+} u^+=0 & \mbox{ in }\ B_{2h} \backslash {\overline{ S_{2h}}}, \\[5pt] 
u^+= u^- & \mbox{ on }\ \Gamma_{2h}^\pm, 
\end{cases}
\end{equation}
where $u^+=u|_{{B_{2h}\backslash {\overline{ S_{2h}}}}}$, $u^-=u|_{S_{2h}}$ and $q_{\pm}$, $k$ are complex constants. Assume that $u^+$ and $u^-$ are respectively real analytic in ${B_{2h}\backslash {\overline{ S_{2h}}}}$ and $S_{2h}$. There exists $\alpha\in (0, 1)$ such that $u^- \in C^\alpha(\overline{S_h})$, where $S_h$ is defined in \eqref{eq:sh}. 
\end{lemma}
}

\begin{theorem}\label{Th:4.1}
Consider the conductive scattering problem \eqref{eq:model} associated with two conductive scatterers $(\Omega_j; k, d, q_j, \eta_j)$, $j=1,2$, in $\mathbb{R}^2$. Let $u_\infty^j(\hat x; u^i)$ be the far-field pattern associated with the scatterer $(\Omega_j; k, d, q_j, \eta_j)$ and the incident field $u^i$. Suppose that $(\Omega_j; k, d, q_j, \eta_j)$, $j=1,2$ are admissible and
\begin{equation}\label{eq:nn1}
u_\infty^1(\hat x; u^i)=u_\infty^2(\hat x; u^i)
\end{equation}
for all $\hat{x}\in\mathbb{S}^{1}$ and a fixed incident wave $u^i$. Then 
\begin{equation}\label{eq:nn3}
\Omega_1\Delta\Omega_2:=\big(\Omega_1\backslash\Omega_2\big)\cup \big(\Omega_2\backslash\Omega_1\big)
\end{equation}
cannot possess a corner. Hence, if $\Omega_1$ and $\Omega_2$ are convex polygons in $\mathbb{R}^2$, one must have
\begin{equation}\label{eq:nn4}
\Omega_1=\Omega_2. 
\end{equation}
\end{theorem}

% We suppose further that this corner point coincides with the origin and  pick a fixed number $h >0$ such that $B(0,h) \subset \Omega_1^e$. Since $\Omega_2$ is a convex polygon, rotating the coordinate axes if necessary, we have $S_h \subset \Omega_2$, where $S_h$ is defined in \eqref{eq:sh}. 

\begin{proof}
% We present the proof only for the 2D case and the proof of the 3D case follows from a completely similar argument. 
 By contradiction, we assume that there is a corner contained in $\Omega_1\Delta\Omega_2$. Without loss of generality we may assume that the vertex $O$ of the corner $\Omega_2 \cap W$ is such that $O \in \partial \Omega_2$ and $ O \notin \overline{\Omega}_1$. {\bl Without loss of generality, we may assume that $O$ is the origin of $\mathbb R^2$.}

	Since $u_\infty^1(\hat x; u^i)=u_\infty^2(\hat x; u^i)  $ for all $\hat x \in  {\mathbb S}^1$, applying Rellich's Theorem (see \cite{CK}),  we know that $u_1^s=u_2^s$ in $\R^2 \backslash (  \overline{\Omega}_1 \cup \overline{\Omega}_2 )$. Thus
	\begin{equation}\label{eq:u1u2}
		u_1(x)=u_2(x)
	\end{equation}
	for all $x \in \R^2 \backslash (  \overline{\Omega}_1 \cup \overline{\Omega}_2 )$. Following the notations in \eqref{eq:sh}, we have from \eqref{eq:u1u2} that
	$$
	u_2^-=u_2^+=u_1^+,\quad \partial u_2^- = \partial u_2^+ + \eta_2 u_2^+=\partial u_1^+ + \eta_2 u_1^+ \mbox{ on } \Gamma_h^\pm,
	$$
	where the superscripts $(\cdot)^-, (\cdot)^+$ stand for the limits taken from $\Omega_2$ and $\mathbb{R}^2\backslash\overline{\Omega_2}$, respectively. Moreover, suppose the neighbourhood {\bl $B_{2h}$} is  sufficiently  small such that
	$$
	\Delta u_1^+ + k^2 u_1^+ =0,\quad \Delta u_2^-  +k^2 q_2 u_2^-=0 \mbox{ in } {\bl B_{2h}}. 
	$$

{\bl 	It is clear that  $u_1^+$ and $u_2^-$ are respectively real analytic in ${B_{2h}\backslash {\overline{ S_{2h}}}}$ and $S_{2h}$. Since $q_2 \big |_{S_h} $ is a constant, by virtue of Lemma \ref{lem41}, we know that  $u_2^- \in C^\alpha(\overline{S_h})$, which implies that 
\begin{equation}\label{eq:416}
	q_2u_2^- \in C^\alpha(\overline{S_h}).
\end{equation}  }
	
Clearly $u_1^+\in H^2(S_h)$ and $u_2^-\in H^1(S_h)$. Now we prove that
$$
u_1^+ - u_2^-\in H^2( \Sigma_{\Lambda_h} ),
$$
where $\Sigma_{\Lambda_h}$ is defined in \eqref{eq:sh}. We first note that on the boundary $\Gamma^\pm_h$, one has $u_2^-=u_1^+$, where $u_1^+ \in H^{3/2}(\Gamma^\pm_h )$ from the trace theorem. {\bl Denote
\begin{equation}\label{add7}
	D_1^+ =B_{h/4}(x_{\Gamma^+}) \cap \Sigma_{\Lambda_h},\quad D_1^- =B_{h/4}(x_{\Gamma^-}) \cap \Sigma_{\Lambda_h},
\end{equation}
where $x_{\Gamma^+}$ and $x_{\Gamma^-}$ are the mid-points of $\Gamma^{+}_{(h/2,h)}  $ and $\Gamma^{-}_{(h/2,h)}  $, respectively. 
}

Since $\Gamma_h^+ \in C^{1,1}$, from \cite[Theorem 4.18]{McLean}, we have the following regularity estimate for $u_2^-$ up to the boundary $\Gamma^{+}_{(h/2,h)}  $ of $\Sigma_{\Lambda_h}$:
$$
{\bl 
\left\|u_2^- \right \|_{H^2( D_1^+ )} \leq C\left(\left\|u_2^- \right \|_{H^1(S_h )}+ \left\|u_1^+ \right \|_{H^{3/2}( \Gamma^{+}_{h} )}  \right),} 
$$
where $C>0$ is a constant and $D_1^+ $ is defined in \eqref{add7}.
%is a open set with a boundary $\Gamma^{+}_{(h/2,h)} $ such that $\Sigma_{\Lambda_h} \cap D_1^+ \neq \emptyset$.  
Using the similar argument, we can prove that $u_2^-$ has $H^2$-regularity up to the boundary $\Gamma^{-}_{(h/2,h)}  $ of $\Sigma_{\Lambda_h}$. Therefore $u_2^- \in H^2( \Sigma_{\Lambda_h} )$ {\bl by using the interior regularity of the standard elliptic PDE theory}, which means that $u_1^+ - u_2^- \in H^2( \Sigma_{\Lambda_h} )$. Since  $(\Omega_j; k,d, q_j, \eta_j)$, $j=1,2$, are admissible, we know that $\eta_j\in C^\alpha(\overline{\Gamma}_h^\pm)$.  {\bl Noting \eqref{eq:416}} and by applying Theorem  \ref{Th:1.2}, if $\eta_2(0)\neq 0$, and Remark~\ref{rem:2.5} if $\eta_2(0)=0$ on $\Gamma_{ h}^\pm$, and also utilizing the fact that $u_1$ is continuous at the vertex $0$, we have 
	$$
	u_1(0)=0, 
	$$
	which  contradicts to the admissibility condition (c) in Definition \ref{def:adm}.    
	
	The proof is complete. 
\end{proof}

{{Based on Definition \ref{def:adm}, if we further assume that the surface conductive parameter $\eta$ is constant, we can recover $\eta$ simultaneously once the admissible conductive  scatterer $\Omega$ is determined. However, in determining the surface conductive parameter, we need to assume that $q_1=q_2:=q$ are known. 
		
		\begin{theorem}\label{Th: unique eta}
			Consider the conductive scattering problem \eqref{eq:model} associated with the  admissible  conductive scatterers $(\Omega_j; k, d, q, \eta_j)$, where $\Omega_j=\Omega$ for $j=1,2$ and $\eta_j\neq0$, $j=1,2$, are two constants. Let $u_{\infty}^j(\hat{x}; u^i)$ be the far-field pattern associated with the scatterer $(\Omega; k, d, q, \eta_j)$ and the incident field $u^i$. Suppose that $(\Omega; k, d, q, \eta_j)$, $j=1,2$, are admissible and 
			\begin{equation}\label{eq:far}
			u_{\infty}^1(\hat{x}; u^i)=u_{\infty}^2(\hat{x}; u^i)
			\end{equation} 
			for all $\hat{x}\in\mathbb{S}^1$ and a fixed incident wave $u^i$. Then if $k$ is not an eigenvalue of the partial differential operator $\Delta+k^2q$ in $H_0^1(\Omega)$, we have $\eta_1=\eta_2$. 
		\end{theorem}
	
%			
%			Let $(\Omega_1, \eta_1)$ and $(\Omega_2, \eta_2)$ be two bounded convex polygons such that $u_1^\infty(\hat{x}) =u_2^\infty({\hat{x}})$, for all $\hat{x}\in\mathbb{S}^1$. Suppose that $\eta_j\neq 0, j=1,2$ are two constant functions. $q$, $S_h$ and the sector $W$ satisfy the same assumptions in Theorem \ref{Th:4.1}. Further assume that 

%			there exist two constant conductive parameters $\eta_j$ with $\eta_j\neq0$, $j=1,2$, such that for $j=1,2$,
%			\begin{equation}\label{eq:eta eq}
%			\begin{cases}
%			\Delta u_j^-+k^2 q u_j^-=0 & \mbox{ in }\ \Omega, \\[5pt] 
%			\Delta u_j^+ +k^2  u_j^+=0 & \mbox{ in }\  \R^n \backslash \Omega, \\[5pt] 
%			u^+_j= u^-_j,\quad \partial_\nu u^+_j + \eta_j u^+_j=\partial_\nu u^-_j & \mbox{ on }\ \partial \Omega, 
%			\end{cases}
%			\end{equation}
%
%		 $k$ is not an eigenvalue of the operator $\Delta+k^2 q$ in $H_0^1(\Omega)$. Then we have $\eta_1=\eta_2$.

%			From the equal far-field patterns, by Theorem \ref{Th:4.1}, we know that the polyhedral scatter is uniquely recovered, that is $\Omega_1=\Omega_2=\Omega$. 
%			For $j=1,2$, the constant functions $\eta_j$ satisfy
%			\begin{equation}\label{eq:eta eq}
%				\begin{cases}
%				\Delta u_j^-+k^2 q u_j^-=0 & \mbox{ in }\ \Omega, \\[5pt] 
%				\Delta u_j^+ +k^2  u_j^+=0 & \mbox{ in }\  \R^n \backslash \Omega, \\[5pt] 
%				u^+_j= u^-_j,\quad \partial_\nu u^+_j + \eta_j u^+_j=\partial_\nu u^-_j & \mbox{ on }\ \partial \Omega.
%				\end{cases}
%			\end{equation}
%	    	where $(\cdot)^-, (\cdot)^+$ stand for the limits taken from $\Omega$ and $\Omega^e$ respectively. Then by 
	    \begin{proof}	
	    	Since $u_{\infty}^1(\hat{x}; u^i) =u_\infty^2({\hat{x}}; u^i)$ for all $\hat{x}\in\mathbb{S}^1$, we can derive that $u_1^+=u_2^+$ for all $x\in\mathbb{R}^2\backslash\overline{\Omega}$ and thus $\partial_\nu u_1^+=\partial_\nu u_2^+$ on $\partial\Omega$. Combining with the transmission condition in the scattering problem \eqref{eq:model}, we deduce that 
			\begin{equation}\notag
			u_1^-=u_1^+=u_2^+=u_2^-\ \mbox{ on }\ \partial\Omega,
			\end{equation}
			Thus, we have
			\begin{equation}\notag
			\partial_\nu(u_1^--u_2^-)=\partial_\nu(u_1^+-u_2^+)+\eta_1u_1^+-\eta_2u_2^+=(\eta_1-\eta_2)u_1^- \ \mbox{ on }\ \partial\Omega.
			\end{equation}
			Define $v:=u_1^--u_2^-$. Then $v$ fulfills
			\begin{equation}\label{eq:v}
			\begin{cases}
			(\Delta+k^2q)v=0 & \mbox{ in }\ \Omega,\\
			v=0 &\mbox{ on }\ \partial\Omega,\\
			\partial_\nu v=(\eta_1-\eta_2)u_1^- &\mbox{ on }\ \partial\Omega.
			\end{cases}	
			\end{equation}
			Since $k$ is not an eigenvalue of the operator $\Delta+k^2q$ in $H_0^1(\Omega)$, 
			hence one must have $v=0$ to \eqref{eq:v}. Substituting this into the Neumann boundary condition of \eqref{eq:v}, we know that $(\eta_1-\eta_2)u_1^-=\partial_\nu v=0$ on $\partial\Omega$.
			
			Next, we prove the uniqueness of $\eta$ by contradiction. Assume that $\eta_1\neq\eta_2$. Since $(\eta_1-\eta_2)u_1^-=0$ on $\partial\Omega$ and $\eta_j$, $j=1,2$ are constants, we can deduce that $u_1^-=0$ on $\partial\Omega$. Then $u_1^-$ satisfies
			\begin{equation}\notag
			\begin{cases}
			(\Delta +k^2q)u_1^-=0 & \mbox{ in }\ \Omega,\\
			u_1^-=0 & \mbox{ on }\ \partial\Omega.
			\end{cases}
			\end{equation}
		    Similar to \eqref{eq:v}, this Dirichlet problem also only has a trivial solution $u_1^-=0$ in $\Omega$, since $k$ is not an eigenvalue of $\Delta +k^2q$. Then, we can derive  $u_1^+=u_1^-=0$ and 
			\begin{equation}\notag
			\partial_\nu u_1^-=\partial_\nu u_1^++\eta_1u_1^+=\partial_\nu u_1^+=0 \mbox{ on } \partial\Omega,
			\end{equation}
			which implies that $u_1\equiv0$ in $\mathbb{R}^2$ and thus $u_1^s=-u^i$. This contradicts to the fact that $u^s_1$ satisfies the Sommerfeld radiation condition. 
			
			The proof is complete. 
		\end{proof}

\begin{remark}
In Theorem \ref{Th: unique eta}, it is required that $k$ is not an eigenvalue of $\Delta+k^2q$ in $H_0^1(\Omega)$. Clearly, if $q$ is negative-valued in $\Omega$ or $\Im q\neq 0$ in $\Omega$, this condition is fulfilled. On the other hand, if $q$ is positive-valued in $\Omega$, then this condition can be readily fulfilled when $k\in\mathbb{R}_+$ is sufficiently small. 
%  Indeed, the condition stated in Theorem \ref{Th: unique eta} that $k$ is not an eigenvalue of the elliptic differential operator $\Delta+k^2q$ in $H_0^1(\Omega)$ can be fulfilled when $k$ is sufficiently small (cf. \cite{Evans}).
\end{remark}
}}

\section*{Acknowledgement}

The work  of H Diao was supported in part by the Fundamental Research Funds for the Central Universities under the grant 2412017FZ007. The work of H Liu was supported by a startup fund from City University of Hong Kong and the Hong Kong RGC grants (projects 12302017, 12301218 and 12302919).

\end{document}